\documentclass[10pt]{article}

\usepackage{amsmath,amsthm,amssymb}
\usepackage{nccmath}
\usepackage{enumitem} 
\usepackage[sorted]{amsrefs}
\usepackage{graphicx} 
\usepackage{float} 
\usepackage{tikz}
\usepackage{xcolor}
\usepackage[colorlinks=true, linkcolor=blue, citecolor=blue]{hyperref}
\usepackage[textwidth=6.0in, textheight=9.0in]{geometry}
\usepackage{microtype}

\makeatletter
\renewcommand\paragraph{\@startsection{paragraph}{4}{\z@}{\z@}{-1em}{\normalfont\normalsize\bfseries}}
\makeatother
\setlist[enumerate]{topsep=0pt,partopsep=1ex,parsep=1ex}
\setlist[itemize]{topsep=0pt,partopsep=1ex,parsep=1ex}
\newtheorem{theorem}{Theorem}[section]
\newtheorem{proposition}[theorem]{Proposition}
\newtheorem{lemma}[theorem]{Lemma}

\newtheorem{fact}[theorem]{Fact}
\newtheorem{corollary}[theorem]{Corollary}
\newtheorem{observation}[theorem]{Observation}
\newtheorem{conjecture}[theorem]{Conjecture}
\newtheorem{definition}[theorem]{Definition}
\newtheorem{example}[theorem]{Example}
\newtheorem{question}[theorem]{Question}

\theoremstyle{remark}
\newtheorem{remark}[theorem]{Remark}

\newcommand{\RR}{\ensuremath{\mathbb R}}

\renewcommand{\epsilon}{\varepsilon}
\newcommand{\link}{\operatorname{lk}}
\newcommand{\cls}[1]{\textsf{#1}}

\newcommand{\bicoloured}{\cls{bicoloured-interval}}
\newcommand{\sandwich}{\cls{interval-sandwich}}
\newcommand{\parallelogram}{\cls{parallelogram}}
\newcommand{\cocomparability}{\cls{co-comparability}}
\newcommand{\comparability}{\cls{comparability}}
\newcommand{\perfect}{\cls{perfect}}
\newcommand{\interval}{\cls{interval}}
\newcommand{\weaklychordal}{\cls{weakly-chordal}}
\newcommand{\chordal}{\cls{chordal}}
\newcommand{\permutation}{\cls{permutation}}
\newcommand{\trapezoid}{\cls{trapezoid}}
\newcommand{\tol}{\cls{tolerance}}
\newcommand{\bounded}{\cls{bounded-tolerance}}
\newcommand{\unit}{\cls{unit-tolerance}}
\newcommand{\proper}{\cls{proper-tolerance}}
\newcommand{\halftolerance}{\cls{50\%-tolerance}}

\newcommand{\NP}{\textsf{NP}}
\newcommand{\PSPACE}{\textsf{PSPACE}}

\newcommand{\ord}{\mathrm{ord}}
\newcommand{\pre}{\mathrm{pre}}
\newcommand{\hi}{\mathrm{hi}}
\newcommand{\lo}{\mathrm{lo}}

\title{Bicoloured-interval and interval-sandwich graphs: two new classes in the tolerance hierarchy}

\author{Abdul Basit\thanks{Centre for Artificial Intelligence and Machine Learning, School of Science, Edith Cowan University, Perth, Australia; \texttt{a.basit@ecu.edu.au}. AB was supported in part by Australian Research Council grant DP250104390.}
\and David Suter\thanks{Centre for Artificial Intelligence and Machine Learning, School of Science, Edith Cowan University, Perth, Australia; \texttt{d.suter@ecu.edu.au}}
\and Erchuan Zhang\thanks{School of Science, Sun Yat-sen University, Shenzhen, P.R.China; \texttt{zhangerch@mail.sysu.edu.cn}}
}

\date{\today}

\begin{document}

\maketitle

\begin{abstract}
We introduce two new graph classes, the bicoloured-interval graphs and the interval-sandwich graphs, arising from the study of the robust line fitting problem in computer vision.
In a bicoloured-interval graph, each vertex is assigned a real interval and one of two colours.
Vertices of different colours are adjacent precisely when their intervals intersect, and vertices of the same colour precisely when the centre of one interval lies in the other.
Discarding the colouring yields the interval-sandwich graphs, sandwiched between the $50\%$-tolerance graph and the intersection graph of a family of intervals.

We prove structural results for both classes, determining which holes, antiholes, trees, and complete bipartite graphs each contains.
Consequently, neither class is characterised by finitely many forbidden induced subgraphs.
We place both classes strictly between the unit tolerance graphs and the co-comparability graphs in the tolerance hierarchy, separating them from the neighbouring classes.
In particular, we construct an infinite family of proper tolerance graphs, each of which is a minimal forbidden induced subgraph for the bicoloured-interval graphs.

Finally, we consider the recognition problems for both classes.
We show that every graph in either class has a polynomial-size integer representation, and hence that both problems lie in~$\NP$.
\end{abstract}

\section{Introduction}
\label{sec:intro}

Robust fitting is a fundamental problem in computer vision and statistics, where the goal is to estimate model parameters from data contaminated by outliers.
The central difficulty is that genuine observations and the model explaining them must be identified together.
Maximum consensus makes this precise by asking for the largest subset of the data that some instance of the model fits within a prescribed tolerance~\cite{chin-suter}.

Computing the maximum consensus exactly is $\NP$-hard when the dimension of the model is part of the input~\cite{chin-cai-neumann}, and practical algorithms trade optimality for efficiency.
One line of work in computer vision uses hypergraphs to model the problem.
Their higher-order structure is then used to separate inliers from outliers~\cites{Wang_2015_ICCV,7582510,8283797,lin2019hypergraph}.

Line fitting under the maximum consensus criterion is a geometric problem.
Given a width and a finite set $P$ of points in the plane with pairwise distinct $x$-coordinates, find the largest subset of $P$ contained in a single slab of that {\em vertical} width.
For a point $p$, the slabs of a given vertical width containing $p$ form a convex set in the parameter plane.
Hence, by Helly's theorem~\cites{helly, tancer2012intersection}, a set of at least three points lies in a common slab if and only if every three of them do.
The feasibility relations among points of $P$ are therefore captured by a $3$-uniform hypergraph $H(P)$ on the points, whose hyperedges are the triples contained in some slab, and the maximum consensus is exactly the maximum clique of $H(P)$.
We refer to these hypergraphs as {\em slab hypergraphs}.
That the width of slabs is measured vertically is essential for this argument.
If the width is measured orthogonally, the corresponding slabs do not form a convex set in the parameter plane, and Helly's theorem cannot be applied.

The {\em link} of a vertex $v$ of a $3$-uniform hypergraph is the graph on the remaining vertices in which $u$ and $w$ are adjacent precisely when $vuw$ is a hyperedge, so the link of a point records which pairs of points can join it in a feasible triple.
Our starting point is an exact description of these links, which are precisely the graphs in a class we denote $\bicoloured$ (Proposition~\ref{prop:B-is-links}).
In a bicoloured representation, each vertex receives an interval and one of two colours, and adjacency is governed by one of two thresholds.
Vertices of different colours are adjacent precisely when their intervals intersect, and vertices of the same colour are adjacent precisely when the centre of one interval lies in the other.
Discarding the colouring and keeping the two thresholds as an upper and a lower constraint yields the larger class of {\em interval-sandwich} graphs, the graphs sandwiched between the $50\%$-tolerance graph and the intersection graph of a family of intervals.

Beyond this characterisation, we undertake a systematic structural study of both classes.
We completely resolve their behaviour on holes, antiholes, trees, and complete bipartite graphs.
In particular, the holes $C_k$ with $k \geq 5$, the antiholes $\overline{C_n}$ with $n \geq 5$, the spider $T_2$, and $K_{3,3}$ are minimal forbidden induced subgraphs for $\bicoloured$, and the same holds for $\sandwich$ except that the even antiholes lie in $\sandwich$.
Consequently, neither class admits a characterisation by finitely many forbidden induced subgraphs.

We then place both classes in the tolerance hierarchy of Golumbic and Trenk~\cite{golumbic2004tolerance}.
Both lie strictly between the unit tolerance graphs and the co-comparability graphs, and Figure~\ref{fig:hierarchy} on page~\pageref{fig:hierarchy} summarises the resulting picture.
Geometrically, $\bicoloured$ is the class of intersection graphs of parallelograms between two horizontal lines, each with a vertical diagonal.
Alongside the inclusions we prove separations, and here the two classes part ways.
The class $\bicoloured$ lies strictly inside the bounded tolerance graphs.
On the other hand, $\sandwich$ is incomparable with both the tolerance graphs and the proper tolerance graphs (\S\ref{sec:sandwich-placement}).

One of the main technical contributions of the paper is a family of graphs $Q^m$, $m \geq 2$, each a proper tolerance graph that is not a bicoloured-interval graph and is minimal with this property, in that every proper induced subgraph of $Q^m$ lies in $\bicoloured$ (Theorem~\ref{thm:proper-incomparable}).
The construction adapts a family of graphs in $\proper \setminus \unit$ due to Bogart, Fishburn, Isaak, and Langley~\cite{bogart-fishburn-isaak-langley}, and since $\unit \subseteq \bicoloured$, the $Q^m$ in turn enlarge the known catalogue of such graphs.
The graphs $Z$ and $J$ of \S\ref{sec:hierarchy} already make $\proper$ and $\bicoloured$ incomparable, and the family shows further that $\proper \cap \bicoloured$ has no characterisation by finitely many forbidden induced subgraphs within $\proper$.

Finally, we consider recognition.
Following the argument of Hayward and Shamir~\cite{hayward-shamir} for tolerance graphs, we show that every graph in either class has a polynomial-size integer representation, so both recognition problems lie in $\NP$.
Once the linear order of the interval endpoints and centres is prescribed, together with the colouring in the case of $\bicoloured$, recognition reduces to linear programming.
Whether either problem is polynomial or $\NP$-complete remains open, as it does for $\unit$ and $\proper$, whereas recognising tolerance graphs is $\NP$-complete~\cite{mertzios-sau-zaks}.
For the slab hypergraphs themselves, by contrast, membership of a triple is a quadratic condition on the coordinates, the natural upper bound is the existential theory of the reals, and we do not know whether recognition lies in $\NP$.

The rest of the paper is organised as follows.
\S\ref{sec:classes} introduces the two classes and proves the correspondence with links of slab hypergraphs.
\S\ref{sec:structural} presents the structural results.
\S\ref{sec:hierarchy} places both classes in the hierarchy, proves the separations, and constructs the family $Q^m$.
\S\ref{sec:recognition} treats recognition, and \S\ref{sec:open} closes with open problems.
Appendices~\ref{ap:proper-sandwich} and~\ref{ap:Qm} are devoted to the analysis of graphs separating $\proper$ from $\sandwich$ and $\bicoloured$, respectively.

\section{Graph classes}
\label{sec:classes}

This section introduces some of the graph classes we refer to throughout the paper.
We begin in \S\ref{sec:standard} by recalling the classical classes and fixing notation.
Central among these are the tolerance graphs, introduced by Golumbic and Monma~\cite{golumbic1982generalization}, which generalise interval graphs by letting intervals overlap without their vertices being adjacent, each vertex specifying how much overlap it tolerates.
We refer the reader to the monograph of Golumbic and Trenk~\cite{golumbic2004tolerance} for background and related results.

\S\ref{sec:bicoloured} and \S\ref{sec:sandwich} introduce the two new classes, and \S\ref{sec:links} connects them back to our original motivation, the slab hypergraphs.

\subsection{Interval and tolerance graphs}
\label{sec:standard}

\begin{definition}[$\interval$]
\label{def:interval}
A graph $G$ is an {\em interval graph} if each vertex $u \in V(G)$ can be assigned a closed bounded interval $I_u$ of positive length so that $uw \in E(G) \iff I_u \cap I_w \neq \emptyset$.
We write $\interval$ for the class of interval graphs.
\end{definition}

\begin{definition}[$\tol$~\cite{golumbic1982generalization}]
\label{def:tolerance}
A graph $G$ is a {\em tolerance graph} if each vertex $u \in V(G)$ can be assigned a closed bounded interval $I_u$ of positive length and a real {\em tolerance} $t_u > 0$ so that $uw \in E(G) \iff |I_u \cap I_w| \geq \min\{t_u, t_w\}$.
The set of pairs $\{(I_u, t_u) : u \in V(G)\}$ is a {\em tolerance representation} of $G$.
We write $\tol$ for the class of tolerance graphs.
\end{definition}

We will need four standard subclasses of $\tol$~\cites{golumbic1982generalization,golumbic-monma-trotter,bogart-fishburn-isaak-langley}.
Two of these bound the tolerances in terms of the intervals, and two constrain the shape of the intervals themselves.

\begin{definition}[$\bounded$, $\unit$, $\proper$, $\halftolerance$]
\label{def:tolerance-variants}
A tolerance representation is
\begin{itemize}
    \item {\em bounded} if $t_u \leq |I_u|$ for every $u$,
    \item {\em unit} if all the $|I_u|$ are equal,
    \item {\em proper} if no $I_u$ properly contains another $I_w$, and
    \item {\em $50\%$} if $t_u = |I_u|/2$ for every $u$.
\end{itemize}
We write $\bounded$, $\unit$, $\proper$, and $\halftolerance$ for the classes of graphs admitting a representation of each kind, respectively.
\end{definition}

We will also need the following geometric class.

\begin{definition}[$\parallelogram$]
\label{def:parallelogram}
A graph $G$ is a {\em parallelogram graph} if there are two distinct parallel lines $\ell_1, \ell_2$ such that each vertex $u \in V(G)$ can be assigned a parallelogram $P_u$, with parallel sides along $\ell_1$ and $\ell_2$, so that $uw \in E(G) \iff P_u \cap P_w \neq \emptyset$. The set $\{P_u : u \in V(G) \}$ is a {\em parallelogram representation} of $G$.
We write $\parallelogram$ for the class of parallelogram graphs.
\end{definition}

We note some relations among the graph classes (see, e.g.,~\cite{golumbic2004tolerance}).

\begin{fact}
\label{fact:classical}
\leavevmode
\begin{enumerate}[label=\normalfont(\roman*)]
  \item\label{it:bd-par} $\bounded = \parallelogram$~\cites{langley-thesis, bogart-fishburn-isaak-langley};
  \item\label{it:unit-half} $\unit = \halftolerance$~\cite{bogart-fishburn-isaak-langley};
  \item\label{it:int-unit} $\interval \subseteq \unit$~\cite{bogart-jacobson-langley-mcmorris} (see also~\cite{golumbic2004tolerance}*{Theorem~2.32}).
\end{enumerate}
\end{fact}

Throughout, it will be convenient to represent intervals $I_u$ via their centres, denoted by $c_u$, and radii, denoted by $r_u > 0$. So $I_u = [c_u - r_u, c_u + r_u]$ and $|I_u| = 2 r_u$.
In this notation, conditions on the intersections of intervals can be expressed as constraints on the distances between their centres.

\begin{observation}
\label{obs:thresholds}
Let $\{I_u\}$ be a family of intervals with centres $c_u$ and radii $r_u > 0$, and let $u \neq w$. Then the following hold:
\begin{enumerate}[label=\normalfont(\roman*)]
  \item\label{it:th-sum} $|c_u - c_w| \leq r_u + r_w$ $\iff$ $I_u \cap I_w \neq \emptyset$;
  \item\label{it:th-max} $|c_u - c_w| \leq \max\{r_u, r_w\}$ $\iff$ $|I_u \cap I_w| \geq \min\{r_u, r_w\}$ $\iff$ one of $I_u$ and $I_w$ contains the centre of the other.
\end{enumerate}
\end{observation}

\begin{proof}
Part~\ref{it:th-sum} is immediate, since two intervals intersect precisely when the distance between their centres is at most the sum of their radii.

Towards Part~\ref{it:th-max}, suppose, without loss of generality, that $r_u \leq r_w$.
Then $\max\{r_u, r_w\} = r_w$ and $\min\{r_u, r_w\} = r_u$, and the assertion is
\[ |c_u - c_w| \leq r_w \iff |I_u \cap I_w| \geq r_u \iff c_u \in I_w \text{ or } c_w \in I_u. \]
That the first and third are equivalent is immediate. We show that the first is equivalent to the second.

Suppose that $|c_u - c_w| \leq r_w$, so that by~\ref{it:th-sum}, $I_u \cap I_w \neq \emptyset$ and $|I_u \cap I_w| = \min\{2r_u,  2r_w,  r_u + r_w - |c_u - c_w|\}$. Noting that $|c_u - c_w| \leq r_w$ implies $r_u + r_w - |c_u - c_w| \geq r_u$, we obtain $|I_u \cap I_w| \geq r_u$.

For the converse, suppose, without loss of generality, that $c_w < c_u$ and that $c_u - c_w > r_w$.
If $I_u \cap I_w = \emptyset$ then $|I_u \cap I_w| = 0 < r_u$.
Otherwise $I_u \cap I_w \subseteq [c_u - r_u,  c_w + r_w]$, so $|I_u \cap I_w| \leq r_u + r_w - c_u + c_w < r_u$, completing the proof.
\end{proof}

We will often use Observation~\ref{obs:thresholds} without comment. If the vertices are labelled $v_0, v_1, \dots, v_k$, we abbreviate $I_{v_i}$, $c_{v_i}$, and $r_{v_i}$ as $I_i$, $c_i$, and $r_i$.
When a representation is fixed and no confusion can arise, we abuse terminology and identify a vertex with its interval or its centre.
For instance, we say that $u$ lies to the left of $w$ when $c_u < c_w$, and that $u$ meets $w$ when $I_u \cap I_w \neq \emptyset$.

\subsection{\texorpdfstring{$\bicoloured$}{Bicoloured-interval graphs}}
\label{sec:bicoloured}

As noted in the introduction, the following class arises geometrically, as the links of the slab hypergraphs (see \S\ref{sec:links}).
The definition below is natural on its own terms, as the rest of this subsection indicates.

\begin{definition}[$\bicoloured$]
\label{def:bicoloured}
A graph $G$ is a {\em bicoloured-interval graph} if each vertex $u \in V(G)$ can be assigned a real interval $I_u = [c_u - r_u, c_u + r_u]$ with $r_u > 0$ and a colour $\sigma(u) \in \{+, -\}$ so that for all $u \neq w$,
\[  uw \in E(G) \iff |c_u - c_w| \leq \tau_{uw}, \]
where
\[
  \tau_{uw} = \begin{cases}
     \max\{r_u, r_w\} & \text{if } \sigma(u) = \sigma(w),\\[2pt]
     r_u + r_w        & \text{if } \sigma(u) \neq \sigma(w).
  \end{cases}
\]
The triple $(\sigma, c, r)$ is a {\em bicoloured-interval representation} (or, more simply, bicoloured representation) of $G$.
We write $\bicoloured$ for the class of bicoloured-interval graphs.
\end{definition}

Restricting a bicoloured representation to a subset $W \subseteq V(G)$ gives a bicoloured representation of the induced subgraph $G[W]$.

\begin{observation}
\label{obs:B-hereditary}
$\bicoloured$ is hereditary.
\end{observation}

By Observation~\ref{obs:thresholds}, both cases of Definition~\ref{def:bicoloured} correspond to standard conditions on the intervals, connecting $\bicoloured$ to the tolerance classes defined in \S\ref{sec:standard}.
A bichromatic pair $uw$ is adjacent precisely when $I_u \cap I_w \neq \emptyset$, which is the adjacency rule of $\interval$, and a monochromatic pair precisely when $|I_u \cap I_w| \geq \min\{r_u, r_w\} = \min\{|I_u|, |I_w|\}/2$, which is that of $\halftolerance$.

We may also view $\tau_{uw}$ as a uniform threshold by considering the colour as the sign of the radius.
This is the form in which the class appears in \S\ref{sec:links}.

\begin{proposition}
\label{prop:signed-radius}
Let $\sigma(u), \sigma(w) \in \{+,-\}$ and $r_u, r_w > 0$, and set $\rho_u = \sigma(u) r_u$ and $\rho_w = \sigma(w) r_w$, considering the colours as $\pm 1$.
Then
\[  \tau_{uw} = \operatorname{diam}\{0, \rho_u, \rho_w\} ,  \]
where $\operatorname{diam} X = \max X - \min X$ for a finite $X \subseteq \RR$.

Consequently, $G \in \bicoloured$ if and only if each vertex $u \in V(G)$ can be assigned a real $c_u$ and a non-zero real $\rho_u$ so that
\[  uw \in E(G) \iff |c_u - c_w| \leq \operatorname{diam}\{0, \rho_u, \rho_w\} .  \]
\end{proposition}

\begin{proof}
If $\sigma(u) = \sigma(w)$ then $\rho_u$ and $\rho_w$ lie on the same side of $0$, so the diameter is $\max\{|\rho_u|, |\rho_w|\} = \max\{r_u, r_w\}$.
If $\sigma(u) \neq \sigma(w)$ then they lie on opposite sides of $0$ and the diameter is $|\rho_u| + |\rho_w| = r_u + r_w$.
In both cases, this is $\tau_{uw}$ as in Definition~\ref{def:bicoloured}.

The last assertion follows since $\rho_u$ determines $\sigma(u)$ and $r_u$ as its sign and absolute value.
\end{proof}

Our next proposition describes $\bicoloured$ using the parallelograms of Definition~\ref{def:parallelogram}, along with the additional constraint that each parallelogram must have a diagonal perpendicular to $\ell_1$ and $\ell_2$.

\begin{proposition}
\label{prop:B-parallelogram}
A graph $G$ is in $\bicoloured$ if and only if there are two distinct parallel lines $\ell_1, \ell_2$ such that each vertex $u \in V(G)$ can be assigned a
parallelogram $P_u$, with parallel sides along $\ell_1$ and $\ell_2$, and a diagonal perpendicular to both, so that
\[  uw \in E(G) \iff P_u \cap P_w \neq \emptyset .  \]
In particular, $\bicoloured \subseteq \parallelogram = \bounded$.
\end{proposition}

\begin{proof}
The equality $\parallelogram = \bounded$ is Fact~\ref{fact:classical}\ref{it:bd-par}, so it suffices to prove the characterisation.

Applying a rigid motion and then rescaling the vertical coordinate, we may assume that $\ell_1$ and $\ell_2$ are the lines $y = 1/2$ and $y = -1/2$.
Let $P$ be a parallelogram with parallel sides along $\ell_1$ and $\ell_2$, say $[p, p+L]$ on $\ell_1$ and $[q, q+L]$ on $\ell_2$, where $L > 0$ is their common length.
A diagonal of $P$ is perpendicular to $\ell_1, \ell_2$ precisely when it is vertical.
Since the diagonals of $P$ join $(p, 1/2)$ to $(q+L, -1/2)$ and $(p+L, 1/2)$ to $(q, -1/2)$, one of them is vertical precisely when $q = p+L$ or $p = q+L$.

If $q = p + L$, the two sides share the $x$-coordinate $p + L$.
Setting $c = p + L$ and $\rho = -L$, the sides are $[c+\rho,  c]$ on $\ell_1$ and $[c,  c-\rho]$ on $\ell_2$.
If instead $p = q + L$, the two sides share the $x$-coordinate $q + L$.
Setting $c = q + L$ and $\rho = L$, the sides are $[c,  c+\rho]$ on $\ell_1$ and $[c-\rho,  c]$ on $\ell_2$.
In either case $P$ is determined by a pair $(c, \rho)$ with $c \in \RR$ and $\rho \in \RR \setminus \{0\}$, and conversely every such pair gives a parallelogram, which we denote by $P(c, \rho)$.

Now suppose each vertex $u \in V(G)$ is assigned such a parallelogram, say $P_u = P(c_u, \rho_u)$, and set $r_u = |\rho_u|$, so that the projection of $P_u$ onto the $x$-axis is $I_u = [c_u - r_u,  c_u + r_u]$.
Fix $u \neq w$. Since $P_u$ and $P_w$ are compact convex sets, they are disjoint if and only if there exists a line strictly separating them.
Noting that such a line cannot be horizontal, since each of $P_u$ and $P_w$ meets both $\ell_1$ and $\ell_2$, we have that $P_u$ and $P_w$ are disjoint if and only if one lies strictly to the left of the other on both lines.

Suppose first that $\rho_u$ and $\rho_w$ have opposite signs, say $\rho_u > 0 > \rho_w$, and that $c_u \leq c_w$ (the case $c_w < c_u$ is symmetric, with $\ell_2$ in place of $\ell_1$).
Then $P_u$ and $P_w$ have opposite shear, so the rightmost point of $P_u$ and the leftmost point of $P_w$ both lie on $\ell_1$.
Hence, if their sides on $\ell_1$ intersect then $P_u$ and $P_w$ intersect, and if not then $c_u + r_u < c_w - r_w$, implying that $I_u$ and $I_w$ are disjoint.
It follows that $P_u$ and $P_w$ intersect if and only if $I_u \cap I_w \neq \emptyset$, i.e., if and only if $|c_u - c_w| \leq r_u + r_w$.

Otherwise, $\rho_u$ and $\rho_w$ have the same sign, and we may take both positive, since reflecting in the $x$-axis swaps $\ell_1$ with $\ell_2$ and sends each $P(c,\rho)$ to $P(c,-\rho)$, leaving every $I_u$ unchanged.
Then $P_u$ has sides $[c_u, c_u + r_u]$ on $\ell_1$ and $[c_u - r_u, c_u]$ on $\ell_2$, and $P_w$ has sides $[c_w, c_w + r_w]$ on $\ell_1$ and $[c_w - r_w, c_w]$ on $\ell_2$.
So $P_u$ lies strictly to the left of $P_w$ if and only if $c_u + r_u < c_w$ and $c_u < c_w - r_w$, i.e., $c_w - c_u > \max\{r_u, r_w\}$.
Hence, $P_u$ and $P_w$ intersect if and only if $|c_u - c_w| \leq \max\{r_u, r_w\}$.

Set $\sigma(u) = \operatorname{sign}(\rho_u)$ for each $u$.
From the above discussion, $P_u \cap P_w \neq \emptyset$ if and only if $|c_u - c_w| \leq \tau_{uw}$, so $\{P_u\}$ represents $G$ if and only if $(\sigma, c, r)$ is a bicoloured representation of $G$.
Conversely, every bicoloured representation $(\sigma, c, r)$ of $G$ arises this way, by taking $\rho_u = \sigma(u) r_u$, completing the proof.
\end{proof}

\begin{example}[$C_4 \in \bicoloured$]
\label{ex:C4}
Consider $C_4$ with vertices $v_0, v_1, v_2, v_3$ in cyclic order.
Colour vertices $v_0, v_2$ with $+$ and $v_1,v_3$ with $-$, and set
\[  r = (6,6,6,6), \quad\text{ and }\quad c = (-5,-5,5,5). \]
Note that all edges $v_iv_j$ are bichromatic and that $|c_i - c_j| \leq 10 < 12 = r_i + r_j$.
The two non-edges $v_iv_j$ are monochromatic and satisfy $|c_i - c_j| = 10 > 6 = \max\{r_i, r_j\}$.
It follows that $C_4 \in \bicoloured$.
See Figure~\ref{fig:C4-bicoloured} for an illustration.
\end{example}

\begin{figure}[tb]
\centering
\includegraphics[width=0.8\textwidth]{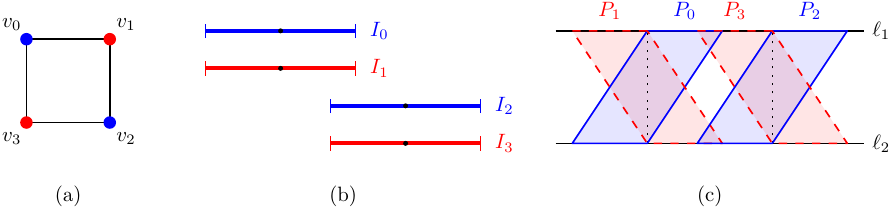}
\caption{The graph $C_4$ of Example~\ref{ex:C4} (a), the intervals of its bicoloured representation (b), and the corresponding parallelograms of Proposition~\ref{prop:B-parallelogram} (c), with $+/-$ denoted by blue/red.
The dotted segments are the vertical diagonals, and each parallelogram projects onto its interval.}
\label{fig:C4-bicoloured}
\end{figure}

\begin{remark}
\label{rem:C4-unit}
The colouring in Example~\ref{ex:C4} is not the only realisable one.
Colouring every vertex $+$ and setting
\[  r = (3, 1, 3, 1), \quad\text{ and }\quad
    c = (-2, -1, 2, 1) \]
gives a monochromatic representation of $C_4$.
However, not every colouring can be realised. Lemma~\ref{lem:C4-colouring} restricts which $2$-colourings of a $C_4$ occur in a bicoloured representation.
\end{remark}

Taking one colour class empty in Definition~\ref{def:bicoloured} leaves only the monochromatic rule, so that $\tau_{uw} = \max\{r_u, r_w\}$ for every pair.
This is exactly the $50\%$-tolerance condition, with $t_u = r_u$, so the graphs admitting a monochromatic bicoloured representation are precisely those in $\halftolerance$ (hence $C_4 \in \halftolerance = \unit$).
Together with Fact~\ref{fact:classical}\ref{it:int-unit} and~\ref{it:unit-half},
this gives the following.

\begin{proposition}
\label{prop:unit-in-B}
$\interval \subseteq \unit = \halftolerance \subseteq \bicoloured$.
\end{proposition}

\subsection{\texorpdfstring{$\sandwich$}{Interval-sandwich graphs}}
\label{sec:sandwich}

The next class is a relaxation of $\bicoloured$, obtained by discarding the colouring that decides which threshold applies to each pair of vertices.

\begin{definition}[$\sandwich$]
\label{def:sandwich}
A graph $G$ is an {\em interval-sandwich graph} if each vertex $u \in V(G)$ can be assigned a real interval $I_u = [c_u - r_u, c_u + r_u]$ with $r_u > 0$ so that for all $u \neq w$,
\begin{align*}
    uw \in E(G) & \implies |c_u - c_w| \leq r_u + r_w,\\[2pt]
    uw \notin E(G) & \implies |c_u - c_w| > \max\{r_u, r_w\}.
\end{align*}
The pair $(c, r)$ is an {\em interval-sandwich representation} (or, more simply, sandwich representation) of $G$.
We write $\sandwich$ for the class of interval-sandwich graphs.\footnote{
The name {\em interval sandwich} is also used, in a different sense, for a problem of Golumbic, Kaplan, and Shamir~\cite{golumbic-kaplan-shamir}.
The sense here is that of Proposition~\ref{prop:sandwich-name}.}
\end{definition}

As with $\bicoloured$, restricting a representation to $W \subseteq V(G)$ gives a representation of $G[W]$.

\begin{observation}
\label{obs:C-hereditary}
$\sandwich$ is hereditary.
\end{observation}

The two implications in Definition~\ref{def:sandwich} leave a gap. Specifically, a pair $uw$ with
$\max\{r_u, r_w\} < |c_u - c_w| \leq r_u + r_w$ is unconstrained and may be an edge or a non-edge.
In terms of the interval system $\{I_u\}$ itself, the first implication says that every edge of $G$ is an edge of its intersection graph, and the second says that every non-edge of $G$ is a non-edge of its $50\%$-tolerance graph.
Note that the latter is contained in the former, since an overlap of at least half the shorter interval is non-empty.
Hence, we have the following, which explains the name.

\begin{proposition}
\label{prop:sandwich-name}
Let $\{I_u = [c_u - r_u, c_u + r_u]\}_{u \in V}$ be a family of intervals with $r_u > 0$, let $X$ be its intersection
graph, and let $T$ be its $50\%$-tolerance graph.
Then $T \subseteq X$, and a graph $G$ on $V$ has $(c, r)$ as a sandwich representation if and only if
\[  T  \subseteq  G  \subseteq  X . \]
\end{proposition}

\begin{proof}
This is a restatement of the two implications above.
\end{proof}

Note that a sandwich representation need not determine the graph it represents. Every $G$ with $T \subseteq G \subseteq X$ has $(c, r)$ as a representation, and $X$ and $T$ differ exactly on the unconstrained pairs above.
This is in sharp contrast to Definitions~\ref{def:interval} and~\ref{def:bicoloured}, where the representation determines the graph outright.

Definitions~\ref{def:bicoloured} and~\ref{def:sandwich} both amount to choosing, for each pair $uw$, which of the two thresholds $\max\{r_u, r_w\}$ and $r_u + r_w$ that pair is tested against, and asking that $uw \in E(G)$ exactly when $|c_u - c_w|$ is at most the chosen threshold.
In $\sandwich$ the choice is unconstrained and may always be made canonically, taking $r_u + r_w$ on the edges and $\max\{r_u, r_w\}$ on the non-edges.
In $\bicoloured$ the pairs assigned $r_u + r_w$ must be exactly the bichromatic pairs of a $2$-colouring of the vertices.
Since $\max\{r_u, r_w\} \leq \tau_{uw} \leq r_u + r_w$ in either case of
Definition~\ref{def:bicoloured}, a bicoloured representation satisfies both implications of Definition~\ref{def:sandwich}.

\begin{corollary}
\label{cor:B-in-C}
If $(\sigma, c, r)$ is a bicoloured representation of $G$, then $(c, r)$
is a sandwich representation of $G$. In particular,
$\bicoloured \subseteq \sandwich$.
\end{corollary}

\subsection{Slab hypergraphs and links}
\label{sec:links}
In this subsection, we connect the class $\bicoloured$ back to our original motivation of understanding the structure of hypergraphs arising from the robust line fitting problem. The hypergraphs themselves are studied in a companion paper~\cite{companion}.

For $\epsilon \geq 0$, an {\em $\epsilon$-slab}, or a slab of {\em width} $\epsilon$, is the set of points of $\RR^2$
lying within vertical distance $\epsilon/2$ of some non-vertical line, i.e., a set
\[  S(m,b)  =  \left\{ (x,y) \in \RR^2 : |y - mx - b| \leq \epsilon/2 \right\}
  \quad\text{ with }\quad m, b \in \RR. \]
Fix $\epsilon > 0$, and let $P = \{p_0, \dots, p_{n-1}\} \subseteq \RR^2$ be a finite set of points with pairwise distinct $x$-coordinates, and write $p_i = (x_i, y_i)$.
The {\em slab hypergraph} of $P$ is the $3$-uniform hypergraph $H(P)$ on $\{0, \dots, n-1\}$ in which
\[ ijk \in E(H(P)) \iff p_i, p_j, p_k \text{ lie in a common $\epsilon$-slab}. \]
Restricting attention to triples loses nothing.
Indeed, for a point $p$, the set of pairs $(m,b)$ with $p \in S(m,b)$ is a convex subset of $\RR^2$, so by Helly's theorem, a set of at least three points lies in a common slab if and only if every three of them do.
We remark that the hypothesis on the $x$-coordinates, which the proof of Theorem~\ref{thm:link} uses, is not simply a general position requirement.
Indeed, three points on a vertical line need a slab as wide as their $y$-extent, while an arbitrarily small perturbation can put them in a slab of width $0$.

A $3$-uniform hypergraph is a {\em slab hypergraph} if it is isomorphic to $H(P)$ for some finite $P \subseteq \RR^2$ with pairwise distinct $x$-coordinates. We refer to such a $P$ as a {\em realisation} of $H$, identifying $V(H)$ with $P$ and writing $p_u = (x_u, y_u)$ for the point of the vertex $u$.
Since scaling by $\lambda > 0$ scales all slab widths by $\lambda$, the
class does not depend on the choice of $\epsilon$. We set $\epsilon = 1$ in this subsection and suppress it from the notation.

For a triple of points $p_i,p_j,p_k$, the {\em $x$-extent}, denoted by $\xi(p_i,p_j,p_k)$, is the largest of the three pairwise horizontal distances, and $\Delta(p_i,p_j,p_k)$ is twice the signed area of the triangle spanned by the points. More formally,
\[
  \xi(p_i,p_j,p_k)  =  \max   \left\{|x_i - x_j|,  |x_i - x_k|,  |x_j - x_k|  \right\},\quad\text{ and }\quad
  \Delta(p_i,p_j,p_k)
  = \begin{vmatrix}
      1   & 1   & 1\\
      x_i & x_j & x_k\\
      y_i & y_j & y_k
    \end{vmatrix}.
\]
Note that $\xi$ and $|\Delta|$ are symmetric in the three arguments.

\begin{lemma}
\label{lem:residual}
Suppose $p_i,p_j,p_k \in \RR^2$ have distinct $x$-coordinates. The smallest $w \geq 0$ for which there is a $w$-slab containing $p_i,p_j,p_k$ is
\[  w(p_i,p_j,p_k) =  \mfrac{  \left|\Delta(p_i,p_j,p_k)  \right|}{\xi(p_i,p_j,p_k)}. \]
Consequently, $ijk \in E(H(P))$ if and only if $w(p_i,p_j,p_k) \leq 1$.
\end{lemma}

\begin{proof}
Let $g_\ell(m) = y_\ell - m x_\ell$ for $\ell \in \{i,j,k\}$.
A slab of slope $m$, say $|y - mx - b| \leq w/2$, contains the three points precisely when $[b - w/2, b + w/2]$ contains $g_i(m), g_j(m), g_k(m)$, so the least width of such a slab is $f(m) = \max_\ell g_\ell(m) - \min_\ell g_\ell(m)$, and
$w(p_i,p_j,p_k) = \min_m f(m)$.

Suppose, without loss of generality, that $x_i < x_j < x_k$.
So $\xi(p_i,p_j,p_k) = x_k - x_i$ and $x_j = \lambda x_i + (1-\lambda) x_k$ where $\lambda = (x_k - x_j)/(x_k - x_i) \in (0,1)$.
Now $\lambda g_i(m) + (1-\lambda) g_k(m)$ is a convex combination of $g_i(m)$ and $g_k(m)$, so it lies in the interval $\left[\min_\ell g_\ell(m), \max_\ell g_\ell(m)  \right]$.
So does $g_j(m)$.
Hence, the two differ by at most $f(m)$, implying
\begin{align*}
f(m) & \geq \left| g_j(m) -   \left[\lambda g_i(m) + (1-\lambda) g_k(m)  \right] \right|
 =    \left| y_j - \lambda y_i - (1-\lambda) y_k   \right|\\
&  =  \mfrac{  \left| (x_k - x_i) y_j - (x_k - x_j) y_i - (x_j - x_i) y_k   \right|}{x_k - x_i}
  =  \mfrac{|\Delta(p_i,p_j,p_k)|}{\xi(p_i,p_j,p_k)},
\end{align*}
and hence $w(p_i,p_j,p_k) = \min_m f(m) \geq |\Delta|/\xi$.

The bound is attained at $m^{*} = (y_k - y_i)/(x_k - x_i)$, the slope of the line through $p_i$ and $p_k$, since $g_i(m^{*}) = g_k(m^{*})$, so $f(m^{*}) = \left| g_j(m^{*}) - g_i(m^{*}) \right| = |\Delta|/\xi$. It follows that $w(p_i,p_j,p_k) = |\Delta|/\xi$, completing the proof.
\end{proof}

Lemma~\ref{lem:residual} is the three-point case of the alternation theorem for Chebyshev approximation, as $w(p_i,p_j,p_k)/2$ is the least maximum vertical residual of a line~\cite{cheney}.

Let $H$ be a $3$-uniform hypergraph and let $v \in V(H)$.
The {\em link} of $v$ is the
graph $\link_H(v)$ on vertex set $V(H) \setminus \{v\}$ with
\[
  uw \in E(\link_H(v)) \iff vuw \in E(H).
\]
Links let us view the $3$-uniform structure ``from one vertex'' as a graph.
For a slab hypergraph $H$, the link of any vertex is a bicoloured-interval graph, and a realisation of $H$ gives a bicoloured representation of it.

\begin{theorem}
\label{thm:link}
Let $H$ be a slab hypergraph, $v \in V(H)$, and fix a realisation translated so that $p_v$ is the origin, as translations preserve slab widths.
For $u \neq v$, set
\[
  c_u = \mfrac{y_u}{x_u}, \quad \rho_u = \mfrac{1}{x_u},
  \quad r_u = |\rho_u|, \quad \sigma(u) = \operatorname{sign}(\rho_u) .
\]
Then $(\sigma, c, r)$ is a bicoloured representation of $\link_H(v)$. That is, for all
$u \neq w$ in $V(H) \setminus \{v\}$, \[ uw \in E(\link_H(v)) \iff |c_u - c_w| \leq \tau_{uw} . \]
\end{theorem}

\begin{proof}
The points have distinct $x$-coordinates and $p_v$ is the origin, so
$x_u \neq 0$ for every $u \neq v$ and the quantities $c_u$, $\rho_u$, $r_u$, and $\sigma(u)$ above are well-defined, with $r_u > 0$.

Since $p_v = (0,0)$,
\[ \Delta(p_v, p_u, p_w) = x_u y_w - x_w y_u = x_u x_w (c_w - c_u)
  \quad\text{ and }\quad
  \xi(p_v, p_u, p_w) = \max\{|x_u|,  |x_w|,  |x_u - x_w|\}. \]
By Lemma~\ref{lem:residual}, $vuw \in E(H)$ if and only if $|\Delta| \leq \xi$,
which gives
\begin{equation}
\label{eq:slab-link}
    |c_u - c_w|  \leq  \mfrac{\max\{|x_u|, |x_w|, |x_u - x_w|\}}{|x_u| |x_w|} .
\end{equation}
If $\sigma(u) = \sigma(w)$,
then $p_u$ and $p_w$ lie on the same side of the $y$-axis, and $|x_u - x_w| < \max\{|x_u|, |x_w|\}$. So the right-hand side of~\eqref{eq:slab-link} is
\[
  \mfrac{\max\{|x_u|, |x_w|\}}{|x_u| |x_w|}
  = \mfrac{1}{\min\{|x_u|, |x_w|\}} = \max\{r_u, r_w\} .
\]
If $\sigma(u) \neq \sigma(w)$, then $p_u$ and $p_w$ lie on opposite sides of the $y$-axis, and $|x_u - x_w| = |x_u| + |x_w|$. So the right-hand side of~\eqref{eq:slab-link} is
\[
  \mfrac{|x_u| + |x_w|}{|x_u| |x_w|}
  =  \mfrac{1}{|x_u|} +  \mfrac{1}{|x_w|} = r_u + r_w.
\]

In both cases, the right-hand side is $\operatorname{diam}\{0, \rho_u, \rho_w\}$, which is $\tau_{uw}$ by Proposition~\ref{prop:signed-radius}, completing the proof.
\end{proof}

As an immediate consequence of Theorem~\ref{thm:link}, we obtain the following.

\begin{corollary}
\label{cor:links-in-B}
If $H$ is a slab hypergraph then $\link_H(v) \in \bicoloured$ for every
$v \in V(H)$.
\end{corollary}

The converse also holds --- every bicoloured-interval graph arises as a link of some slab hypergraph.

\begin{proposition}
\label{prop:B-is-links}
$\bicoloured$ is exactly the class of links of slab hypergraphs. Specifically,  $G \in \bicoloured$
if and only if $G \cong \link_H(v)$ for some slab hypergraph $H$ and vertex $v \in V(H)$.
\end{proposition}

\begin{proof}
One direction is Corollary~\ref{cor:links-in-B}.
Conversely, let $G \in \bicoloured$ with representation $(\sigma, c, r)$.
We exhibit a slab hypergraph $H$ and vertex $v$ such that $G \cong \link_H(v)$.

By increasing each radius by a small amount, we may assume that the radii are pairwise distinct within each colour class, since only finitely many equalities have to be avoided.
Indeed, both thresholds $\max\{r_u, r_w\}$ and $r_u + r_w$ are monotone in the radii, so no edge condition is weakened, and the finitely many non-edge conditions are strict, so a small enough increase preserves them all.

Now place the point $p_v$ at the origin together with points
\[
  p_u = (x_u, c_u x_u), \quad \text{where}\quad x_u = \sigma(u)/r_u \quad\text{for each }u \in V(G),
\]
considering the colours as $\pm 1$.
These $x$-coordinates are non-zero, and they are pairwise distinct, since two vertices
of the same colour have distinct radii, and two of different colours have
$x$-coordinates of opposite sign.

Let $H$ be the slab
hypergraph of this point set. Vertex $u$ has $y_u/x_u = c_u$ and
$\rho_u = 1/x_u = \sigma(u) r_u$, hence $|\rho_u| = r_u$ and
$\operatorname{sign}(\rho_u) = \sigma(u)$. By Theorem~\ref{thm:link},
colouring by $\operatorname{sign}(\rho_u) = \sigma(u)$, a pair $uw$ is an edge of $\link_H(v)$ if and only if the inequality $|c_u - c_w| \leq \tau_{uw}$ of Definition~\ref{def:bicoloured} holds, i.e., if and only if $uw \in E(G)$.
It follows that $\link_H(v) \cong G$, completing the proof.
\end{proof}

\section{Structural results}
\label{sec:structural}

This section collects structural properties of $\sandwich$ and $\bicoloured$.
We say a graph is a {\em minimal forbidden induced subgraph} for a class if it lies outside the class while every proper induced subgraph lies in it.
Since both classes are hereditary (Observations~\ref{obs:B-hereditary} and~\ref{obs:C-hereditary}), verifying minimality only requires checking the deletions of single vertices.
We resolve the situation completely for holes, antiholes, trees, and complete bipartite graphs, in \S\ref{sec:holes}, \S\ref{sec:antiholes}, \S\ref{sec:trees}, and \S\ref{sec:bipartite} respectively.
The {\em holes} $C_k$ with $k \geq 5$, the {\em antiholes} $\overline{C_n}$ with $n \geq 5$, the spider $T_2$, and $K_{3,3}$ are minimal forbidden induced subgraphs for $\bicoloured$.
The same holds for $\sandwich$, except that the even antiholes are contained in $\sandwich$.
As a consequence, we obtain the following.

\begin{corollary}
\label{cor:no-finite}
Neither $\bicoloured$ nor $\sandwich$ admits a characterisation by finitely many forbidden induced subgraphs.
\end{corollary}

Let $G$ be a graph with vertex set $V = V(G)$ and edge set $E = E(G)$.
For $S \subseteq V$, we write $G[S]$ for the subgraph induced by $S$ and $G - S$ for $G[V \setminus S]$. For a vertex $v \in V$, we set $G - v = G[V \setminus \{v\}]$, while $G + v$ denotes the graph obtained from $G$ by adding a new vertex $v$ whose neighbours are either specified alongside or clear from context.
We let $N_G(v)$ denote the neighbourhood of a vertex $v$ and $N_G[v] = N_G(v) \cup \{v\}$ its closed neighbourhood, dropping the subscript when $G$ is clear from the context.
For representations of $G$, we continue with the interval notation of \S\ref{sec:classes}, and additionally write $\tau_{i,j}$ for the threshold $\tau_{v_iv_j}$.
We begin by collecting some tools that will be used throughout the paper. 

Recall that, by Corollary~\ref{cor:B-in-C}, if $(\sigma, c, r)$ is a bicoloured representation then $(c, r)$ is a sandwich representation of the same graph, so every statement about sandwich representations applies to bicoloured representations as well.
We use this repeatedly without further comment.
The converse need not hold, since a statement about a bicoloured representation may constrain the colouring, and a sandwich representation has none.

The following is an immediate consequence of Definition~\ref{def:sandwich} and Observation~\ref{obs:thresholds}.

\begin{observation}
\label{obs:representation}
Let $(c, r)$ be a sandwich representation of a graph $G$, and let $u \neq w$ be vertices of $G$.
Then the following hold:
\begin{enumerate}[label=\normalfont(\roman*)]
  \item\label{it:wa-edge} if $I_u \cap I_w = \emptyset$ then $uw \notin E(G)$;
  \item\label{it:wa-nonedge} if $uw \notin E(G)$ then $c_u \notin I_w$ and $c_w \notin I_u$, and in particular $c_u \neq c_w$.
\end{enumerate}
\end{observation}

Centres of non-adjacent vertices are separated by more than either radius.
The next lemma shows that this separation, along with the centre ordering, forces certain substructures in the graph.

\begin{lemma}
\label{lem:additivity}
Let $(c, r)$ be a sandwich representation of a graph $G$, and let $v_0, v_1, v_2$ be vertices of $G$ with $v_0v_1, v_1v_2 \notin E(G)$ and $c_0 < c_1 < c_2$.
Then the following hold:
\begin{enumerate}[label=\normalfont(\roman*)]
  \item\label{it:add-far} $c_2 - c_0 > r_0 + r_2$ and, hence, $v_0v_2 \notin E(G)$;
  \item\label{it:add-cover} if $v_3$ is adjacent to both $v_0$ and $v_2$, then $c_1 \in I_3$ and, hence, $v_3$ is adjacent to $v_1$.
\end{enumerate}
\end{lemma}

\begin{proof}
\ref{it:add-far}:
The two non-edges give
\[ c_1 - c_0 > \max\{r_0, r_1\} \geq r_0 \quad \text{and} \quad c_2 - c_1 > \max\{r_1, r_2\} \geq r_2 . \]
Summing these, we obtain $c_2 - c_0 > r_0 + r_2$, so $I_0 \cap I_2 = \emptyset$ and, hence, $v_0v_2 \notin E(G)$.

\ref{it:add-cover}:
The edges $v_3v_0$ and $v_3v_2$ give $c_3 - r_3 \leq c_0 + r_0$ and $c_3 + r_3 \geq c_2 - r_2$, so $[c_0 + r_0, c_2 - r_2] \subseteq I_3$.
It follows that $c_1 \in I_3$ and, hence, $v_3$ is adjacent to $v_1$.
\end{proof}

The following consequence of Lemma~\ref{lem:additivity}\ref{it:add-far} says that a connected set of non-neighbours of a vertex lies entirely to one side of its centre.

\begin{corollary}
\label{cor:propagate}
Let $(c, r)$ be a sandwich representation of a graph $G$, and let $v_0 \in V(G)$ and $S \subseteq V(G)$ be such that $S$ induces a connected subgraph of $G$ and $S \cap N[v_0] = \emptyset$.
Then either $c_v < c_0$ for every $v \in S$ or $c_v > c_0$ for every $v \in S$.
\end{corollary}

\begin{proof}
By Observation~\ref{obs:representation}\ref{it:wa-nonedge}, $c_v \neq c_0$ for each $v \in S$.
Suppose, for contradiction, that vertices of $S$ have centres on both sides of $c_0$.
Since $G[S]$ is connected, there exist vertices $u, w \in S$ such that $uw \in E(G)$ and $c_w < c_0 < c_u$.
As $v_0$ is adjacent to neither $u$ nor $w$, Lemma~\ref{lem:additivity}\ref{it:add-far} applied to $w, v_0, u$ gives $uw \notin E(G)$, a contradiction.
\end{proof}

The following two consequences of Lemma~\ref{lem:additivity}\ref{it:add-cover} will also be needed, the first purely geometric and the second constraining the colouring.

\begin{corollary}
\label{cor:forcing}
Let $(\sigma, c, r)$ be a bicoloured representation of a graph $G$, and let $u_0, u_1$ be vertices of $G$.
Then the following hold:
\begin{enumerate}[label=\normalfont(\roman*)]
  \item\label{it:forcing-geo} Suppose $I_{u_0} \cap I_{u_1} = \emptyset$ and that $v_0, v_1$ are non-adjacent common neighbours of $u_0, u_1$.
        Then any $v_2 \notin \{v_0, v_1\}$ non-adjacent to both $v_0$ and $v_1$ has $c_{v_2} < \min\{c_{v_0}, c_{v_1}\}$ or $c_{v_2} > \max\{c_{v_0}, c_{v_1}\}$.
  \item\label{it:forcing-col} If $u_0$ and $u_1$ are non-adjacent and have three pairwise non-adjacent common neighbours, then $\sigma(u_0) = \sigma(u_1)$.
\end{enumerate}
\end{corollary}

\begin{proof}
\ref{it:forcing-geo}:
The vertices $v_0, v_1, v_2$ are pairwise non-adjacent, so they have distinct centres by Observation~\ref{obs:representation}\ref{it:wa-nonedge}.
Suppose, without loss of generality, that $c_{v_0} < c_{v_1}$, and, for contradiction, that $c_{v_0} < c_{v_2} < c_{v_1}$.
But then $v_0, v_1$ are adjacent to $u_0, u_1$, so Lemma~\ref{lem:additivity}\ref{it:add-cover} implies that $c_{v_2} \in I_{u_0} \cap I_{u_1}$, contradicting $I_{u_0} \cap I_{u_1} = \emptyset$.

\ref{it:forcing-col}:
Suppose $\sigma(u_0) \neq \sigma(u_1)$.
Since $u_0u_1$ is a bichromatic non-edge, we have $I_{u_0} \cap I_{u_1} = \emptyset$.
The three common neighbours have pairwise distinct centres, so one of them, say $v$, lies strictly between the other two.
But then~\ref{it:forcing-geo}, with $v$ as $v_2$, says that $c_v$ lies strictly to one side of the centres of the other neighbours, a contradiction.
\end{proof}

Corollary~\ref{cor:propagate} yields a structural property of the whole class, which we use both for holes and for trees.
Recall that three pairwise non-adjacent vertices of a graph form an {\em asteroidal triple} if any two of them are joined by a path avoiding the closed neighbourhood of the third.
A graph is {\em AT-free} if it has no asteroidal triple.

\begin{proposition}
\label{prop:AT-free}
Every graph in $\sandwich$ is AT-free.
\end{proposition}

\begin{proof}
Let $(c, r)$ be a sandwich representation of $G$ and suppose, for contradiction, that $u, v, w$ form an asteroidal triple.
The three are pairwise non-adjacent, so their centres are distinct by Observation~\ref{obs:representation}\ref{it:wa-nonedge}, and we may assume $c_u < c_v < c_w$.
By definition, some path from $u$ to $w$ avoids $N[v]$.
The vertices of this path induce a connected subgraph of $G$ disjoint from $N[v]$, so their centres all lie on one side of $c_v$ by Corollary~\ref{cor:propagate}.
But this contradicts the assumption that $c_u < c_v < c_w$.
\end{proof}

We note that Proposition~\ref{prop:AT-free} also follows from Theorem~\ref{thm:c-in-cocomparability} and the classical fact that co-comparability graphs are AT-free~\cite{golumbic-monma-trotter}*{Theorem~4}.
However, the proofs in this section are elementary and self-contained. 

\subsection{Holes}
\label{sec:holes}

Throughout, $C_k$ has vertex set $\{v_0, \dots, v_{k-1}\}$, with indices modulo $k$, so that its edges are the $k$ pairs $v_iv_{i+1}$, $i = 0, \dots, k-1$.

\begin{theorem}
\label{thm:c-cycle-free}
If $G \in \sandwich$ then $G$ has no hole $C_k$ with $k \geq 6$.
\end{theorem}

\begin{proof}
Suppose that $v_0, \dots, v_{k-1}$ induce a hole in $G$, in this cyclic order, with $k \geq 6$.
Then $v_0, v_2, v_4$ are pairwise non-adjacent, and $v_0v_1v_2$, $v_2v_3v_4$, and $v_0v_{k-1} \cdots v_5v_4$ are paths avoiding $N[v_4]$, $N[v_0]$, and $N[v_2]$ respectively.
So $v_0, v_2, v_4$ form an asteroidal triple of $G$, and $G \notin \sandwich$ by Proposition~\ref{prop:AT-free}.
\end{proof}

The hole $C_5$ is excluded as well, since $C_5 \cong \overline{C_5}$ and Theorem~\ref{thm:odd-antiholes} below excludes the odd antiholes.
Deleting any vertex from $C_k$ leaves $P_{k-1}$, which is an interval graph and so lies in $\bicoloured$ by Proposition~\ref{prop:unit-in-B}.
Hence, for $k \geq 5$, the hole $C_k$ is a minimal forbidden induced subgraph for both $\sandwich$ and $\bicoloured$.

Example~\ref{ex:C4} gives a bicoloured representation of $C_4$, so $C_4 \in \bicoloured \subseteq \sandwich$ and the bound $k \geq 5$ is sharp.
However, our next result restricts which $2$-colourings of $C_4$ can occur in a bicoloured representation.

\begin{lemma}
\label{lem:C4-colouring}
Let $(\sigma, c, r)$ be a bicoloured representation of a graph $G$, and let $v_0, v_1, v_2, v_3$ induce a $4$-cycle in $G$, in this cyclic order.
Suppose some pair of opposite edges is monochromatic, i.e., either $\sigma(v_0) = \sigma(v_1)$ and $\sigma(v_2) = \sigma(v_3)$, or $\sigma(v_1) = \sigma(v_2)$ and $\sigma(v_3) = \sigma(v_0)$.
Then all four vertices have the same colour.
\end{lemma}

\begin{proof}
Suppose, for contradiction, that the colours are not all equal.
Rotating the labels along the cycle and swapping the two colours if necessary, we may assume $\sigma(v_0) = \sigma(v_1) = +$ and $\sigma(v_2) = \sigma(v_3) = -$.
Definition~\ref{def:bicoloured} then gives
\begin{itemize}
\item $|c_0 - c_1| \leq \max\{r_0, r_1\}$ and $|c_2 - c_3| \leq \max\{r_2, r_3\}$, from the monochromatic edges $v_0v_1$ and $v_2v_3$;
\item $|c_1 - c_2| \leq r_1 + r_2$ and $|c_3 - c_0| \leq r_3 + r_0$, from the bichromatic edges $v_1v_2$ and $v_3v_0$;
\item $|c_0 - c_2| > r_0 + r_2$ and $|c_1 - c_3| > r_1 + r_3$, from the bichromatic non-edges $v_0v_2$ and $v_1v_3$.
\end{itemize}
Negating every centre if necessary, we may assume
\begin{equation}
\label{eq:C4-norm}
  c_2 - c_0  >  r_0 + r_2 .
\end{equation}

We consider two cases based on the centre order of the non-edge $v_1v_3$.

{\em Case 1: $c_3 - c_1 > r_1 + r_3$.} Adding the hypothesis to~\eqref{eq:C4-norm} and regrouping,
\[
  r_0 + r_1 + r_2 + r_3 < (c_2 - c_0) + (c_3 - c_1)
  = (c_2 - c_1) + (c_3 - c_0) \leq (r_1 + r_2) + (r_3 + r_0),
\]
where the last inequality follows from the two bichromatic edges, a contradiction.

{\em Case 2: $c_1 - c_3 > r_1 + r_3$.} The monochromatic edges give $c_1 - c_0 \leq \max\{r_0, r_1\}$ and $c_2 - c_3 \leq \max\{r_2, r_3\}$, so by~\eqref{eq:C4-norm},
\[ c_1 - c_3 = (c_1 - c_0) - (c_2 - c_0) + (c_2 - c_3) < (\max\{r_0, r_1\} - r_0) + (\max\{r_2, r_3\} - r_2) \leq r_1 + r_3 , \]
a contradiction.
\end{proof}

Lemma~\ref{lem:C4-colouring} is sharp, and every other colouring is realisable.
Up to automorphisms of $C_4$ and interchanging the two colours, it suffices to consider the following four colourings:
\begin{itemize}
  \item $\{v_0,v_1,v_2,v_3\}$: realised by the representation in Remark~\ref{rem:C4-unit}.
  \item $\{v_0,v_1,v_2\}, \{v_3\}$: colour $v_0,v_1,v_2$ with $+$ and $v_3$ with $-$, and set
        \[  r = (3,1,3,1), \quad\text{ and }\quad
             c = (0,-3,-4,0). \]
  \item $\{v_0,v_2\}, \{v_1,v_3\}$: realised by the representation in Example~\ref{ex:C4}.
  \item $\{v_0,v_1\}, \{v_2,v_3\}$: forbidden by Lemma~\ref{lem:C4-colouring}.
\end{itemize}

\subsection{Antiholes}
\label{sec:antiholes}

Throughout, $\overline{C_n}$ has vertex set $\{v_0, \dots, v_{n-1}\}$, with indices modulo $n$, so that its non-edges are the $n$ pairs $v_iv_{i+1}$, $i = 0, \dots, n-1$.

Antiholes behave quite differently from holes, and differently again for the two classes.
We begin with what a representation of an antihole must look like.

\begin{lemma}
\label{lem:antihole-alt}
Let $n \geq 4$ and let $(c, r)$ be a sandwich representation of $\overline{C_n}$.
Then the following hold:
\begin{enumerate}[label=\normalfont(\roman*)]
  \item\label{it:alt-local} for every $i$, the centres $c_{i-1}$ and $c_{i+1}$ lie strictly on the same side of $c_i$;
  \item\label{it:alt-global} $n$ is even and, after negating every centre if necessary,
        \[ c_i < c_{i-1} \quad\text{and}\quad c_i < c_{i+1}
          \quad\text{for every even } i. \]
\end{enumerate}
\end{lemma}

\begin{proof}
\ref{it:alt-local}:
Let $0 \leq i \leq n-1$ be fixed.
The pairs $v_{i-1}v_i$ and $v_iv_{i+1}$ are non-edges, so by definition $c_{i-1} \neq c_i$ and $c_i \neq c_{i+1}$.
If $c_{i-1}$ and $c_{i+1}$ are on opposite sides of $c_i$, then Lemma~\ref{lem:additivity}\ref{it:add-far} applied to the non-edges $v_{i-1}v_i$ and $v_iv_{i+1}$ implies that $v_{i-1}v_{i+1}$ is a non-edge.
But since $n \geq 4$, $v_{i-1}v_{i+1}$ is an edge, a contradiction.

\ref{it:alt-global}:
By the argument in~\ref{it:alt-local}, consecutive vertices have distinct centres.
Orient each pair $v_iv_{i+1}$ from the smaller centre to the larger.
This gives an orientation of the cycle formed by the non-edges of $\overline{C_n}$.
A directed path $v_{i-1}, v_i, v_{i+1}$ means $c_{i-1} < c_i < c_{i+1}$, and $v_{i+1}, v_i, v_{i-1}$ means $c_{i+1} < c_i < c_{i-1}$.
In either case $c_{i-1}$ and $c_{i+1}$ lie on opposite sides of $c_i$, which is impossible by~\ref{it:alt-local}.
So no vertex has both an in-neighbour and an out-neighbour, i.e., every vertex is a source or a sink.
Note that the sources and the sinks form the classes of a proper $2$-colouring of the cycle and, hence, the cycle is bipartite.
It follows that $n$ is even and that the two classes are the even and the odd indices.
Negating every centre if necessary, we may assume that the vertices with even indices are the sources, which implies the assertion.
\end{proof}

Lemma~\ref{lem:antihole-alt}\ref{it:alt-global} immediately implies the following.

\begin{theorem}
\label{thm:odd-antiholes}
For every $m \geq 2$ we have $\overline{C_{2m+1}} \notin \sandwich$.
\end{theorem}

Even antiholes, by contrast, all lie in $\sandwich$.

\begin{proposition}
\label{prop:even-antiholes}
For every $m \geq 2$ we have $\overline{C_{2m}} \in \sandwich$.
\end{proposition}

\begin{proof}
For each $0 \leq i \leq 2m-1$, set
\[
  c_i = \begin{cases} 0 & \text{if $i$ is even},\\[2pt]
                      \mfrac{3}{2} & \text{if $i$ is odd},\end{cases}
  \quad\text{ and }\quad r_i = 1.
\]
If $v_i, v_j$ have the same parity, then they are adjacent in $\overline{C_{2m}}$, since consecutive vertices have opposite parities.
Note that $|c_i - c_j| = 0 \leq 2 = r_i + r_j$, so the implications of Definition~\ref{def:sandwich} hold for this pair.

If $v_i, v_j$ have opposite parities, then $|c_i - c_j| = 3/2$.
If $v_i, v_j$ are adjacent, then $3/2 \leq 2 = r_i + r_j$.
If they are non-adjacent, and so consecutive on the cycle, then $3/2 > 1 = \max\{r_i, r_j\}$.
In each case, the implications of Definition~\ref{def:sandwich} hold for this pair.
\end{proof}

Theorem~\ref{thm:odd-antiholes} and Proposition~\ref{prop:even-antiholes} together settle antiholes for $\sandwich$ completely.
For $n \geq 4$, $\overline{C_n} \in \sandwich$ if and only if $n$ is even.
So $\sandwich$ contains arbitrarily large even antiholes, in sharp contrast to the holes, of which the class contains none of length at least five.

For $\bicoloured$ the picture is different.
No antihole on at least five vertices is a bicoloured-interval graph.
This follows from Proposition~\ref{prop:B-parallelogram}, since tolerance graphs are weakly chordal~\cite{golumbic2004tolerance}*{Theorem~2.17}, that is, contain no hole or antihole on at least five vertices.
We nevertheless give a direct proof, which requires the following two lemmas.
The first lemma applies Lemma~\ref{lem:C4-colouring} to the induced $4$-cycles of $\overline{C_n}$, and restricts which colourings can occur in a bicoloured representation.

\begin{lemma}
\label{lem:antihole-colourings}
Let $n \geq 6$ and let $(\sigma, c, r)$ be a bicoloured representation of $\overline{C_n}$.
Then one of the colour classes of $\sigma$ is empty, a single vertex, or a pair of consecutive vertices $\{v_i, v_{i+1}\}$.
\end{lemma}

\begin{proof}
Let $B = \{ i : \sigma(v_i) \neq \sigma(v_{i+1}) \}$, i.e., $B$ is the set of colour changes along one traversal of the cycle.
In particular, $|B|$ is even.
If $B = \emptyset$ then $\sigma$ is constant and one colour class is empty, so the assertion holds.
From here on, we assume $B \neq \emptyset$, so $|B| \geq 2$.

Suppose $i, j \in B$ with $v_i, v_j$ at cyclic distance at least $3$.
The four vertices $v_i, v_{i+1}, v_j, v_{j+1}$ are thus distinct.
Moreover, each of the pairs $v_iv_j$, $v_jv_{i+1}$, $v_{i+1}v_{j+1}$, $v_{j+1}v_i$ is at cyclic distance at least $2$, hence an edge of $\overline{C_n}$.
So $v_i, v_j, v_{i+1}, v_{j+1}$ induce a $4$-cycle in that cyclic order, with the non-edges $v_iv_{i+1}$ and $v_jv_{j+1}$ as its diagonals.
As $i, j \in B$ we have $\sigma(v_i) \neq \sigma(v_{i+1})$ and $\sigma(v_j) \neq \sigma(v_{j+1})$.
Now either $\sigma(v_i) = \sigma(v_j)$ and, hence, $\sigma(v_{i+1}) = \sigma(v_{j+1})$, or $\sigma(v_j) = \sigma(v_{i+1})$ and, hence, $\sigma(v_{j+1}) = \sigma(v_i)$.
In either case a pair of opposite edges of the $4$-cycle is monochromatic so, by Lemma~\ref{lem:C4-colouring}, all four vertices have the same colour, contradicting $\sigma(v_i) \neq \sigma(v_{i+1})$.
It follows that the elements of $B$ are at pairwise cyclic distance at most $2$.

Fix $i \in B$, so by the argument above $B \subseteq \{i-2, i-1, i, i+1, i+2\}$.
Since $n \geq 6$, the pairs $\{i-2, i+1\}$ and $\{i-1, i+2\}$ are at cyclic distance $3$, so $B$ contains at most one member of each.
Hence, $|B| \leq 3$.
Since $B$ is non-empty and even, this implies $|B| = 2$.

Let $B = \{i, j\}$ where, without loss of generality, $j \in \{i+1, i+2\}$.
Note that $B$ cuts the cycle into two arcs and that $\sigma$ is constant on each arc.
If $j = i+1$ the arcs are $\{v_{i+1}\}$ and $\{v_{i+2}, \dots, v_i\}$, so one class is a single vertex.
If $j = i+2$ the arcs are $\{v_{i+1}, v_{i+2}\}$ and $\{v_{i+3}, \dots, v_i\}$, so one class is a pair of consecutive vertices.
\end{proof}

The second lemma bounds, for a $4$-set, the sum of two non-edge thresholds by the sum of two edge thresholds.

\begin{lemma}
\label{lem:antihole-swap}
Let $n \geq 6$ be even and let $(\sigma, c, r)$ be a bicoloured representation of $\overline{C_n}$ with $c_i < c_{i-1}$ and $c_i < c_{i+1}$ for every even $i$, as in Lemma~\ref{lem:antihole-alt}\ref{it:alt-global}.
Let $i$ be even and $j$ odd, at cyclic distance at least $2$.
Then
\[ \tau_{i,i+1} + \tau_{j,j+1}  <  \tau_{i+1,j+1} + \tau_{i,j}. \]
\end{lemma}

\begin{proof}
The four vertices $v_i, v_{i+1}, v_j, v_{j+1}$ are distinct.
Since $v_iv_{i+1}$ and $v_jv_{j+1}$ are non-edges, and $c_i < c_{i+1}$ and $c_{j+1} < c_{j}$ by hypothesis, we obtain
\[ (c_{i+1} - c_i) + (c_j - c_{j+1})  >  \tau_{i,i+1} + \tau_{j,j+1}. \]
On the other hand, neither $v_iv_j$ nor $v_{i+1}v_{j+1}$ is a consecutive pair, so both are edges, which gives
\[ |c_{i+1} - c_{j+1}| + |c_j - c_i| \leq \tau_{i+1,j+1} + \tau_{i,j}. \]
Since $(c_{i+1} - c_i) + (c_j - c_{j+1}) = (c_{i+1} - c_{j+1}) + (c_j - c_i) \leq |c_{i+1} - c_{j+1}| + |c_j - c_i|$, the assertion follows.
\end{proof}

\begin{theorem}
\label{thm:antihole-bicoloured}
For every $n \geq 5$ we have $\overline{C_n} \notin \bicoloured$.
\end{theorem}

\begin{proof}
Suppose, for contradiction, that $(\sigma, c, r)$ is a bicoloured representation of $\overline{C_n}$.
By Theorem~\ref{thm:odd-antiholes}, $n$ is even, so $n \geq 6$.
By Lemma~\ref{lem:antihole-alt}\ref{it:alt-global} we may assume
\begin{equation}
\label{eq:alt}
 c_i < c_{i-1} \quad\text{and}\quad c_i < c_{i+1}
\quad\text{for every even } i.
\end{equation}
Note that rotating the labels by an even amount preserves~\eqref{eq:alt}.
While rotating the labels by an odd amount reverses the inequalities in~\eqref{eq:alt}, negating every centre restores them.
Hence, we may rotate the labels by any amount, negating the centres if necessary.
By Lemma~\ref{lem:antihole-colourings} one colour class $M$ is empty, a single vertex, or a pair of consecutive vertices.
By rotating the labels and swapping the two colours if necessary, we may assume that $M$ is the class coloured $-$, and that $M$ is $\emptyset$, $\{v_0\}$, or $\{v_0, v_1\}$.

We use the following elementary facts repeatedly.
For positive reals $\alpha, \beta, \lambda$,
\begin{align}
  \label{eq:maxfact}
  \max\{\alpha, \beta\} < \max\{\lambda, \beta\} &\quad\Longrightarrow\quad \lambda > \max\{\alpha, \beta\}, \\
  \label{eq:sumfact}
  \alpha + \max\{\beta, \lambda\} < \beta + \max\{\alpha, \lambda\} &\quad\Longrightarrow\quad \alpha < \min\{\beta, \lambda\}.
\end{align}

{\em Case 1: $M = \emptyset$.}
As $M$ is empty, the colouring is constant, so we may rotate again and assume in addition that $r_0 = \max_k \{ r_k \}$.
Now $\tau_{0,j} = \max\{r_0, r_j\} = r_0$ for each $j = 1, \dots, n-1$.

Since $n \geq 6$ is even, $v_{n-3}$ has odd index, and is at cyclic distance $3$ from $v_0$.
So Lemma~\ref{lem:antihole-swap} with $i = 0$ and $j = n-3$ gives $\tau_{0,1} + \tau_{n-3,n-2} < \tau_{1,n-2} + \tau_{0,n-3}$.
Here $\tau_{0,1}$ and $\tau_{0,n-3}$ both equal $r_0$, so these terms cancel, leaving $\tau_{n-3,n-2} < \tau_{1,n-2}$.
As $M = \emptyset$, every threshold is a maximum, so this reads $\max\{r_{n-3}, r_{n-2}\} < \max\{r_1, r_{n-2}\}$ and, by~\eqref{eq:maxfact}, we obtain $r_1 > \max\{r_{n-3}, r_{n-2}\}$.
In particular, $r_1 > r_{n-2}$.

Similarly, $v_2$ has even index and is at cyclic distance $3$ from $v_{n-1}$.
Lemma~\ref{lem:antihole-swap} with $i = 2$ and $j = n-1$ gives $\tau_{2,3} + \tau_{n-1,0} < \tau_{3,0} + \tau_{2,n-1}$.
Here $\tau_{n-1,0}$ and $\tau_{3,0}$ both equal $r_0$, so these terms cancel, leaving $\tau_{2,3} < \tau_{2,n-1}$, i.e., $\max\{r_2, r_3\} < \max\{r_2, r_{n-1}\}$.
By~\eqref{eq:maxfact}, we obtain $r_{n-1} > \max\{r_2, r_3\}$.
In particular, $r_{n-1} > r_2$.

Finally, $v_{n-2}$ has even index and is at cyclic distance $3$ from $v_1$, so Lemma~\ref{lem:antihole-swap} with $i = n-2$ and $j = 1$ gives
\[
  \max\{r_{n-2}, r_{n-1}\} + \max\{r_1, r_2\}
  <  \max\{r_{n-1}, r_2\} + \max\{r_{n-2}, r_1\}
  =  r_{n-1} + r_1.
\]
But the left-hand side is at least $r_{n-1} + r_1$, a contradiction.

{\em Case 2: $M = \{v_0\}$.}
So $\tau_{0,j} = r_0 + r_j$ for every $j \neq 0$, and every other $\tau$ is a maximum.
The argument here is similar to that of Case~1.
We apply Lemma~\ref{lem:antihole-swap} with $(i,j) = (0, n-3)$, $(2, n-1)$, and $(n-2, 1)$.
For each pair, $i$ is even, $j$ is odd, and the two are at cyclic distance $3$.

With $(i,j) = (0, n-3)$ we obtain $\tau_{0,1} + \tau_{n-3,n-2} < \tau_{1,n-2} + \tau_{0,n-3}$.
Here $\tau_{0,1}$ and $\tau_{0,n-3}$ are sums, so the terms $r_0$ cancel, leaving
\[  r_1 + \max\{r_{n-3}, r_{n-2}\}  <  r_{n-3} + \max\{r_1, r_{n-2}\} . \]
By~\eqref{eq:sumfact}, we obtain $r_1 < \min\{r_{n-3}, r_{n-2}\}$.
In particular, $r_1 < r_{n-2}$.

With $(i,j) = (2, n-1)$ we obtain $\tau_{2,3} + \tau_{n-1,0} < \tau_{3,0} + \tau_{2,n-1}$.
Here $\tau_{n-1,0}$ and $\tau_{3,0}$ are sums, so the terms $r_0$ cancel, leaving
\[  r_{n-1} + \max\{r_3, r_2\}  <  r_3 + \max\{r_{n-1}, r_2\} . \]
By~\eqref{eq:sumfact}, we obtain $r_{n-1} < \min\{r_3, r_2\}$.
In particular, $r_{n-1} < r_2$.

Finally, with $(i,j) = (n-2, 1)$ we obtain $\tau_{n-2,n-1} + \tau_{1,2} < \tau_{n-1,2} + \tau_{n-2,1}$.
None of $v_{n-2}, v_{n-1}, v_1, v_2$ is $v_0$, so every $\tau$ here is a maximum, yielding
\[
  \max\{r_{n-2}, r_{n-1}\} + \max\{r_1, r_2\}
  <  \max\{r_{n-1}, r_2\} + \max\{r_{n-2}, r_1\}
  =  r_2 + r_{n-2}.
\]
But the left-hand side is at least $r_{n-2} + r_2$, a contradiction.

{\em Case 3: $M = \{v_0, v_1\}$.}
Here $\tau_{i,j}$ is a sum precisely when exactly one of $v_i, v_j$ lies in $M$.
We apply Lemma~\ref{lem:antihole-swap} with $(i,j) = (2, n-1)$ and $(n-2, 1)$. 

With $(i,j) = (2, n-1)$ we have $r_{n-1} < r_2$ as in the argument for Case~2.
With $(i,j) = (n-2, 1)$ we obtain $\tau_{n-2,n-1} + \tau_{1,2} < \tau_{n-1,2} + \tau_{n-2,1}$.
The pairs $v_{n-2}v_{n-1}$ and $v_{n-1}v_2$ avoid $M$, while $v_1v_2$ and $v_{n-2}v_1$ each meet it once, so the terms $r_1$ cancel, giving
\[  r_2 + \max\{r_{n-2}, r_{n-1}\}  <  r_{n-2} + \max\{r_2, r_{n-1}\} . \]
By~\eqref{eq:sumfact}, we obtain $r_2 < \min\{r_{n-2}, r_{n-1}\}$.
In particular, $r_2 < r_{n-1}$, contradicting $r_{n-1} < r_2$.
\end{proof}

Proposition~\ref{prop:even-antiholes} and Theorem~\ref{thm:antihole-bicoloured} together separate $\bicoloured$ and $\sandwich$, and do so infinitely often.
Specifically, $\overline{C_{2m}} \in \sandwich \setminus \bicoloured$ for every $m \geq 3$, the smallest example being $\overline{C_6}$.

For the corresponding minimality, we require the following.

\begin{lemma}
\label{lem:copath}
For every $k \geq 1$ we have $\overline{P_k} \in \bicoloured$, and hence $\overline{P_k} \in \sandwich$.
\end{lemma}

\begin{proof}
Let $\overline{P_k}$ have vertex set $\{v_0, \dots, v_{k-1}\}$, so that its non-edges are the $k-1$ pairs $v_iv_{i+1}$, and $v_iv_j \in E(\overline{P_k})$ if and only if $|i - j| \geq 2$.

Colour every vertex $+$ and set
\[  r_i = 2i+2, \quad c_i = (-1)^{i+1}(i+2) \quad\text{ for each}\quad 0 \leq i \leq k-1. \]
Every radius is positive and the colouring is constant, so for $i < j$ the pair $v_iv_j$ is an edge of the represented graph if and only if $|c_i - c_j| \leq \max\{r_i, r_j\} = 2j+2$.

If $j - i$ is even then $c_i$ and $c_j$ have the same sign, so $|c_i - c_j| = j - i < 2j+2$, an edge.
As $j - i$ is even and positive, $j - i \geq 2$ and $v_iv_j \in E(\overline{P_k})$.

If $j - i$ is odd the signs differ and $|c_i - c_j| = i + j + 4$.
For $j = i+1$ this is $2j+3 > 2j+2$, a non-edge, and indeed $v_iv_j \notin E(\overline{P_k})$.
For $j - i \geq 3$ we have $i \leq j-3$, so $i + j + 4 \leq 2j+1 < 2j+2$, an edge, and $j - i \geq 2$ gives $v_iv_j \in E(\overline{P_k})$.
\end{proof}

Deleting any vertex from $\overline{C_n}$ leaves $\overline{P_{n-1}}$, so by Lemma~\ref{lem:copath} every proper induced subgraph of $\overline{C_n}$ lies in $\bicoloured$.
With Theorem~\ref{thm:antihole-bicoloured} it follows that, for $n \geq 5$, the antihole $\overline{C_n}$ is a minimal forbidden induced subgraph for $\bicoloured$, and by Theorem~\ref{thm:odd-antiholes} it is one for $\sandwich$ as well when $n$ is odd.

\subsection{Trees}
\label{sec:trees}

Recall that a {\em caterpillar} is a tree in which some path contains every vertex of degree at least two.
Let $T_2$ be the {\em spider} with three legs of length two, i.e., the tree consisting of three paths of length two joined at a common vertex.
Here we follow the notation and naming of Golumbic and Trenk~\cite{golumbic2004tolerance}.

We start by recalling the following classical fact for trees.

\begin{fact}[\cite{golumbic2004tolerance}*{Theorem~3.2}]
\label{fact:trees}
For a tree $T$ the following are equivalent:
\begin{enumerate}[label=\normalfont(\roman*)]
  \item\label{it:tree-T2} $T$ has no subtree isomorphic to $T_2$;
  \item\label{it:tree-cat} $T$ is a caterpillar;
  \item\label{it:tree-int} $T$ is an interval graph~\cite{lekkerkerker-boland};
  \item\label{it:tree-bd} $T$ is a bounded tolerance graph~\cite{golumbic-monma-trotter}.
\end{enumerate}
\end{fact}

In particular, $T_2 \notin \bounded$, although it is a tolerance graph~\cite{golumbic-monma-trotter}*{Theorems~6~and~7}.
By Fact~\ref{fact:trees}\ref{it:tree-int} and Proposition~\ref{prop:unit-in-B}, every caterpillar lies in $\bicoloured$, and hence in $\sandwich$.
This leaves the trees containing $T_2$ as a subtree, and a subtree of a tree is necessarily an induced subgraph.
We show that these lie outside $\sandwich$.
\begin{corollary}
\label{cor:spider}
The spider $T_2$ is not in $\sandwich$, and hence not in $\bicoloured$.
\end{corollary}

\begin{proof}
Write $z$ for the centre of $T_2$, let $y_0, y_1, y_2$ be its leaves, and let $x_i$ be the middle vertex of the leg containing $y_i$.
The leaves are pairwise non-adjacent, and for $\{i, j, k\} = \{0, 1, 2\}$ the path $y_i x_i z x_j y_j$ avoids $N[y_k] = \{x_k, y_k\}$.
So the leaves form an asteroidal triple, and Proposition~\ref{prop:AT-free} applies.
\end{proof}

Along with Fact~\ref{fact:trees} this settles trees completely.

\begin{corollary}
\label{cor:trees}
For a tree $T$ the following are equivalent:
\begin{enumerate}[label=\normalfont(\roman*)]
  \item\label{it:cat} $T$ is a caterpillar;
  \item\label{it:cat-B} $T \in \bicoloured$;
  \item\label{it:cat-C} $T \in \sandwich$.
\end{enumerate}
\end{corollary}

Note that every proper induced subgraph of $T_2$ is a disjoint union of caterpillars, hence an interval graph, and so lies in $\bicoloured$ by Proposition~\ref{prop:unit-in-B}.
With Corollary~\ref{cor:spider}, it follows that $T_2$ is a minimal forbidden induced subgraph for both $\sandwich$ and $\bicoloured$.

\subsection{Complete bipartite graphs}
\label{sec:bipartite}

Throughout, $K_{3,3}$ has parts $\{v_0, v_1, v_2\}$ and $\{v_3, v_4, v_5\}$.

\begin{proposition}
\label{prop:K33-not-C}
$K_{3,3} \notin \sandwich$.
\end{proposition}

\begin{proof}
Suppose, for contradiction, that $(c, r)$ is a sandwich representation of $K_{3,3}$.
By definition the three centres in each part are distinct, so we may relabel the vertices so that
\[  c_0 < c_1 < c_2 \quad\text{and}\quad c_3 < c_4 < c_5 . \]
The pairs $v_0v_1$ and $v_1v_2$ are non-edges, so Lemma~\ref{lem:additivity}\ref{it:add-far} gives $c_2 - c_0 > r_0 + r_2$.
The same argument in the other part gives $c_5 - c_3 > r_3 + r_5$.
Together, these imply
\[  (c_2 - c_0) + (c_5 - c_3)  >  r_0 + r_2 + r_3 + r_5. \]

On the other hand, $v_2v_3$ and $v_0v_5$ are edges, so $c_2 - c_3 \leq r_2 + r_3$ and $c_5 - c_0 \leq r_0 + r_5$, yielding
\[
  (c_2 - c_0) + (c_5 - c_3)  =  (c_2 - c_3) + (c_5 - c_0)
  \leq  r_0 + r_2 + r_3 + r_5 ,
\]
a contradiction.
\end{proof}

Removing a single edge destroys the obstruction.

\begin{proposition}
\label{prop:K33e-in-B}
$K_{3,3} - e \in \bicoloured$, and hence $K_{3,3} - e \in \sandwich$.
\end{proposition}

\begin{proof}
Suppose, without loss of generality, that the deleted edge is $v_0v_3$.
Colour $v_0, v_1, v_2$ with $+$ and $v_3, v_4, v_5$ with $-$, and set
\[
  r = (5,5,5,3,3,3), \quad\text{ and }\quad c = (12, 6, 0, 0, 4, 8).
\]
Two vertices in the same part are monochromatic, so their threshold is the radius common to that part.
Within the first part the centres are at pairwise distance at least $6 > 5$, and within the second at pairwise distance at least $4 > 3$, so all six such pairs are non-edges.
Two vertices in different parts are bichromatic, so their threshold is $5 + 3 = 8$.
The only such pair at distance more than $8$ is $v_0v_3$, at distance $12$, so the remaining eight pairs are edges.
\end{proof}

Deleting $v_0$ from the representation above leaves a representation of $K_{2,3}$, and every proper induced subgraph of $K_{3,3}$ is an induced subgraph of $K_{2,3}$.
Hence, we have the following.

\begin{corollary}
\label{cor:K33-minimal}
Every proper induced subgraph of $K_{3,3}$ lies in $\bicoloured$.
\end{corollary}

Along with Proposition~\ref{prop:K33-not-C}, we obtain that $K_{3,3}$ is a minimal forbidden induced subgraph for both $\sandwich$ and $\bicoloured$.

If $s, t \geq 3$ then $K_{s,t}$ contains $K_{3,3}$ as an induced subgraph, so Proposition~\ref{prop:K33-not-C} and Observation~\ref{obs:C-hereditary} imply $K_{s,t} \notin \sandwich$.
Complete bipartite graphs with a part of size at most two, on the other hand, all belong to $\bicoloured$.

\begin{proposition}
\label{prop:K2t}
$K_{2,t} \in \bicoloured$ for every $t \geq 1$.
\end{proposition}

\begin{proof}
For $t = 1$ this is $P_3$, a caterpillar, so Corollary~\ref{cor:trees} applies.
For $t \geq 2$, write the parts as $\{a_0, a_1\}$ and $\{b_0, \dots, b_{t-1}\}$, colour $a_0$ and $a_1$ with $+$ and every $b_i$ with $-$, and set
\[
  r_{a_0} = r_{a_1} = 2t-3, \quad c_{a_0} = 0, \quad c_{a_1} = 2t-2,
  \quad r_{b_i} = 1, \quad c_{b_i} = 2i .
\]
The pair $a_0a_1$ is monochromatic with threshold $2t-3$, and its centres are $2t-2$ apart, so it is a non-edge.
Each pair $b_ib_j$ is monochromatic with threshold $1$, and its centres are at least $2$ apart, so it is a non-edge.
Each pair $a_jb_i$ is bichromatic with threshold $(2t-3) + 1 = 2t-2$, and $|c_{a_0} - c_{b_i}| = 2i$ while $|c_{a_1} - c_{b_i}| = 2(t-1-i)$, both at most $2t-2$, so all $2t$ of these pairs are edges.
\end{proof}

For a part of size one, the graph $K_{1,t}$ is a star, hence a caterpillar, and lies in $\bicoloured$ by Corollary~\ref{cor:trees}.
This settles complete bipartite graphs.

\begin{corollary}
\label{cor:bipartite}
For $s, t \geq 1$ the following are equivalent:
\begin{enumerate}[label=\normalfont(\roman*)]
  \item\label{it:bip-min} $\min\{s, t\} \leq 2$;
  \item\label{it:bip-B} $K_{s,t} \in \bicoloured$;
  \item\label{it:bip-C} $K_{s,t} \in \sandwich$.
\end{enumerate}
\end{corollary}

\section{Placement in the hierarchy}
\label{sec:hierarchy}

In addition to the interval and tolerance classes of \S\ref{sec:standard}, we shall need the following.

\begin{definition}[$\perfect$]
\label{def:perfect}
A graph is {\em perfect} if in every induced subgraph the chromatic number equals the clique number.
We write $\perfect$ for the class of perfect graphs.
\end{definition}

\begin{definition}[$\comparability$, $\cocomparability$]
\label{def:comparability}
A graph is a {\em comparability graph} if its edges admit a transitive orientation, and a {\em co-comparability graph} if it is the complement of one.
We write $\comparability$ and $\cocomparability$ for the two classes.
\end{definition}

\begin{definition}[$\chordal$, $\weaklychordal$]
\label{def:chordal}
A graph is {\em chordal} if every cycle of length at least four has a chord, and {\em weakly chordal} if neither it nor its complement has an induced cycle of length at least five.
We write $\chordal$ and $\weaklychordal$ for the two classes.
\end{definition}

\begin{definition}[$\trapezoid$, $\permutation$]
\label{def:trapezoid}
A graph $G$ is a {\em trapezoid graph} if there are two distinct parallel lines $\ell_1, \ell_2$ such that each vertex $u \in V(G)$ can be assigned a trapezoid $T_u$, with one side on $\ell_1$ and the opposite side on $\ell_2$, so that $uw \in E(G) \iff T_u \cap T_w \neq \emptyset$.
It is a {\em permutation graph} if this can be done with every trapezoid degenerate, i.e., a segment joining a point of $\ell_1$ to a point of $\ell_2$.
We write $\trapezoid$ and $\permutation$ for the two classes.
\end{definition}

Comparability graphs and co-comparability graphs are perfect~\cite{golumbic-agtpg}.

We settle the position of $\bicoloured$ and $\sandwich$ relative to the classical classes and to each other.
Both lie strictly between $\unit$ and $\cocomparability$, and Figure~\ref{fig:hierarchy} collects the resulting picture.
We treat the two classes in turn, beginning with $\sandwich$.

\begin{figure}[p]
\centering
\includegraphics[height=0.75\textheight]{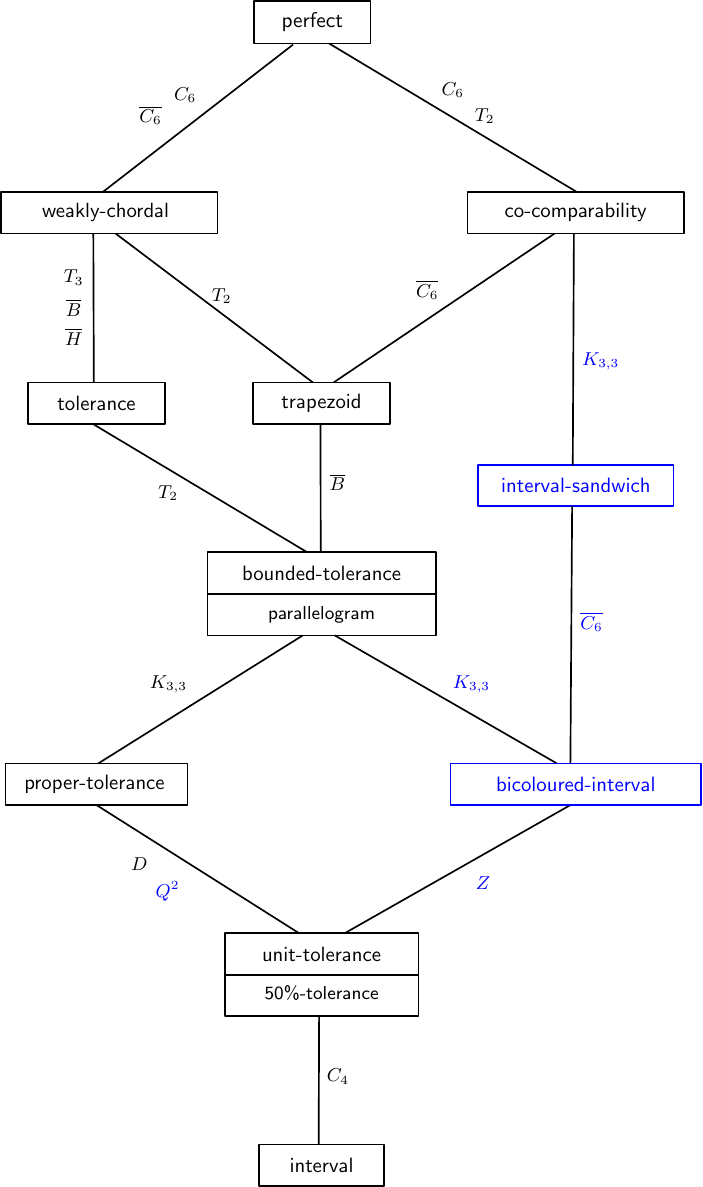}
\caption{The classes $\bicoloured$ and $\sandwich$ within the hierarchy of perfect graph classes.
An edge indicates that the lower class is contained in the upper, and is labelled by graphs separating the two.
Classes shown in blue are introduced in this paper, and the edges at them and the labels shown in blue are established here.
The others, and the layout, follow Figure~2.8 of Golumbic and Trenk~\cite{golumbic2004tolerance}.}
\label{fig:hierarchy}
\end{figure}

\subsection{\texorpdfstring{Placing $\sandwich$}{Placing interval-sandwich graphs}}
\label{sec:sandwich-placement}

The following theorem is the main containment result of this section.
Its proof is immediate from the fact that non-adjacency accumulates along the linear order of the centres.

\begin{theorem}
\label{thm:c-in-cocomparability}
$\sandwich \subsetneq \cocomparability$.
\end{theorem}

\begin{proof}
Let $(c, r)$ be a sandwich representation of a graph $G$.
We orient the non-edges of $G$ as follows.
By Observation~\ref{obs:representation}\ref{it:wa-nonedge} non-adjacent vertices have distinct centres, so we may orient a non-edge $u \to w$ when $c_u < c_w$.
If $u \to v$ and $v \to w$ then $c_u < c_v < c_w$, so Lemma~\ref{lem:additivity}\ref{it:add-far} makes $uw$ a non-edge with $c_u < c_w$, which is oriented $u \to w$.
The orientation is therefore transitive, so $\overline{G}$ is a comparability graph and $G$ is a co-comparability graph.

For strictness, let $s, t \geq 3$.
The complement of $K_{s,t}$ is a disjoint union of two cliques, each the comparability graph of a chain, and hence a comparability graph, so $K_{s,t} \in \cocomparability$.
But $K_{s,t} \notin \sandwich$ by Corollary~\ref{cor:bipartite}, completing the proof.
\end{proof}

As a consequence, we obtain the following.

\begin{corollary}
\label{cor:c-perfect}
$\sandwich \subseteq \perfect$.
\end{corollary}

Theorem~\ref{thm:c-in-cocomparability} gives a second proof of Theorem~\ref{thm:odd-antiholes}, that odd antiholes are excluded from $\sandwich$.
Indeed, $\overline{C_n}$ is a co-comparability graph if and only if $C_n$ is a comparability graph, which for $n \geq 4$ holds if and only if $C_n$ is bipartite (see, e.g.,~\cite{golumbic-agtpg}).

Next, we settle the position of $\sandwich$ relative to the remaining classical classes.

\begin{proposition}
\label{prop:C-position}
$\sandwich$ is incomparable with each of $\tol$, $\bounded$, $\trapezoid$, $\permutation$, $\chordal$, and $\weaklychordal$.
\end{proposition}

\begin{proof}
Each of $\tol$, $\bounded$, $\trapezoid$, $\permutation$, and $\chordal$ is contained in $\weaklychordal$ --- that $\chordal \subseteq \weaklychordal$ is standard, for $\tol$ and $\trapezoid$ this is~\cite{golumbic2004tolerance}*{Theorems~2.17 and~2.21}, and the rest follow from $\bounded \subseteq \tol$ and $\permutation \subseteq \trapezoid$.
The antihole $\overline{C_{2m}}$ for $m \geq 3$ is not weakly chordal, and hence lies in none of the six classes.
But by Proposition~\ref{prop:even-antiholes} the even antiholes lie in $\sandwich$.
It follows that $\sandwich$ is contained in none of the six classes.

For the converse we consider two cases.
For $\tol$, $\bounded$, $\trapezoid$, $\permutation$, and $\weaklychordal$, let $s, t \geq 3$ and consider $K_{s,t}$, which is bipartite and, as in the proof of Theorem~\ref{thm:c-in-cocomparability}, a co-comparability graph, hence a bounded tolerance graph, a trapezoid graph, and a permutation graph~\cite{golumbic2004tolerance}*{Theorem~3.9}.
Since $\bounded \subseteq \tol$ and trapezoid graphs are weakly chordal, $K_{s,t}$ lies in each of the five.
Yet $K_{s,t} \notin \sandwich$ by Corollary~\ref{cor:bipartite}.

For $\chordal$, trees are chordal, and by Corollary~\ref{cor:trees} a tree lies in $\sandwich$ only if it is a caterpillar.
So every non-caterpillar tree, the spider $T_2$ among them, is chordal but outside $\sandwich$.
Hence, none of the six classes is contained in $\sandwich$.
\end{proof}

We are left with $\proper$, which behaves in the same manner.
As $\proper \subseteq \tol \subseteq \weaklychordal$, the even antiholes $\overline{C_{2m}}$ for $m \geq 3$ again give $\sandwich \not\subseteq \proper$.
For the converse, $K_{3,3}$ is itself not a proper tolerance graph~\cite{golumbic2004tolerance}*{Example~2.30}, so more work is required.
Let $J$ be the graph on $\{v_0, \dots, v_{19}\}$ whose $78$ edges are given by the adjacency lists
\[
\begin{array}{ll}
\hline
  v_{0}\colon 1, 2, 3, 4, 5, 6, 7, 8 & v_{10}\colon 4, 5, 6, 7, 8, 9, 11, 12, 13, 14, 15, 16, 17, 18\\
  v_{1}\colon 0, 5, 9 & v_{11}\colon 4, 5, 7, 8, 9, 10, 13, 16, 19\\
  v_{2}\colon 0, 5, 9 & v_{12}\colon 4, 5, 9, 10, 13, 16, 19\\
  v_{3}\colon 0, 4, 5, 6, 7, 8, 9 & v_{13}\colon 5, 6, 7, 8, 9, 10, 11, 12, 16, 19\\
  v_{4}\colon 0, 3, 5, 6, 7, 8, 9, 10, 11, 12 & v_{14}\colon 5, 10, 16, 19\\
  v_{5}\colon 0, 1, 2, 3, 4, 6, 7, 8, 9, 10, 11, 12, 13, 14, 15 & v_{15}\colon 5, 10, 16, 19\\
  v_{6}\colon 0, 3, 4, 5, 9, 10, 13, 16 & v_{16}\colon 6, 7, 8, 9, 10, 11, 12, 13, 14, 15, 19\\
  v_{7}\colon 0, 3, 4, 5, 9, 10, 11, 13, 16 & v_{17}\colon 10, 19\\
  v_{8}\colon 0, 3, 4, 5, 9, 10, 11, 13, 16 & v_{18}\colon 10, 19\\
  v_{9}\colon 1, 2, 3, 4, 5, 6, 7, 8, 10, 11, 12, 13, 16 & v_{19}\colon 11, 12, 13, 14, 15, 16, 17, 18\\
\hline
\end{array}
\]
where $v_i \colon N$ lists the indices of the neighbours of $v_i$.

\begin{theorem}
\label{thm:proper-sandwich}
$J \in \proper \setminus \sandwich$.
\end{theorem}

The proof, an explicit proper tolerance representation of $J$ and a system of inequalities that no sandwich representation can satisfy, is deferred to Appendix~\ref{ap:proper-sandwich}.
With the discussion in the preceding paragraph, Theorem~\ref{thm:proper-sandwich} shows that $\proper$ and $\sandwich$ are incomparable.
Computations indicate that $J$ is minimal, so that every proper induced subgraph of $J$ lies in $\sandwich$, but we do not prove this.

Returning to the even antiholes $\overline{C_{2m}}$ for $m \geq 3$, we note that these are minimal in $\sandwich \setminus \tol$.
Indeed, $\overline{C_{2m}} - v = \overline{P_{2m-1}}$ is a permutation graph (paths are permutation graphs and permutation graphs are closed under complements) and hence lies in $\bounded \subseteq \tol$, as $\permutation \subseteq \bounded$~\cite{golumbic2004tolerance}*{Theorem~2.6}.

Finally, we remark that Theorem~\ref{thm:c-in-cocomparability} and Corollary~\ref{cor:B-in-C} place both $\sandwich$ and $\bicoloured$ inside $\cocomparability$, where each of the following four problems is solvable in polynomial time~\cite{golumbic-agtpg}.

\begin{corollary}
\label{cor:algorithmic}
Maximum clique, maximum independent set, minimum colouring, and minimum clique cover are solvable in polynomial time on $\sandwich$, and hence on $\bicoloured$.
\end{corollary}

\subsection{\texorpdfstring{Placing $\bicoloured$}{Placing bicoloured-interval graphs}}
\label{sec:bicoloured-placement}

Propositions~\ref{prop:B-parallelogram} and~\ref{prop:unit-in-B} place $\bicoloured$ between the unit tolerance graphs and the parallelogram graphs, and the rest of the picture is completed by the standard inclusions among the classical classes (see Figure~\ref{fig:hierarchy} and, e.g.,~\cite{golumbic2004tolerance}).

The three inclusions immediately around $\bicoloured$ are all strict.
That $\bicoloured \subsetneq \sandwich$ is already noted in \S\ref{sec:antiholes}, where the even antiholes, which lie in $\sandwich$, are shown to be minimal forbidden induced subgraphs for $\bicoloured$.

Next, for $s, t \geq 3$ the graph $K_{s,t}$ is bipartite and, as in the proof of Theorem~\ref{thm:c-in-cocomparability}, a co-comparability graph, hence a bounded tolerance graph (see~\cite{golumbic2004tolerance}*{Theorem~3.9}), but it is not in $\bicoloured$ by Corollary~\ref{cor:bipartite}.
This gives the following.

\begin{corollary}
\label{cor:B-strict-bd-tol}
$\bicoloured \subsetneq \bounded$.
\end{corollary}

Finally, we show $\unit \subsetneq \bicoloured$.
A unit tolerance representation has no interval properly containing another, so $\unit \subseteq \proper$, and it is enough to exhibit a graph in $\bicoloured \setminus \proper$.
We use the following lemma of Bogart, Fishburn, Isaak, and Langley.

\begin{lemma}[\cite{bogart-fishburn-isaak-langley}*{Lemma~4}]
\label{lem:proper-P4}
If $v_0v_1v_2v_3$ is an induced path, in this order, in a proper tolerance graph, then in every proper tolerance representation the centres $c_0$ and $c_1$ are both less than, or both greater than, each of the centres $c_2$ and $c_3$.
\end{lemma}

Let $Z$ be the {\em domino}, the graph on $\{v_0, \dots, v_5\}$ consisting of the six-cycle $v_0v_1v_2v_3v_4v_5$ together with the chord~$v_0v_3$.

\begin{proposition}
\label{prop:Z}
$Z \in \bicoloured \setminus \proper$.
\end{proposition}

\begin{proof}
To see $Z \in \bicoloured$, consider the following representation.
Colour $v_0, v_2, v_4$ with $-$ and $v_1, v_3, v_5$ with $+$, and set
\[  r = (1,1,1,1,1,1) \quad\text{ and }\quad c = (2, 0, 0, 2, 4, 4) . \]
Two vertices of the same parity are monochromatic, so their threshold is $1$.
Their centres are at distance at least $2$, so all six such pairs are non-edges, as required.
Two vertices of opposite parity are bichromatic, so their threshold is $2$.
Of the nine such pairs, only $v_1v_4$ and $v_2v_5$ have centres at distance more than $2$, namely $4$, and these are exactly the two opposite-parity non-edges.

We now show $Z \notin \proper$.
Suppose, for contradiction, that $Z$ has a proper tolerance representation.
We apply Lemma~\ref{lem:proper-P4} repeatedly to the following induced paths in $Z$:
\[  \pi_1 = v_2v_1v_0v_5, \quad \pi_2 = v_2v_3v_0v_5, \quad \pi_3 = v_1v_0v_3v_4, \quad\text{ and }\quad \pi_4 = v_1v_0v_5v_4 . \]
Considering $\pi_1$ and $\pi_2$, we see that $c_1, c_2$ and $c_2, c_3$ each lie on one side of $c_0, c_5$, and the same side in both, as both pairs contain $c_2$.
Reflecting the representation if necessary, we may assume they lie to the left, i.e.,
\[  c_1, c_2, c_3 < \min\{c_0, c_5\} , \]
and in particular $c_3 < c_0$, $c_1 < c_0$, and $c_1 < c_5$.

Considering $\pi_3$, $c_0, c_1$ are on the same side of $c_3, c_4$.
This cannot be the left side, since $c_3 < c_0$.
So $c_3, c_4$ lie to the left of $c_0, c_1$, and in particular $c_4 < c_1$.

Considering $\pi_4$, $c_0, c_1$ lie on the same side of $c_4, c_5$.
This cannot be the right side, since $c_1 < c_5$.
So $c_0, c_1$ lie to the left of $c_4, c_5$, and in particular $c_1 < c_4$.
But this contradicts $c_4 < c_1$ from $\pi_3$.
\end{proof}

In particular, $Z$ is a bipartite graph in $\bounded \setminus \proper$.

We now consider the remaining classes $\chordal$, $\permutation$, and $\proper$, showing that each is incomparable with $\bicoloured$.
For $\chordal$ this is immediate, since $C_4$ is not chordal but $C_4 \in \bicoloured$ by Example~\ref{ex:C4}, and the spider $T_2$ is chordal but $T_2 \notin \bicoloured$ by Corollary~\ref{cor:spider}.

Let $W = P_7^2$ be the square of the seven-vertex path, i.e., the graph on $\{v_0, \dots, v_6\}$ with $v_iv_j \in E(W)$ exactly when $|i - j| \leq 2$.
We remark that $W$ separates $\parallelogram$ from $\permutation$~\cite{golumbic2004tolerance}*{Figure~2.8}.
The following proposition shows that $W$ separates $\bicoloured$ from $\permutation$ as well.

\begin{proposition}
\label{prop:B-permutation}
$\bicoloured$ and $\permutation$ are incomparable.
\end{proposition}

\begin{proof}
That $\permutation \not\subseteq \bicoloured$ follows from Proposition~\ref{prop:C-position}, as $\bicoloured \subseteq \sandwich$.

For $\bicoloured \not\subseteq \permutation$, assigning $v_i$ the interval $[i-1, i+1]$ realises $W$ as a unit interval graph, so $W \in \interval \subseteq \unit \subseteq \bicoloured$ by Proposition~\ref{prop:unit-in-B}.
But $W$ is not a comparability graph (see~\cite{golumbic2004tolerance}*{Exercise~2.5}), while $\permutation \subseteq \comparability$, so $W \notin \permutation$.
\end{proof}

For $\proper$, incomparability is already settled, by Proposition~\ref{prop:Z} in one direction and by the graph $J$ of Theorem~\ref{thm:proper-sandwich} in the other, as $\bicoloured \subseteq \sandwich$.
We strengthen the second direction to an infinite family of minimal graphs, deferring the construction to \S\ref{sec:proper} and the proofs to Appendix~\ref{ap:Qm}.

\begin{theorem}
\label{thm:proper-incomparable}
For every $m \geq 2$ there is a graph $Q^m$ on $6m+1$ vertices that is minimal in $\proper \setminus \bicoloured$.
\end{theorem}

Theorem~\ref{thm:proper-incomparable} shows that $\proper \cap \bicoloured$ has no characterisation by finitely many forbidden induced subgraphs within $\proper$.

\subsection{\texorpdfstring{An infinite family in $\proper \setminus \bicoloured$}{An infinite family in proper-tolerance minus bicoloured-interval}}
\label{sec:proper}

We now construct the family of Theorem~\ref{thm:proper-incomparable} by adapting a family of graphs in $\proper \setminus \unit$ due to Bogart, Fishburn, Isaak, and Langley~\cite{bogart-fishburn-isaak-langley}.
For the rest of this subsection, let $m \geq 2$ be fixed.

\begin{definition}[The base graph $B^m$]
\label{def:GmBm}
Let $B^m$ be the graph with vertex set $\{1, \dots, 4m\} \cup \{x, z\}$ and edges
\[
  \{2i-1,  2i+1\} \ \text{ and }\ \{2i,  2i+1\}
  \quad \text{where } 1 \leq i \leq 2m-1,
\]
together with
\[
  \{x, j\} \quad \text{where } 2 \leq j \leq 4m
  \quad\text{and}\quad
  \{z, j\} \quad \text{where } 2m+1 \leq j \leq 4m.
\]
\end{definition}

So $B^m - \{x, z\}$, which we call the {\em chain}, is the caterpillar formed by the path $1, 3, \dots, 4m-1$ with vertex $2i$ adjacent to vertex $2i+1$ for $1 \leq i \leq 2m-1$, together with the isolated vertex $4m$.
The vertex~$x$ is adjacent to every chain vertex except $1$, and $z$ is adjacent to the chain vertices $2m+1, \dots, 4m$.
We refer to the vertices $3, 5, \dots, 4m-1$ as the {\em joints}.
We note that $B^m$ is the graph $H^m$ of~\cite{bogart-fishburn-isaak-langley} with one vertex deleted, so $B^m \in \unit$ by~\cite{bogart-fishburn-isaak-langley}*{Theorem~13}.

\begin{definition}[Ordering and cuts]
\label{def:cuts}
Let $\ord(m)$ be the ordering of $V(B^m)$ given by
\[ \ord(m) = (1, 2, \dots, 2m,\ x,\ 2m+1, \dots, 4m,\ z), \]
and let $\omega_p$ denote the $p^{\text{th}}$ element of $\ord(m)$.

For $1 \leq s \leq 4m+1$, a vertex $v \notin V(B^m)$ has {\em cut} $s$ if $N(v) \cap V(B^m) = \{\omega_1, \dots, \omega_s\}$, i.e.,
\[ N(v) \cap V(B^m) =
\begin{cases}
    \{1,\dots,s\} & \text{if }1 \leq s \leq 2m,\\
    \{1,\dots,2m\} \cup \{x\} \cup \{2m+1,\dots,s-1\} & \text{if }2m+1 \leq s \leq 4m+1.
\end{cases} \]
\end{definition}

Since $z = \omega_{4m+2}$ and $s \leq 4m+1$, no vertex with a cut is adjacent to $z$.

\begin{definition}[The graph $Q^m$]
\label{def:Qm}
Let $Q^m$ be obtained from $B^m$ by adding vertices $q_1, \dots, q_{2m-1}$ inducing $K_{2m-1}$ minus the edge $q_1q_{2m-1}$, where $q_i$ has cut
\[
s_i = \begin{cases}
    4m+1-2i & \text{if }1 \leq i \leq m-1,\\
    4m-2i   & \text{if } m \leq i \leq 2m-1.
\end{cases}
\]
We call $q_1$ the {\em inner} vertex and $q_2, \dots, q_{2m-1}$ the {\em outer} vertices.
\end{definition}

So $|V(Q^m)| = 6m+1$, and the cuts $s_1 > s_2 > \cdots > s_{2m-1}$ are the odd values $2m+3 \leq s \leq 4m-1$ and the even values $2 \leq s \leq 2m$.
Note that each cut ends at an even chain vertex.
In particular, the inner vertex $q_1$ has cut $4m-1$, so $N(q_1) \cap V(B^m) = \{1,\dots,4m-2\} \cup \{x\}$, and the outer vertex $q_{2m-1}$, the only one not adjacent to $q_1$, has cut $2$, so $N(q_{2m-1}) \cap V(B^m) = \{1, 2\}$.

We prove the following three propositions.

\begin{proposition}
\label{prop:Qm-proper}
$Q^m \in \proper$.
\end{proposition}

\begin{proposition}
\label{prop:Qm-not-bicoloured}
$Q^m \notin \bicoloured$.
\end{proposition}

\begin{proposition}
\label{prop:Qm-minimal}
$Q^m - u \in \bicoloured$ for every $u \in V(Q^m)$.
\end{proposition}

By Observation~\ref{obs:B-hereditary} and Proposition~\ref{prop:Qm-minimal}, every proper induced subgraph of $Q^m$ is in $\bicoloured$.
Hence, the three propositions together imply Theorem~\ref{thm:proper-incomparable}.
The proofs of the three propositions are in Appendix~\ref{ap:Qm}.
Since the appendix is long and technical, we give a brief sketch here.

Proposition~\ref{prop:Qm-proper} is proved by an explicit tolerance representation in which the left endpoints and the right endpoints induce the same ordering of the vertices, so the representation is proper.
We start with the proper representation of $B^m - z$ in~\cite{bogart-fishburn-isaak-langley}*{Lemma~9}.
The vertex $z$ is placed on the right, and the outer vertices $q_2, \dots, q_{2m-1}$ to the left of everything, as in~\cite{bogart-fishburn-isaak-langley}*{Theorem~7}.
For each $i$, the interval of $q_i$ overlaps that of $\omega_{s_i}$ by exactly the tolerance of $q_i$, so that $q_i$ is adjacent to precisely the first $s_i$ base vertices, and the outer intervals overlap one another by more than their tolerances, so the outer vertices form a clique.
The inner vertex $q_1$ instead begins inside the base, between its third and fourth left endpoints, so it is adjacent to every outer vertex whose cut is at least $4$, and not to the vertex with cut $2$, i.e., $q_{2m-1}$.
This is where our construction diverges from that of~\cite{bogart-fishburn-isaak-langley}, whose added vertices on the left all lie outside the base, serve only to pin its order, and form a clique.
Placing $q_1$ inside is what makes the non-edge $q_1q_{2m-1}$ possible, and that non-edge is precisely what Proposition~\ref{prop:Qm-not-bicoloured} exploits.

The proof of Proposition~\ref{prop:Qm-not-bicoloured} proceeds by contradiction.
Suppose $Q^m$ has a bicoloured representation.
The first step forces the colouring.
The vertices $x$ and $z$ are non-adjacent with three pairwise non-adjacent common neighbours so, by Corollary~\ref{cor:forcing}\ref{it:forcing-col}, they have the same colour.
We then show that $q_1$ has the same colour as well, using Lemma~\ref{lem:fan} twice.
Each outer vertex $q_i$ then fixes the colour of one joint via its cut, by Lemma~\ref{lem:twins}, and $q_1$ fixes the remaining top joint $4m-1$ by Lemma~\ref{lem:fan}.
This is the only reason we need the outer vertices, and it gives Proposition~\ref{prop:joints}, that $q_1$, $x$, $z$, and all $2m-1$ joints share a colour.

The second step orders the joints.
The path $1, 3, 5, \dots, 4m-1, 4m-2$ has a monochromatic interior.
By Corollary~\ref{cor:noextremum} the centres of the joints are strictly monotone, say $c_3 < \cdots < c_{4m-1}$. Moreover, by Lemma~\ref{lem:P4gen} each even vertex $2i$ with $2 \leq i \leq 2m-1$ lies to the right of the joint $2i-1$.

The third step is the contradiction, a bicoloured counterpart of~\cite{bogart-fishburn-isaak-langley}*{Lemma~8} that propagates a separation between consecutive joints, given at the top of the chain by $q_1$, down to the bottom (Proposition~\ref{prop:cl8}).
As a result, the gaps between consecutive joints increase along the chain.
The interval of $x$, which must reach beyond $I_{4m-1}$ from a centre that $z$ pins below $c_{2m+1}$, is then wide enough to reach $1$ as well, contradicting that $x$ and $1$ are non-adjacent.
All of this takes place inside $B^m + q_1$, the base graph with the inner vertex $q_1$ with cut $4m-1$ adjoined, so that subgraph already admits no representation with the forced colouring (Proposition~\ref{prop:impossible}).

For Proposition~\ref{prop:Qm-minimal}, the deleted vertex $u$ falls into three cases.
In the first two cases, deleting one of $x, z$, or a chain vertex leaves a graph with a monochromatic, i.e., a $50\%$-tolerance, representation.
The chain here is represented by centres at prescribed gaps, which we take to increase along the chain, mirroring Proposition~\ref{prop:cl8}.
Then, if needed, $x$ and $z$ are adjoined.
The vertex~$z$ goes to the right of everything, so is easily placed.
The interval of $x$ contains every chain centre but $c_1$, so $c_x$ lies above the midpoint of $c_1$ and the largest chain centre.
Moreover, if $z$ is present, $x$ is placed in the gap just above $2m$ to accommodate the adjacencies of $z$.
This is possible only when that midpoint lies below the smallest centre above $2m$, which is precisely the condition that the {\em budget} in Appendix~\ref{ap:Qm} is positive (see~\eqref{eq:budget}).
Finally, $q_1$ is placed, which is possible when the centre order respects the cuts and a second, budget-like inequality holds.
The outer vertices are then attached as a clique on the left.

The case when a chain vertex is deleted requires particular care.
Both $x$ and $z$ survive, so (unless $u = 1$) the budget must be positive, which fails for the gaps used when $x$ or $z$ is deleted.
To make it positive, three boundary cases aside, we shorten the gaps above the deleted vertex, lowering the largest chain centre.
This is possible only because the deleted vertex interrupts the propagation of Proposition~\ref{prop:cl8}, which otherwise forces the gaps to increase.

When a vertex $q_i$ is deleted, we use a different approach.
For $i \geq 2$ the subgraph $B^m + q_1$ survives, so by Proposition~\ref{prop:impossible} no monochromatic representation exists and we must use the second colour.
The case $i = 1$ can be handled in the same manner.
It turns out that one vertex of the second colour suffices, namely the joint $y$ just above the cut $s_i$.
We obtain a monochromatic representation of $Q^m - \{q_i, y\}$ as above, and add an interval for $y$ that meets exactly the intervals of its neighbours, which gives a bicoloured representation with $y$ in the second colour.

\section{Recognition}
\label{sec:recognition}

The recognition problem for a class of graphs asks whether a given graph $G$ belongs to the class.
Hayward and Shamir~\cite{hayward-shamir} showed that the recognition problems for $\tol$ and $\bounded$ are in $\NP$, and Mertzios, Sau, and Zaks~\cite{mertzios-sau-zaks} proved that they are in fact $\NP$-complete, even when the input is promised to be a trapezoid graph.
The recognition problems for $\unit$ and $\proper$ remain open (see~\cite{golumbic2004tolerance}*{Chapter~14} and~\cite{mertzios-sau-zaks}*{Section~5}).
By contrast, $\trapezoid$, which contains $\bounded$, is recognisable in polynomial time~\cite{ma-spinrad}, as is $\cocomparability$~\cite{golumbic-agtpg}, which contains $\sandwich$ by Theorem~\ref{thm:c-in-cocomparability}.

In this section we show that the recognition problems for $\bicoloured$ and $\sandwich$ are in $\NP$.
Our argument follows the outline of Hayward and Shamir~\cite{hayward-shamir}:
\begin{itemize}
  \item whether $(c, r)$ is a sandwich representation of $G$ depends only on a single linear order, with ties, on the interval endpoints and centres, and in the bicoloured case only on that order and the colouring,
  \item this order can be imposed by a system of linear constraints, and every solution of the system is again a representation of $G$, and
  \item the system is feasible and has a solution giving a polynomial-size integer representation of $G$.
  \end{itemize}

Fix a graph $G$ with vertex set $V = V(G)$, where $|V| = n$.
For $c \in \RR^V$ and $r \in \RR^V$ with $r_u > 0$ for every $u \in V$, write $\mathcal{P}(c,r)$ for the family of $3n$ reals
\[  c_u - r_u, \quad c_u, \quad c_u + r_u \quad\text{ with }\quad u \in V, \]
each labelled by its vertex and its role as the left endpoint, centre, or right endpoint of $I_u$.
A {\em linear order} of the labels is an arrangement of them in a sequence, together with a relation, $=$ or $<$, between each label and the next.
A linear order is {\em induced} by $(c,r)$ if these relations are those of the members of $\mathcal{P}(c,r)$.
Two pairs $(c, r)$ and $(c', r')$ are {\em order-equivalent} if they induce a common linear order.

\begin{lemma}
\label{lem:order-equivalence}
Let $(c, r)$ and $(c', r')$ be order-equivalent.
Then $(c, r)$ is a sandwich representation of $G$ if and only if $(c', r')$ is.
Moreover, for every colouring $\sigma$, the triple $(\sigma, c, r)$ is a bicoloured representation of $G$ if and only if $(\sigma, c', r')$ is.
\end{lemma}

\begin{proof}
Fix $u \neq w$.
Unfolding the absolute value, $|c_u - c_w| \leq r_u + r_w$ if and only if
\[  c_u - r_u \leq c_w + r_w \quad\text{ and }\quad c_w - r_w \leq c_u + r_u . \]
Similarly, $|c_u - c_w| \leq \max\{r_u, r_w\}$ if and only if $|c_u - c_w| \leq r_u$ or $|c_u - c_w| \leq r_w$, which is equivalent to
\[  c_u - r_u \leq c_w \leq c_u + r_u \quad\text{ or }\quad c_w - r_w \leq c_u \leq c_w + r_w . \]
Thus each of the two threshold conditions is a Boolean combination of comparisons between members of $\mathcal{P}(c,r)$.
Each condition therefore holds for $(c,r)$ if and only if it holds for $(c', r')$.

For $\sandwich$, Definition~\ref{def:sandwich} requires that the first condition hold when $uw \in E(G)$ and that the second fail when $uw \notin E(G)$.
For $\bicoloured$, Definition~\ref{def:bicoloured} requires that $uw \in E(G)$ if and only if one of the two conditions holds, depending on $\sigma(u)$ and $\sigma(w)$.
Since both requirements involve only the two conditions and $u, w$ are arbitrary, the lemma follows.
\end{proof}

By Lemma~\ref{lem:order-equivalence}, a representation may be replaced by any order-equivalent one.
We use this to prove the existence of a polynomial-size integer representation.

\begin{theorem}
\label{thm:integer-representation}
Let $G$ be a graph on $n$ vertices, and suppose that $G \in \sandwich$ or $G \in \bicoloured$.
Then $G$ has a sandwich or a bicoloured representation, respectively, in which every $c_u$ and every $r_u$ is a non-negative integer less than $n \, 2^{4n}$.
\end{theorem}

\begin{proof}
Let $(c,r)$ be a sandwich representation of $G$, or the interval part of a bicoloured representation $(\sigma, c, r)$ of $G$.
By translating if necessary, we may assume that the smallest member of $\mathcal{P}(c,r)$ is $0$.
Fix a linear order of the labels induced by $(c, r)$.
Note that, for each vertex, the centre's label lies strictly between the two endpoint labels.
In particular, the two endpoint labels of a vertex are not consecutive.

Consider $2n$ real unknowns $x_u, y_u$ for $u \in V$, and the $3n$ linear expressions
\[  x_u - y_u, \quad x_u, \quad x_u + y_u \quad (u \in V), \]
whose values at $(x,y) = (c,r)$ are the members of $\mathcal{P}(c,r)$.
Each expression inherits the label of its value.
Let $S$ be the system of constraints defined as follows.
The first expression in the ordering equals~$0$.
For each pair of consecutive labels, the two expressions are equal when the relation between the labels is $=$, and the later exceeds the earlier by at least $1$ when the relation is $<$.

The system $S$ is feasible.
Indeed, $(\lambda c, \lambda r)$ satisfies it for every large enough $\lambda > 0$, where scaling fixes the first expression at $0$, preserves the ties, and multiplies every strict gap by $\lambda$.
Now let $(x,y)$ be any solution of $S$.
Along the ordering the expressions start at $0$ and never decrease, so all of them are non-negative.
Each $y_u$ is at least $1$, since a strict relation separates the labels of $x_u - y_u$ and $x_u$.
Moreover, $(x,y)$ induces the fixed linear order, as $S$ enforces the relation between each pair of consecutive expressions.
By Lemma~\ref{lem:order-equivalence}, the pair $(x,y)$ is therefore a sandwich representation of $G$, or the triple $(\sigma, x, y)$ a bicoloured one, in the respective cases.

Adding a slack unknown for each inequality and discarding redundant equations, we rewrite $S$ as the matrix equation $Az = b$ with $z \geq 0$.
Here $A$ has $m \leq 3n$ linearly independent rows and $b \in \{0, 1\}^m$.
Each row involves the expressions of at most two vertices and at most one slack, so it has at most five non-zero entries.
The entries lie in $\{0, \pm 1\}$, since a coefficient of $\pm 2$ could only come from comparing the two endpoint expressions of one vertex, but those labels are never consecutive.

Since every solution of $S$ is non-negative, this system is feasible, and so has a basic feasible solution $z^*$.
The non-zero entries of $z^*$ are among the entries of $B^{-1} b$ for some non-singular $m \times m$ column submatrix $B$ of $A$~\cite{schrijver}.
Set $z' = |\det B| \, z^*$.
By Cramer's rule, the entries of $\det(B) \, B^{-1}$ are, up to sign, minors of $B$ of order $m - 1$.
So $z'$ has integer entries, and they are non-negative as $z^* \geq 0$.
Moreover, each minor is at most $5^{(m-1)/2}$ in absolute value by Hadamard's inequality, since the rows of $B$ have at most five non-zero entries, each of which is $\pm 1$.
Now $b \in \{0, 1\}^m$, so every entry of $z'$ is, up to sign, a sum of at most $m$ entries of $\det(B) \, B^{-1}$.
It follows that every entry of $z'$ is at most $m \, 5^{(m-1)/2} < n \, 2^{4n}$, where the second inequality holds since $m \leq 3n$.

Finally, $z'$ is a solution of $Az = |\det B| \, b$.
Scaling $b$ by $|\det B| \geq 1$ keeps the first expression at $0$ and the ties, and every strict gap at least $1$.
So the coordinates $x_u, y_u$ of $z'$ satisfy $S$, as required.
\end{proof}

\begin{corollary}
\label{cor:NP}
The recognition problems for $\bicoloured$ and for $\sandwich$ are in $\NP$.
\end{corollary}

\begin{proof}
For membership in $\NP$ it suffices to exhibit, for every graph in either class, a certificate with polynomially many bits that can be checked in polynomial time.

Suppose first that $G \in \sandwich$, and take the representation $(c,r)$ of Theorem~\ref{thm:integer-representation}.
The certificate $(c,r)$ consists of $2n$ non-negative integers with $O(n)$ bits each.
To verify it, we check that $r_u > 0$ for every $u$, and that the two implications of Definition~\ref{def:sandwich} hold for each of the $\binom{n}{2}$ pairs.
Each pair costs a constant number of additions and comparisons of such integers.

Suppose next that $G \in \bicoloured$, and take the representation $(\sigma, c, r)$ of Theorem~\ref{thm:integer-representation}.
Set $\rho_u = \sigma(u) r_u$.
By Proposition~\ref{prop:signed-radius}, the pair $(c, \rho)$ is again a certificate, of $2n$ integers with $O(n)$ bits each.
To verify it, we check that $\rho_u \neq 0$ for every $u$, and that $uw \in E(G)$ if and only if $|c_u - c_w| \leq \operatorname{diam}\{0, \rho_u, \rho_w\}$ for each pair $u \neq w$.
Again each pair costs a constant number of operations, as required.
\end{proof}

The argument above relies on the linearity of the constraints once the order is fixed.
When the constraints are polynomials of higher degree, the natural upper bound is the existential theory of the reals, a class between $\NP$ and $\PSPACE$~\cite{canny}.
For intersection graphs of segments and of disks, recognition is in fact complete for this class~\cites{kratochvil-matousek, schaefer, mcdiarmid-muller}.
There, integer representations may require exponentially many bits~\cite{mcdiarmid-muller}, so the analogue of Theorem~\ref{thm:integer-representation} fails.
Slab hypergraphs are also defined by higher-degree constraints, since membership of a triple in $H(P)$ is a quadratic condition on the coordinates by Lemma~\ref{lem:residual}.
So we do not know whether recognising slab hypergraphs lies in $\NP$.

\section{Conclusion and open problems}
\label{sec:open}

Our starting point was the links of slab hypergraphs, which Proposition~\ref{prop:B-is-links} identifies with the class $\bicoloured$.
We have placed $\bicoloured$ and its relaxation $\sandwich$ in the tolerance hierarchy and shown that their recognition problems lie in $\NP$.
We close with some open questions.

\paragraph{Forbidden induced subgraphs}
Corollary~\ref{cor:no-finite} can be sharpened for $\bicoloured$.
By \S\ref{sec:antiholes} and Proposition~\ref{prop:even-antiholes}, the even antiholes $\overline{C_{2m}}$ with $m \geq 3$ are minimal forbidden induced subgraphs for $\bicoloured$ lying in $\sandwich$, and any set characterising $\bicoloured$ within $\sandwich$ by forbidden induced subgraphs must contain each of them.

\begin{corollary}
\label{cor:no-finite-cocomp}
$\bicoloured$ is not characterised by finitely many forbidden induced subgraphs within $\sandwich$, nor within $\cocomparability$.
\end{corollary}

The second statement follows from the first, since $\sandwich \subseteq \cocomparability$ by Theorem~\ref{thm:c-in-cocomparability}.

For $\sandwich$ no such sharpening is available.
Of its minimal forbidden induced subgraphs from \S\ref{sec:structural}, only $K_{3,3}$ is a co-comparability graph (see~\cite{golumbic2004tolerance}*{Example~2.7} and~\cite{golumbic-agtpg}).
The graph $J$ of Theorem~\ref{thm:proper-sandwich} contains a second one.
Indeed, $J$ lies in $\proper \subseteq \cocomparability$~\cite{golumbic-monma-trotter}*{Theorems~1 and~10} but not in $\sandwich$, so it contains a minimal forbidden induced subgraph $F$ for $\sandwich$, which is again a co-comparability graph.
Now $F$ is not $K_{3,3}$, since $K_{3,3}$ is not a proper tolerance graph~\cite{golumbic2004tolerance}*{Example~2.30} while $J$ is.
As noted after Theorem~\ref{thm:proper-sandwich}, computations indicate that $J$ is itself minimal, and we are not aware of any other obstructions inside $\cocomparability$.

\begin{question}
\label{q:sandwich-finite}
Is $\sandwich$ characterised by finitely many forbidden induced subgraphs within $\cocomparability$, or within $\proper$?
\end{question}

\paragraph{The hierarchy}
The results of \S\ref{sec:classes} and \S\ref{sec:hierarchy} give the chain
\begin{equation}
\label{eq:chain}
  \bicoloured \subseteq \sandwich \cap \bounded \subseteq \sandwich \cap \tol ,
\end{equation}
the first inclusion by Corollaries~\ref{cor:B-in-C} and~\ref{cor:B-strict-bd-tol}, and the second since $\bounded \subseteq \tol$.

The first inclusion in~\eqref{eq:chain} is strict.
An exhaustive computational search found that the smallest graphs in $(\sandwich \cap \bounded) \setminus \bicoloured$ have eight vertices, one being the graph on $\{v_0, \dots, v_7\}$ with edges $v_0v_4$, $v_0v_5$, $v_0v_6$, $v_0v_7$, $v_1v_4$, $v_1v_6$, $v_2v_5$, $v_2v_6$, $v_2v_7$, $v_3v_5$, $v_3v_7$, $v_4v_7$, and~$v_6v_7$, so $\bicoloured$ is not determined by its two upper bounds.
We believe the family $Q^m$ of Theorem~\ref{thm:proper-incomparable} gives infinitely many such examples.

\begin{question}
\label{q:Qm-sandwich}
Is $Q^m \in \sandwich$ for every $m \geq 2$?
\end{question}

Computations indicate that $Q^m \in \sandwich$ for $m \leq 6$, but we have not attempted a proof.
By Theorem~\ref{thm:proper-incomparable} and $\proper \subseteq \bounded$~\cite{golumbic-monma-trotter}*{Theorem~10}, an affirmative answer would make each $Q^m$ minimal in $(\sandwich \cap \bounded) \setminus \bicoloured$.
This would imply that $\bicoloured$ has no characterisation by finitely many forbidden induced subgraphs within $\sandwich \cap \bounded$.

Whether the second inclusion in~\eqref{eq:chain} is strict is open.
Equality is exactly Conjecture~\ref{conj:GM-restricted} below, a special case of the following conjecture of Golumbic, Monma, and Trotter.

\begin{conjecture}[\cite{golumbic-monma-trotter}; see also~\cite{golumbic2004tolerance}*{Question~3.1}]
\label{conj:GM}
$\bounded = \tol \cap \cocomparability$.
\end{conjecture}

Conjecture~\ref{conj:GM} is known to hold for complements of trees~\cite{golumbic2004tolerance}*{\S3.1} and, more generally, for complements of triangle-free graphs~\cite{mertzios-zaks}.
Since $\sandwich \subseteq \cocomparability$ by Theorem~\ref{thm:c-in-cocomparability}, Conjecture~\ref{conj:GM} would imply the following.

\begin{conjecture}
\label{conj:GM-restricted}
$\sandwich \cap \tol \subseteq \bounded$.
\end{conjecture}

The standard example of a tolerance graph that is not bounded, the spider $T_2$, is not a counterexample, as it lies outside $\sandwich$ by Corollary~\ref{cor:spider}.

\paragraph{Recognition}
Corollary~\ref{cor:NP} places the two recognition problems in $\NP$, and their exact complexity is open.

\begin{question}
\label{q:recognition}
Are the recognition problems for $\bicoloured$ and for $\sandwich$ $\NP$-complete, or solvable in polynomial time?
\end{question}

By the proof of Theorem~\ref{thm:integer-representation}, recognition is polynomial once the order of the labels is prescribed, together with a colouring in the case of $\bicoloured$, so the difficulty lies in the choice of the order.
Prescribing only the colouring does not obviously help, since under a constant colouring the problem becomes the recognition of $\halftolerance = \unit$, which is open.
The order itself is constrained, since by the proof of Theorem~\ref{thm:c-in-cocomparability} the centres of a sandwich representation transitively orient the complement of the graph, but we do not know whether this can be exploited.
An easier question is whether the $O(n)$ bits per entry of Theorem~\ref{thm:integer-representation} can be reduced to $O(\log n)$, as for interval graphs.

\paragraph{Slab hypergraphs}
We end where we began.
Proposition~\ref{prop:B-is-links} describes the links of slab hypergraphs exactly, but the hypergraphs themselves are more complicated objects.
As noted in \S\ref{sec:recognition}, the natural upper bound on their recognition is the existential theory of the reals.

\begin{question}
\label{q:slab-NP}
Is the recognition problem for slab hypergraphs in $\NP$?
\end{question}

By Lemma~\ref{lem:residual}, three rational points can be tested for containment in a common slab in polynomial time.
So it would suffice that every slab hypergraph has a realisation by points with polynomially many bits.
The hypergraphs are the subject of a companion paper~\cite{companion}.

\section*{Statement on the use of AI}
The authors used Claude (Anthropic) to discuss proof strategies, check proofs, and edit the exposition.
The graph $J$ of Theorem~\ref{thm:proper-sandwich} and the proof in Appendix~\ref{ap:proper-sandwich} were obtained with AI assistance, as were parts of the analysis in Appendix~\ref{ap:Qm}.
All other results and proofs are the authors' own.
The authors have checked every proof and take full responsibility for all content.

\bibliography{graph-classes}

\appendix

\section{\texorpdfstring{A graph in $\proper \setminus \sandwich$}{A graph in proper-tolerance minus interval-sandwich}}
\label{ap:proper-sandwich}

We prove Theorem~\ref{thm:proper-sandwich}, beginning with the membership $J \in \proper$.
For convenience, we reproduce the adjacency lists of $J$.
\[
\begin{array}{ll}
\hline
  v_{0}\colon 1, 2, 3, 4, 5, 6, 7, 8 & v_{10}\colon 4, 5, 6, 7, 8, 9, 11, 12, 13, 14, 15, 16, 17, 18\\
  v_{1}\colon 0, 5, 9 & v_{11}\colon 4, 5, 7, 8, 9, 10, 13, 16, 19\\
  v_{2}\colon 0, 5, 9 & v_{12}\colon 4, 5, 9, 10, 13, 16, 19\\
  v_{3}\colon 0, 4, 5, 6, 7, 8, 9 & v_{13}\colon 5, 6, 7, 8, 9, 10, 11, 12, 16, 19\\
  v_{4}\colon 0, 3, 5, 6, 7, 8, 9, 10, 11, 12 & v_{14}\colon 5, 10, 16, 19\\
  v_{5}\colon 0, 1, 2, 3, 4, 6, 7, 8, 9, 10, 11, 12, 13, 14, 15 & v_{15}\colon 5, 10, 16, 19\\
  v_{6}\colon 0, 3, 4, 5, 9, 10, 13, 16 & v_{16}\colon 6, 7, 8, 9, 10, 11, 12, 13, 14, 15, 19\\
  v_{7}\colon 0, 3, 4, 5, 9, 10, 11, 13, 16 & v_{17}\colon 10, 19\\
  v_{8}\colon 0, 3, 4, 5, 9, 10, 11, 13, 16 & v_{18}\colon 10, 19\\
  v_{9}\colon 1, 2, 3, 4, 5, 6, 7, 8, 10, 11, 12, 13, 16 & v_{19}\colon 11, 12, 13, 14, 15, 16, 17, 18\\
\hline
\end{array}
\]

Assign to each vertex $v_i$ of $J$ the interval $I_i = [L_i, R_i]$ and tolerance $t_i$ given by
\[
\begin{array}{l*{10}{r}}
\hline
i & 0 & 1 & 2 & 3 & 4 & 5 & 6 & 7 & 8 & 9\\ \hline
L_i & 0 & 1 & 2 & 3 & 4 & 5 & 6 & 7 & 8 & 9\\
R_i & 10 & 33 & 35 & 42 & 84 & 85 & 97 & 98 & 99 & 100\\
t_i & 2 & 32 & 33 & 34 & 35 & 1 & 91 & 93 & 92 & 24\\
\hline
\end{array}
\]
\[
\begin{array}{l*{10}{r}}
\hline
i & 10 & 11 & 12 & 13 & 14 & 15 & 16 & 17 & 18 & 19\\ \hline
L_i & 40 & 41 & 49 & 62 & 83 & 84 & 85 & 97 & 98 & 99\\
R_i & 101 & 103 & 104 & 105 & 106 & 107 & 108 & 109 & 110 & 111\\
t_i & 3 & 57 & 55 & 23 & 94 & 54 & 12 & 93 & 56 & 4\\
\hline
\end{array}
\]
Both rows $L$ and $R$ are strictly increasing, so no interval contains another and the representation is proper.
For $i < j$, by monotonicity, $|I_i \cap I_j| = \max\{0, R_i - L_j\}$, so, as all tolerances are positive, Definition~\ref{def:tolerance} gives
\[   v_iv_j \in E(J) \iff R_i - L_j \geq \min\{t_i, t_j\}, \]
a single integer comparison.
An exhaustive computational check over the $\binom{20}{2}$ pairs confirms that this is a tolerance representation of $J$, so $J \in \proper$.

We turn to the non-membership.
Suppose $(c, r)$ is a sandwich representation of $J$.
Note that the pairs $\{v_1, v_2\}$, $\{v_7, v_8\}$, $\{v_{14}, v_{15}\}$, and $\{v_{17}, v_{18}\}$ are non-adjacent with equal neighbourhoods, so transposing the members of any one of them is an automorphism of $J$.
By Observation~\ref{obs:representation}\ref{it:wa-nonedge}, the centres of each of these pairs, and of the non-adjacent pair $v_0, v_9$, are distinct, so by negating $c$ and then transposing pairs as needed, we may assume
\begin{equation}
\label{eq:J-pins}
  c_0 < c_9, \quad c_1 < c_2, \quad c_7 < c_8, \quad c_{14} < c_{15}, \quad c_{17} < c_{18} .
\end{equation}

We apply Corollary~\ref{cor:propagate} repeatedly.
Each application takes a vertex $v_u$, a connected set $S$ of non-neighbours of $v_u$, and one vertex of $S$ whose centre is already known to lie on a given side of $c_u$, and places the centres of all of $S$ on that side.

First, $v_0$ is adjacent to exactly $v_1, \dots, v_8$, and $v_9, \dots, v_{19}$ induce a connected subgraph, so $c_0 < c_9$ gives
\begin{equation}
\label{eq:J-fan}
  c_0 < c_j \quad \text{ for each }\quad 9 \leq j \leq 19 .
\end{equation}
Next, for each $u \in \{10, 11, 12, 13, 14, 16, 17, 19\}$ let $S_u$ be the component of $J - N[v_u]$ containing $v_0$, namely
\[
\begin{array}{ll}
  S_{10} = \{v_0, \dots, v_3\}, &
  S_{11} = \{v_0, \dots, v_3, v_6\},\\
  S_{12} = \{v_0, \dots, v_3, v_6, v_7, v_8, v_{11}\}, &
  S_{13} = \{v_0, \dots, v_4\},\\
  S_{14} = \{v_0, \dots, v_4, v_6, \dots, v_9, v_{11}, v_{12}, v_{13}\}, &
  S_{16} = \{v_0, \dots, v_5\},\\
  S_{17} = \{v_0, \dots, v_9, v_{11}, \dots, v_{16}\}, &
  S_{19} = \{v_0, \dots, v_{10}\} .
\end{array}
\]
Since $c_0 < c_u$ by~\eqref{eq:J-fan}, we obtain $c_j < c_u$ for every $v_j \in S_u$.
In particular, $S_{10}$ and $S_{11}$ give $c_2 < c_{10}$ and $c_6 < c_{11}$.
Applying the corollary to $v_2$ with the component $\{v_3, v_4, v_6, v_7, v_8, v_{10}, \dots, v_{19}\}$ of $J - N[v_2]$ gives $c_2 < c_3$.
Applying it to $v_6$ with the component $\{v_7, v_8, v_{11}, v_{12}, v_{14}, v_{15}, v_{17}, v_{18}, v_{19}\}$ of $J - N[v_6]$ gives $c_6 < c_7$.

Together with~\eqref{eq:J-pins} this determines every centre order needed below.
Applying Definition~\ref{def:sandwich} to suitable pairs now gives ten chains of inequalities, each relation holding between consecutive terms.
Every $\leq$ comes from an edge, and every $<$ from a non-edge whose centre order is one of those just established.
\[
\begin{array}{rll}
  (1) & c_3 < c_{10} - r_{10} & (\times 2)\\
  (2) & c_9 - r_9 \leq c_1 + r_1 < c_2 < c_3 - r_3 & (\times 1)\\
  (3) & c_9 + r_9 < c_{14} < c_{15} - r_{15} \leq c_5 + r_5 < c_{16} & (\times 1)\\
  (4) & c_{10} < c_{19} - r_{19} \leq c_{11} + r_{11} < c_{12} & (\times 4)\\
  (5) & c_{12} - r_{12} \leq c_4 + r_4 < c_{13} & (\times 2)\\
  (6) & c_{12} + r_{12} < c_{14} < c_{15} - r_{15} \leq c_5 + r_5 < c_{16} & (\times 2)\\
  (7) & c_{13} + r_{13} < c_{14} < c_{15} - r_{15} \leq c_5 + r_5 < c_{16} & (\times 1)\\
  (8) & c_{16} + r_{16} < c_{17} < c_{18} - r_{18} \leq c_{10} + r_{10} & (\times 2)\\
  (9) & c_{13} - r_{13} \leq c_6 + r_6 < c_7 < c_8 - r_8 \leq c_3 + r_3 & (\times 1)\\
  (10) & c_{16} - r_{16} \leq c_6 + r_6 < c_7 < c_8 - r_8 \leq c_0 + r_0 < c_9 & (\times 2)
\end{array}
\]
Each of the ten lines chains at least one strict inequality, so its left end is strictly less than its right end.
With the multiplicities shown, these are eighteen strict inequalities in all.
Summing the left ends and likewise the right ends, every radius cancels.
On the left, $\mp r_9$ cancels between $(2)$ and $(3)$, $\mp 2r_{12}$ between $(5)$ and $(6)$, $\pm r_{13}$ between $(7)$ and $(9)$, and $\pm 2r_{16}$ between $(8)$ and $(10)$.
On the right, $\mp 2r_{10}$ cancels between $(1)$ and $(8)$, and $\mp r_3$ between $(2)$ and $(9)$.

Both sides therefore reduce to the same quantity
\[  2c_3 + 2c_9 + 4c_{10} + 4c_{12} + 2c_{13} + 4c_{16} , \]
and the eighteen strict inequalities assert that it is strictly less than itself, so $J \notin \sandwich$.
This completes the proof of Theorem~\ref{thm:proper-sandwich}.

\section{\texorpdfstring{Proofs for the family $Q^m$}{Proofs for the family Qm}}
\label{ap:Qm}

This appendix proves Propositions~\ref{prop:Qm-proper} (Appendix~\ref{ap:Qm-proper}), \ref{prop:Qm-not-bicoloured} (Appendix~\ref{ap:Qm-not-bicoloured}), and~\ref{prop:Qm-minimal} (Appendix~\ref{ap:Qm-minimal}).
We continue with the notation of \S\ref{sec:proper}. 

Throughout, we write $u \sim v$ when $u$ and $v$ are adjacent and $u \not\sim v$ otherwise.

\subsection{\texorpdfstring{Proof of Proposition~\ref{prop:Qm-proper}: $Q^m \in \proper$}{Proof of Proposition \ref{prop:Qm-proper}: Qm is proper-tolerance}}
\label{ap:Qm-proper}

The proof is by explicit construction.
Definition~\ref{def:baserep} and Lemma~\ref{lem:baserep} give a proper tolerance representation of $B^m$, and Definition~\ref{def:qmrep} and Lemma~\ref{lem:Qrep} extend it to one of $Q^m$, which proves Proposition~\ref{prop:Qm-proper}.

We denote the interval and tolerance of $v$ in a tolerance representation (Definition~\ref{def:tolerance}) by $J_v = [L_v,R_v]$ and $t_v > 0$, respectively, reserving $I_v$ for the interval of $v$ in a bicoloured representation.
Thus $u \sim v$ if and only if $|J_u \cap J_v| \geq \min\{t_u,t_v\}$.
If all left endpoints are distinct and all right endpoints are distinct, the representation is proper (Definition~\ref{def:tolerance-variants}) precisely when the two endpoint orderings coincide.
If this is the case, we say the representation is {\em ordered} by that common ordering, and write $u < v$ when $u$ precedes $v$ in it.
For $u < v$, we have $|J_u \cap J_v| = \max\{0,\ R_u - L_v\}$, which gives $u \sim v$ if and only if $R_u - L_v \geq \min\{t_u,t_v\}$ as all tolerances are positive. 

\begin{definition}[Base representation]
\label{def:baserep}
Fix $\delta,\gamma,\epsilon > 0$ with
\begin{equation}
\label{eq:params2}
  (m+1)(m+2)\delta < 1, \quad (m+1)(m+2)\gamma < \delta, \quad\text{ and }\quad
  8m\epsilon < \gamma.
\end{equation}
For $i \geq 1$ let
\[ f_i = \bigl(\mbinom{i+1}{2}-1\bigr)\delta, \quad h_i = \epsilon+\bigl(\mbinom{i+2}{2}-2\bigr)\delta, \]
\[ \tilde f_i = \bigl(\mbinom{i+1}{2}-1\bigr)\gamma, \quad \tilde h_i = \epsilon+\bigl(\mbinom{i+2}{2}-2\bigr)\gamma, \]
and set $K = \bigl(\binom{m+2}{2}-1\bigr)\delta$. Define
\[
\begin{array}{lll}
  J_{2i-1} = [ f_i,\ f_i+2 ],
    & t_{2i-1} = 2-i\delta & \text{where } 1 \leq i \leq m,\\
  J_{2i} = [ h_i,\ h_i+2 ],
    & t_{2i} = 2 & \text{where } 1 \leq i \leq m,\\
  J_{2m+1} = [ 1+f_m,\ 2+K ],
    & t_{2m+1} = 1-\gamma, & \\
  J_x = [ 1+f_m-\epsilon,\ 2+K-\epsilon ],
    & t_x = 1+\delta+2\epsilon-f_m, & \\
  J_{2m+2i-1} = [ 1+K+\tilde f_i,\ 2+K+\tilde f_i ],
    & t_{2m+2i-1} = 1-i\gamma & \text{where } 2 \leq i \leq m,\\
  J_{2m+2i} = [ 1+K+\tilde h_i,\ 2+K+\tilde h_i ],
    & t_{2m+2i} = 1 & \text{where } 1 \leq i \leq m,\\
  J_z = [ 2+K-t_x+\delta/4,\ 3+K ],
    & t_z = t_x-\delta/4. &
\end{array}
\]
\end{definition}

Note that the left endpoint of $J_z$ is $2+K-t_z$.
From Definition~\ref{def:baserep}, we have
\begin{equation}
\label{eq:basefacts}
\begin{gathered}
  K-f_m = (m+1)\delta, \quad R_u = 2+L_u \ \text{ for } u \in \{1, \dots, 2m\},\\
  \text{and}\quad f_i < h_i < f_{i+1}, \quad \tilde f_i < \tilde h_i < \tilde f_{i+1} \ \text{ for every } i .
\end{gathered}
\end{equation}

We remark that the base representation restricted to vertices $\{1,\dots,4m\} \cup \{x\}$ is precisely the representation of~\cite{bogart-fishburn-isaak-langley}*{Lemma~9}, including their choice $t_x = |J_x \cap J_2|$.
Their lemma requires $0 < \epsilon \ll \gamma \ll \delta < 1$, and a routine check, which we omit, shows that the choices in~\eqref{eq:params2} suffice.
The restriction is therefore a proper tolerance representation of $B^m - z$, ordered by $\ord(m)$ with $z$ removed.

\begin{lemma}
\label{lem:baserep}
The intervals and tolerances of Definition~\ref{def:baserep} form a proper tolerance representation of $B^m$, ordered by $\ord(m)$.
\end{lemma}
\begin{proof}
It remains to check $z$.
We use~\eqref{eq:basefacts} throughout, and~\eqref{eq:params2} gives $K < 1/2$ and $\tilde h_m < \delta/2$.
Since $m \geq 2$ we have $f_m \geq 2\delta$, so
\begin{equation}
\label{eq:tx}
  \mfrac{1}{2}+\delta < t_x \leq 1-\delta+2\epsilon ,
\end{equation}
where the lower bound holds because $f_m < K < 1/2$.
Now~\eqref{eq:params2} implies that $m\gamma + 2\epsilon < (m+1)\gamma < \delta$, which, combined with~\eqref{eq:tx}, gives $t_x < 1 - m\gamma$.
Every tolerance other than $t_x$ and $t_z$ is at least $1-m\gamma$, so $t_x$ is the smallest tolerance on $V(B^m) \setminus \{z\}$. Combining this with~\eqref{eq:tx}, we have
\begin{equation}
\label{eq:tz}
  t_z = t_x - \mfrac{1}{4}\delta > 0, \quad\text{ and }\quad t_z < t_x \leq t_u \ \text{ for every } u \in V(B^m) \setminus \{z\}.
\end{equation}

The right endpoint of $J_z$ is $3+K$, which is larger than every other right endpoint, the largest of which is $2+K+\tilde h_m < 3 + K$, since $\tilde h_m < \delta/2 < 1$.
The left endpoint of $J_z$ is $2+K-t_z$, which is larger than every other left endpoint, the largest of which is $1+K+\tilde h_m$, since
\[ 1-t_z = 1 - t_x + \mfrac{1}{4}\delta = f_m-\mfrac{3}{4} \delta - 2\epsilon \geq \mfrac{5}{4}\delta-2\epsilon > \tilde h_m . \]
So the representation is ordered by $\ord(m)$.

By~\eqref{eq:tz}, $\min\{t_u,t_z\} = t_z$, so $z \sim u$ if and only if $R_u - L_z \geq t_z$, i.e., $R_u \geq 2+K = R_{2m+1}$.
Since the right endpoints increase along $\ord(m)$, this holds exactly when $u$ is $2m+1$ or a later vertex, so $N(z) = \{2m+1,\dots,4m\}$, as required.
\end{proof}

We now extend the base representation to $Q^m$, referring to the vertices of $B^m$ as {\em base} vertices.
From the ordering $\ord(m)$ and Definition~\ref{def:baserep}, we get $L_{\omega_1} = 0$, $R_{\omega_1} = 2$, $L_{\omega_2} = \epsilon+\delta$, $L_{\omega_3} = 2\delta$, $L_{\omega_4} = \epsilon+4\delta$, and $L_{\omega_{4m-1}} = 1+K+\tilde h_{m-1}$.
\begin{definition}[Extending the base representation]
\label{def:qmrep}
Let
\[ L_{q_1} = \mfrac12(L_{\omega_3}+L_{\omega_4}), \quad R_{q_1} = 2+L_{q_1}, \quad\text{ and }\quad T = R_{q_1}-L_{\omega_{s_1}}, \]
and extend Definition~\ref{def:baserep} by
\[
  J_{q_1} = [ L_{q_1},\ R_{q_1} ], \quad t_{q_1} = T,
\]
\[
  J_{q_i} = \bigl[ -i,\ L_{\omega_{s_i}}+T_i \bigr], \quad
  t_{q_i} = T_i = \min\bigl\{\mfrac12(2-L_{\omega_{s_i}}),\ T\bigr\}
  \quad \text{where } 2 \leq i \leq 2m-1.
\]
\end{definition}

The interval corresponding to an outer vertex precedes the whole base, with its right endpoint placed so that it is adjacent to exactly the first $s_i$ base vertices.
The interval corresponding to the inner vertex begins between $L_{\omega_3}$ and $L_{\omega_4}$, which is what makes $q_1q_{2m-1}$ a non-edge.

\begin{lemma}
\label{lem:Qrep}
The intervals and tolerances of Definitions~\ref{def:baserep} and~\ref{def:qmrep} form a proper tolerance representation of $Q^m$.
\end{lemma}

\begin{proof}
By Lemma~\ref{lem:baserep}, the base representation is ordered by $\ord(m)$, so in particular the sequences $(L_{\omega_p})_p$ and $(R_{\omega_p})_p$ both strictly increase, and $L_{\omega_1}$ and $R_{\omega_1}$ are the smallest left and right endpoints in the base representation.
We use these facts and~\eqref{eq:basefacts} throughout.

From Definitions~\ref{def:baserep} and~\ref{def:qmrep}, together with $0 < \epsilon < \gamma < \delta$ from~\eqref{eq:params2}, we have
$0 < L_{q_1} = (L_{\omega_3} + L_{\omega_4})/2 = 3 \delta + \epsilon/2 < 5\delta \leq K$,
where the last inequality holds since $\binom{m+2}{2}-1 \geq 5$. Now
$T = R_{q_1} - L_{\omega_{4m-1}} = 1 + L_{q_1} - K - \tilde h_{m-1}$,
so $1 > T > L_{q_1} > 0$, where the first inequality holds since $L_{q_1} < K$ and the second holds since $K < 1/2$ and $\tilde h_{m-1} \leq \tilde h_m < \delta/2 < 1/2$ by~\eqref{eq:params2}.

{\em Properness.}
We show that both endpoint sequences strictly increase along the order
\[ q_{2m-1} < \cdots < q_2 < \omega_1 < \omega_2 < \omega_3 < q_1 < \omega_4 < \cdots < \omega_{4m+2} , \]
so that the representation is proper and ordered by this order.
For the left endpoints, this is immediate since $L_{q_i} = -i$ for the outer vertices, $L_{\omega_1} = 0$, $(L_{\omega_p})_p$ increases, and $L_{q_1}$ is the midpoint of $L_{\omega_3}$ and~$L_{\omega_4}$.

For the right endpoints, let $q_i$ be an outer vertex.
Now $R_{q_i} = L_{\omega_{s_i}}+T_i = \min\{1+L_{\omega_{s_i}}/2,\ L_{\omega_{s_i}}+T\}$, which is a minimum of two strictly increasing functions of $L_{\omega_{s_i}}$.
Moreover, $L_{\omega_{s_2}} > \cdots > L_{\omega_{s_{2m-1}}}$ since $s_2 > \cdots > s_{2m-1}$ and, hence,  $R_{q_{2m-1}} < \cdots < R_{q_2}$.

Since every cut satisfies $s_i \leq s_1 = 4m-1$, we have $L_{\omega_{s_i}} \leq L_{\omega_{4m-1}} = 1 + K + \tilde h_{m-1} < 2$, where the last inequality holds since $K < 1/2$ and $\tilde h_{m-1} < 1/2$ as noted above.
Hence $R_{q_i} \leq 1 + L_{\omega_{s_i}}/2 < 2 = R_{\omega_1}$.

Finally, $(R_{\omega_p})_p$ increases, and $R_{\omega_3} < R_{q_1} < R_{\omega_4}$ because $R_{q_1} = 2+L_{q_1}$ while $R_u = 2+L_u$ for~$u \in \{1, \dots, 2m\}$.

{\em Tolerances.}
We have
\begin{equation}
\label{eq:tchain}
  T_i \leq T \leq t_z \leq t_{\omega_p}
  \quad \text{ where } 2 \leq i \leq 2m-1 \text{ and } \omega_p \in V(B^m) ,
\end{equation}
where the first inequality is by definition, the last by~\eqref{eq:tz}, and the second because, using $K-f_m = (m+1)\delta$ from~\eqref{eq:basefacts},
\[ T = 1+3\delta+\mfrac12\epsilon-K-\tilde h_{m-1} < 1+(2-m)\delta+\mfrac12\epsilon-f_m \leq 1+\mfrac12\epsilon-f_m < 1+\mfrac34\delta+2\epsilon-f_m = t_z . \]

Finally, the new tolerances are positive, as $t_{q_1} = T > 0$ while $t_{q_i} = T_i = \min\{(2-L_{\omega_{s_i}})/2,\ T\}$ is positive because $T > 0$ and $L_{\omega_{s_i}} < 2$.

{\em Cuts.}
For every $1 \leq i \leq 2m-1$ we have $R_{q_i} = L_{\omega_{s_i}}+t_{q_i}$, by the definition of $T$ when $i = 1$ and of $J_{q_i}$ otherwise. 
Moreover, $\min\{t_{q_i},t_{\omega_p}\} = t_{q_i}$ for every base vertex $\omega_p$ by~\eqref{eq:tchain}.
So whenever $q_i$ precedes $\omega_p$ in the order above
\[
  q_i \sim \omega_p \iff R_{q_i}-L_{\omega_p} \geq t_{q_i}
             \iff L_{\omega_p} \leq L_{\omega_{s_i}} \iff p \leq s_i ,
\]
where the last equivalence holds by the monotonicity of $(L_{\omega_p})_p$.

An outer vertex precedes every base vertex, so $q_i$ has cut $s_i$ for $2 \leq i \leq 2m-1$.
The inner vertex $q_1$ precedes $\omega_p$ exactly when $p \geq 4$.
For $p \leq 3$ we have $R_{\omega_p}-L_{q_1} \geq 2-L_{q_1} > 1 > T$, since $R_{\omega_p} = 2+L_{\omega_p} \geq 2$ and $L_{q_1} < K < 1$, so $q_1 \sim \omega_p$.
Hence, $q_1$ is adjacent to $\omega_p$ exactly when $p \leq s_1$, i.e., $q_1$ has cut $s_1$.

{\em Edges between the new vertices.}
Let $2 \leq i < j \leq 2m-1$. Then $q_j$ precedes $q_i$ in the order above, and
$R_{q_j}-L_{q_i} = L_{\omega_{s_j}}+T_j+i > T_j \geq \min\{t_{q_i},t_{q_j}\}$, since $L_{\omega_{s_j}} \geq L_{\omega_1} = 0$, so $q_i \sim q_j$.
Thus, the outer vertices form a clique.

Now let $q_i$ be an outer vertex. Then $q_i$ precedes $q_1$ in the order above, while $\min\{t_{q_1},t_{q_i}\} = \min\{T,T_i\} = T_i$, so
\[
  q_1 \sim q_i \iff R_{q_i}-L_{q_1} \geq T_i
             \iff L_{\omega_{s_i}}+T_i-L_{q_1} \geq T_i
             \iff L_{\omega_{s_i}} \geq L_{q_1} \iff s_i \geq 4 ,
\]
where the last equivalence holds because $L_{\omega_3} < L_{q_1} < L_{\omega_4}$ and $(L_{\omega_p})_p$ is increasing.
Since $s_{2m-1} = 2$ is the only outer cut below $4$, $q_1$ is adjacent to every outer vertex except $q_{2m-1}$.

Finally, the base intervals and tolerances are unchanged, so the edges within $B^m$ remain those of Lemma~\ref{lem:baserep}.
The vertices $q_1,\dots,q_{2m-1}$ therefore induce $K_{2m-1}$ minus the edge $q_1q_{2m-1}$, each $q_i$ has cut $s_i$, and the representation is proper, so it represents $Q^m$.
\end{proof}

\subsection{\texorpdfstring{Proof of Proposition~\ref{prop:Qm-not-bicoloured}: $Q^m \notin \bicoloured$}{Proof of Proposition \ref{prop:Qm-not-bicoloured}: Qm is not bicoloured-interval}}
\label{ap:Qm-not-bicoloured}

Proposition~\ref{prop:Qm-not-bicoloured} follows from two results.
Proposition~\ref{prop:joints} says that in every bicoloured representation of $Q^m$ the vertices $q_1$, $x$, $z$, and all joints share a colour class, while Proposition~\ref{prop:impossible} says that the induced subgraph $B^m + q_1$ admits no bicoloured representation in which they do.
Since any bicoloured representation of $Q^m$ restricts to one of $B^m + q_1$, Proposition~\ref{prop:Qm-not-bicoloured} is immediate.

We first collect some general results that will be used throughout this subsection. Our first result concerns an induced path on four vertices and constrains the order of the centres.
The conclusion mirrors that of~\cite{bogart-fishburn-isaak-langley}*{Lemma~4} for proper tolerance representations (see Lemma~\ref{lem:proper-P4}).

\begin{lemma}
\label{lem:P4gen}
Let $(\sigma, c, r)$ be a bicoloured representation of a graph, and let
$v_0, v_1, v_2, v_3$ induce a path in this order, with $\sigma(v_1) = \sigma(v_2)$.
Then $\max\{c_0,c_1\} < \min\{c_2,c_3\}$ or $\max\{c_2,c_3\} < \min\{c_0,c_1\}$.
\end{lemma}

\begin{proof}
Since $\{v_2,v_3\} \cap N[v_0] = \emptyset$ and $v_2v_3$ is an edge, 
by Corollary~\ref{cor:propagate}, $c_2, c_3$ lie strictly to the same side of $c_0$.
Similarly, $c_0,c_1$ lie strictly on one side of $c_3$.
Negating every centre if necessary, we may assume $c_0 < c_2$. Then $c_0<c_3$ and, hence, $c_1 < c_3$.
It remains to show $c_1<c_2$.

Suppose otherwise, i.e., $c_2\leq c_1$, so that $c_0<c_2\leq c_1<c_3$.
The edge $v_0v_1$ and the non-edge $v_0v_2$ give $ \tau_{0,1}\ \geq\ c_1-c_0\ \geq\ c_2-c_0\ >\ \tau_{0,2}$.
Since $\sigma(v_1) = \sigma(v_2)$, the thresholds $\tau_{0,1}$ and $\tau_{0,2}$ both come from the same case of Definition~\ref{def:bicoloured}. 
Specifically, for $j \in \{1, 2\}$, $\tau_{0, j} = \max\{r_0, r_j\}$ if $\sigma(v_0) = \sigma(v_j)$ and $\tau_{0, j} = r_0 + r_j$ otherwise.
Both $\max\{r_0, r_j\}$ and $r_0 + r_j$ are non-decreasing in $r_j$.
Hence, $\tau_{0,1}>\tau_{0,2}$ implies that $r_1 > r_2$.

The edge $v_2v_3$ and the non-edge $v_1v_3$ give $\tau_{2,3}\ \geq\ c_3-c_2\ \geq\ c_3-c_1\ >\ \tau_{1,3}$.
The argument above, applied to $\tau_{2,3}$ and $\tau_{1,3}$ with $r_3$ taking the place of $r_0$, implies that $r_2 > r_1$, a contradiction.
\end{proof}

Recall, from Proposition~\ref{prop:K2t}, that $K_{2,t} \in \bicoloured$. Lemmas~\ref{lem:fan} and~\ref{lem:twins} show that a copy of $K_{2,3}$ nevertheless constrains the colouring.

\begin{lemma}
\label{lem:fan}
Let $(\sigma, c, r)$ be a bicoloured representation of a graph.
Suppose $\{u_0,u_1\}$ and $\{v_0,v_1,v_2\}$ induce either $K_{2,3}$ or $K_{2,3} - u_0v_0$.
If $\sigma(u_0) \neq \sigma(u_1)$ then $\sigma(v_1) = \sigma(v_2) = \sigma(u_0)$.
\end{lemma}

\begin{proof}
Suppose that $\sigma(u_0) \neq \sigma(u_1)$ and, for contradiction, that one of $v_1, v_2$ has colour $\sigma(u_1)$.
By symmetry, we may assume $\sigma(v_1) = \sigma(u_1)$.

If $\sigma(v_2) \neq \sigma(u_1)$ then $\sigma(v_2) = \sigma(u_0)$.
Now, the vertices $u_0, v_1, u_1, v_2$ induce a $4$-cycle in this cyclic order, whose opposite edges $v_1u_1$ and $v_2u_0$ are both monochromatic.
Hence, by Lemma~\ref{lem:C4-colouring}, all four vertices have the same colour, contradicting $\sigma(u_0) \neq \sigma(u_1)$. From here on, we assume $\sigma(v_2) = \sigma(v_1) = \sigma(u_1)$.

If $u_0 \sim v_0$ then $v_0,v_1,v_2$ are pairwise non-adjacent common neighbours of the non-adjacent vertices $u_0$ and $u_1$, so Corollary~\ref{cor:forcing}\ref{it:forcing-col} gives $\sigma(u_0) = \sigma(u_1)$, a contradiction.

So $u_0 \not\sim v_0$. For $i \in \{1,2\}$ the vertices $v_0, u_1, v_i, u_0$ induce a path in this order, whose middle edge $u_1v_i$ is monochromatic, so Lemma~\ref{lem:P4gen} gives $\max\{c_{v_0},c_{u_1}\} < \min\{c_{v_i},c_{u_0}\}$ or $\max\{c_{v_i},c_{u_0}\} < \min\{c_{v_0},c_{u_1}\}$.
The first alternative gives $c_{u_1} < c_{u_0}$ and the second $c_{u_0} < c_{u_1}$, so the same alternative holds for both $i = 1$ and $i = 2$.
Negating every centre and exchanging $v_1$ with $v_2$ if necessary, we may assume $c_{u_0} < c_{u_1}$ and $c_{v_1} < c_{v_2}$.
Then, the second alternative holds for both $i$, giving 
\begin{equation}
\label{eq:fan-a}
c_{v_1}, c_{v_2} < c_{u_1}.
\end{equation}

The non-edge $u_0u_1$ is bichromatic by assumption, so $c_{u_1} - c_{u_0} > r_{u_0} + r_{u_1}$, i.e., $c_{u_0}+r_{u_0} < c_{u_1}-r_{u_1}$.
The bichromatic edge $u_0v_2$ gives $c_{u_0}+r_{u_0} \geq c_{v_2}-r_{v_2}$.
The monochromatic non-edge $v_1v_2$ gives $c_{v_2}-c_{v_1} > \max\{r_{v_1},r_{v_2}\} \geq r_{v_2}$, so $c_{v_1} < c_{v_2}-r_{v_2}$.
Combining these, we have
\begin{equation}
\label{eq:fan-b}
  c_{v_1} < c_{v_2}-r_{v_2} \leq c_{u_0}+r_{u_0} < c_{u_1}-r_{u_1} .
\end{equation}

Finally, the monochromatic edge $u_1v_1$ gives $|c_{u_1}-c_{v_1}| \leq \max\{r_{u_1},r_{v_1}\}$.
By~\eqref{eq:fan-b}, $c_{u_1}-c_{v_1} > r_{u_1}$, so $\max\{r_{u_1},r_{v_1}\} = r_{v_1}$ and we obtain $c_{u_1} \leq c_{v_1}+r_{v_1}$.
But $c_{u_1} > c_{v_2}$ by~\eqref{eq:fan-a}, and the non-edge $v_1v_2$ gives $c_{v_2} > c_{v_1}+r_{v_1}$, a contradiction.
\end{proof}

The proof of Lemma~\ref{lem:twins} also needs the following constraint on the centres of a monochromatic $C_4$.
It mirrors~\cite{bogart-fishburn-isaak-langley}*{Lemma~5}, which gives the analogous conclusion for bounded proper tolerance representations.

\begin{lemma}
\label{lem:monoC4}
Let $(\sigma, c, r)$ be a bicoloured representation of a graph, and let
$v_0, v_1, v_2, v_3$ induce a $4$-cycle in this cyclic order, all four of one colour.
Then the centres $c_0, c_1, c_2, c_3$ are distinct, and the minimum and maximum are attained by a non-adjacent pair.
\end{lemma}

\begin{proof}
Suppose, without loss of generality, that $r_0 = \max_i \{r_i\}$. 
Now, the non-edge $v_0v_2$ gives $|c_2 - c_0| > \max\{ r_2, r_0 \} = r_0$. Negating every centre if necessary, we may assume $c_2 > c_0$, so $c_2 - c_0 > r_0$.

The edge $v_0v_1$ gives $|c_1 - c_0| \leq \max\{r_0, r_1\} = r_0$, so $c_1 \leq c_0 + r_0 < c_2$. Now, the edge $v_1v_2$ gives $c_2 - c_1 \leq \max\{r_1, r_2\} \leq r_0 < c_2 - c_0$, so $c_1 > c_0$. Hence, we have $c_0 < c_1 < c_2$.
Applying the same argument to $v_3$, we obtain $c_0 < c_1, c_3 < c_2$. Finally, the non-edge $v_1v_3$ implies that $c_1 \neq c_3$, completing the proof. 
\end{proof}

\begin{lemma}
\label{lem:twins}
Let $G$ be a graph in which $\{u_0,u_1\}$ and $\{v_0,v_1,v_2\}$ induce either $K_{2,3}$ or $K_{2,3} + v_1v_2$.
Let $w \notin \{u_0,u_1,v_0,v_1,v_2\}$ be a vertex adjacent to $v_1$ and $v_2$ but to neither $u_0$ nor $u_1$.
Then in every bicoloured representation $(\sigma, c, r)$ of $G$ we have $\sigma(v_1) = \sigma(v_2)$.
\end{lemma}

\begin{proof}
Suppose, for contradiction, that $\sigma(v_1) \neq \sigma(v_2)$, so that $\sigma(v_0)$ agrees with exactly one of $\sigma(v_1)$ and $\sigma(v_2)$.
Exchanging $v_1$ with $v_2$ preserves every hypothesis and the conclusion, so we may assume $\sigma(v_0) = \sigma(v_1) \neq \sigma(v_2)$.

Since $\sigma(v_0) \neq \sigma(v_2)$, Lemma~\ref{lem:fan}, applied to the non-adjacent pair $v_0, v_2$, their common neighbours $u_0, u_1$, and the further neighbour $w$ of $v_2$, gives $\sigma(u_0) = \sigma(u_1) = \sigma(v_0)$.
Swapping the two colours if necessary, suppose $u_0,u_1,v_0,v_1$ are all coloured $+$ while $\sigma(v_2) = -$.
Now, $v_0v_2$ is a bichromatic non-edge, so $I_{v_0} \cap I_{v_2} = \emptyset$, while the non-edge $u_0u_1$ is monochromatic, so after relabelling $u_0,u_1$ if necessary
\begin{equation}
\label{eq:book-ab}
  c_{u_1}-c_{u_0} > \max\{r_{u_0},r_{u_1}\}.
\end{equation}

The vertices $u_0, v_0, u_1, v_1$ induce a monochromatic $4$-cycle in this cyclic order, whose non-adjacent pairs are $u_0u_1$ and $v_0v_1$, so by Lemma~\ref{lem:monoC4} the least and greatest centres are attained either by $v_0,v_1$ or by $u_0,u_1$, giving two cases.
In each case we may reflect all centres, provided that we also exchange $u_0$ with $u_1$, since the reflection reverses~\eqref{eq:book-ab}, which is restored by the exchange.

{\em Case 1: $c_{u_0},c_{u_1}$ strictly between $c_{v_0}$ and $c_{v_1}$.}
Reflecting if necessary, assume $c_{v_0} < c_{u_0} < c_{u_1} < c_{v_1}$.
The monochromatic edge $v_0u_1$ gives $c_{u_1}-c_{v_0} \leq \max\{r_{v_0},r_{u_1}\}$, while $c_{u_1}-c_{v_0} > c_{u_1}-c_{u_0} > \max\{r_{u_0},r_{u_1}\} \geq r_{u_1}$ by~\eqref{eq:book-ab}. 
So $ \max\{r_{v_0},r_{u_1}\} = r_{v_0}$ and $c_{u_1} \in I_{v_0}$. As $c_{v_0} < c_{u_0} < c_{u_1}$, we also have $c_{u_0} \in I_{v_0}$, and, in particular, $[ c_{u_0},c_{u_1} ] \subseteq I_{v_0}$.

Now, since $I_{v_2} \cap I_{v_0} = \emptyset$, $I_{v_2}$ lies entirely to one side of $[ c_{u_0},c_{u_1} ]$.
But $u_0v_2$ and $u_1v_2$ are edges, so $I_{v_2}$ meets both $I_{u_0}$ and $I_{u_1}$.
Suppose first that $I_{v_2}$ lies to the left, so that $c_{v_2}+r_{v_2} < c_{u_0}$.
Then $I_{v_2} \cap I_{u_1} \neq \emptyset$ gives $c_{v_2}+r_{v_2} \geq c_{u_1}-r_{u_1}$, and hence $c_{u_1}-c_{u_0} < r_{u_1}$.
Suppose instead that $I_{v_2}$ lies to the right, so that $c_{v_2}-r_{v_2} > c_{u_1}$.
Then $I_{v_2} \cap I_{u_0} \neq \emptyset$ gives $c_{v_2}-r_{v_2} \leq c_{u_0}+r_{u_0}$, and hence $c_{u_1}-c_{u_0} < r_{u_0}$.
In either case, we obtain a contradiction to~\eqref{eq:book-ab}.

{\em Case 2: $c_{v_0},c_{v_1}$ strictly between $c_{u_0}$ and $c_{u_1}$.}
The vertices $u_0,u_1,w$ are pairwise non-adjacent.
So, by Corollary~\ref{cor:forcing}\ref{it:forcing-geo} applied to $v_0,v_2$, whose intervals are disjoint, their common neighbours $u_0,u_1$ and the vertex $w$, $c_w$ lies strictly to one side of both $c_{u_0}$ and $c_{u_1}$.
Reflecting if necessary, assume $c_{u_0} < c_{u_1} < c_w$.

Moreover, $v_0 \not\sim w$, since otherwise $u_0, u_1, w$ would be three pairwise non-adjacent common neighbours of the non-adjacent pair $v_0, v_2$, and Corollary~\ref{cor:forcing}\ref{it:forcing-col} would give $\sigma(v_0) = \sigma(v_2)$, contrary to our assumption.

Now $w, v_1, u_1, v_0$ induce a path in this order, whose middle edge $v_1u_1$ is monochromatic. 
By Lemma~\ref{lem:P4gen}, we have $\max\{c_w,c_{v_1}\} < \min\{c_{u_1},c_{v_0}\}$ or $\max\{c_{u_1},c_{v_0}\} < \min\{c_w,c_{v_1}\}$.
The former gives $c_w < c_{u_1}$, contradicting $c_{u_1} < c_w$. 
So the latter holds and $c_{u_1} < c_{v_1}$, contradicting the case hypothesis.
\end{proof}

\subsubsection{Colouring the joints}

The colouring is pinned down in two steps.
Proposition~\ref{prop:Wm-mono} shows that $q_1$, $x$, and $z$ already share a colour class in the subgraph $B^m + q_1$.
Proposition~\ref{prop:joints} extends that class to every joint of $Q^m$.

\begin{proposition}
\label{prop:Wm-mono}
In every bicoloured representation of $B^m + q_1$, $\sigma(q_1) = \sigma(x) = \sigma(z)$.
\end{proposition}

\begin{proof}
Let $\alpha = 2m+1$ and $\beta = 2m+2$.
The vertices $\alpha, \beta, 4m$ are distinct and pairwise non-adjacent, and all three lie in $N(x) \cap N(z)$, while $x \not\sim z$, so Corollary~\ref{cor:forcing}\ref{it:forcing-col} gives $\sigma(x) = \sigma(z)$.
It remains to show that $\sigma(q_1) = \sigma(z)$.

Suppose, for contradiction, that $\sigma(q_1) \neq \sigma(z)$.
Now applying Lemma~\ref{lem:fan} with $(u_0,u_1,v_0,v_1,v_2) = (q_1, z, 4m, \alpha, \beta)$ gives $\sigma(\alpha) = \sigma(q_1)$, and with $(u_0,u_1,v_0,v_1,v_2) = (z, q_1, 1, \alpha, \beta)$ it gives $\sigma(\alpha) = \sigma(z)$, so $\sigma(q_1) = \sigma(z)$, the desired contradiction.

The hypotheses of Lemma~\ref{lem:fan} are straightforward to verify in both applications using Definitions~\ref{def:GmBm} and~\ref{def:Qm}.
Here $m \geq 2$ ensures that $\beta \leq 4m-2$ and $\alpha \geq 5$, so that $q_1 \sim \beta$ and $1 \not\sim \alpha, \beta$.
\end{proof}

The outer vertices $q_2, \dots, q_{2m-1}$ force the colour of every joint but the top one, and $q_1$ forces the joint $4m-1$.

\begin{proposition}
\label{prop:joints}
In every bicoloured representation of $Q^m$, the vertices $q_1$, $x$, $z$, and all joints belong
to the same colour class.
\end{proposition}

\begin{proof}
Restricting the representation to $V(B^m) \cup \{q_1\}$ and applying Proposition~\ref{prop:Wm-mono} gives $\sigma(q_1) = \sigma(x) = \sigma(z)$.
Swapping the two colours if necessary, we may assume that this common colour is $+$, so it remains to show that every joint is coloured $+$.
Recall from Definition~\ref{def:Qm} that $q_1$ has cut $s_1 = 4m-1$, and that the cuts of $q_2,\dots,q_{2m-1}$ are the even values $2 \leq s \leq 2m$ and the odd values $2m+3 \leq s \leq 4m-3$.
The vertex $q_{2m-1}$ has cut $2$, and it is the only one of $q_2,\dots,q_{2m-1}$ not adjacent to $q_1$.

All the joints but $4m-1$ are handled by Lemma~\ref{lem:twins}, applied once for each outer vertex $q_i$, while $4m-1$ is treated separately at the end.
Let $q_i$ be one of $q_2,\dots,q_{2m-1}$, of cut $s = s_i$, and set
\[
  a = \begin{cases}
        s-1 & \text{if } s \text{ is even},\\
        s-2 & \text{if } s \text{ is odd},
      \end{cases}
  \quad\text{and}\quad
  b = \begin{cases}
        q_1 & \text{if } s = 2,\\
        x   & \text{if } s \text{ is even and } s \geq 4,\\
        z   & \text{if } s \text{ is odd}.
      \end{cases}
\]
Note that $a + 4 \leq 4m - 1$, so $a$, $a+1$, $a+2$, and $a+4$ are chain vertices. Moreover, $a$ is odd.
We apply the lemma with $v_0 = q_i$ and $(u_0,\ u_1,\ v_1,\ v_2,\ w)  =  (a,\ a+1,\ a+2,\ b,\ a+4)$.
That the hypotheses hold is straightforward to verify using Definitions~\ref{def:GmBm}, \ref{def:cuts}, and~\ref{def:Qm}.
For $s = 2$, the non-adjacency $v_0 \not\sim v_2$ is precisely the missing edge $q_1q_{2m-1}$.
The conclusion $\sigma(v_1) = \sigma(v_2)$ gives $\sigma(a+2) = \sigma(b)$, and each of $x$, $q_1$, $z$ is coloured $+$, so $\sigma(a+2) = +$.

As $a+2$ runs over $3, 5, \dots, 4m-3$ when $s$ runs over the cuts of the outer vertices, every joint other than $4m-1$ is coloured $+$.
Only the joint $4m-1$ remains.
Suppose, for contradiction, that $\sigma(4m-1) = -$.
We apply Lemma~\ref{lem:fan} with $(u_0,u_1,v_0,v_1,v_2) = (4m-1, q_1, 1, 4m-3, 4m-2)$, whose hypotheses are again straightforward to verify from Definitions~\ref{def:GmBm} and~\ref{def:Qm}, using $m \geq 2$.
Since $\sigma(4m-1) \neq \sigma(q_1)$, the lemma gives $\sigma(4m-3) = \sigma(4m-1) = -$, contradicting $\sigma(4m-3) = +$.
\end{proof}

\subsubsection{The contradiction}

The contradiction comes from Proposition~\ref{prop:cl8} below, which is the bicoloured counterpart of~\cite{bogart-fishburn-isaak-langley}*{Lemma~8}.
Its hypothesis asks for each even vertex $2i$ to lie to the right of the joint $2i-1$ with $2 \leq i \leq 2m-1$, together with two inequalities among the centres and radii of $B^m - z$.
Proposition~\ref{prop:impossible} then derives that order from the colouring and verifies the two inequalities, which are forced by $q_1$ and $z$.

\begin{proposition}
\label{prop:cl8}
Let $m \geq 2$ and let $(\sigma, c, r)$ be a bicoloured representation of a graph containing $B^m - z$ as an induced subgraph, with $c_{2i-1} < c_{2i}$ for each $2 \leq i \leq 2m-1$ and
\begin{equation}
\label{eq:cl8-ord}
  c_{4m-1}-c_{4m-3}  >  r_{4m-3}
  \quad\text{and}\quad
  c_x < c_{2m+1} .
\end{equation}
Then $x$ and the joints $3, 5, \dots, 4m-1$ do not all have the same colour.
\end{proposition}
\begin{proof}
Suppose, for contradiction, that $x$ and all the joints carry one colour.
Swapping the two colours if necessary, we may assume that this colour is $+$.
For $2 \leq i \leq 2m-1$, we define the {\em gap} $g_i$ to equal $c_{2i+1}-c_{2i-1}$.
Set $d = 0$ if $\sigma(1) = +$ and $d = r_1$ if $\sigma(1) = -$.
For clarity, we present the proof in steps.

{\em The gaps $g_i$ increase for $i \geq 2$.} We apply reverse induction to show 
\begin{equation}
\label{eq:cl8-1}
  g_i > r_{2i-1} \quad \text{ for each } 2 \leq i \leq 2m-1.
\end{equation}
The case $i = 2m-1$ is the hypothesis $g_{2m-1} = c_{4m-1} - c_{4m-3} > r_{4m-3}$. Now fix $2 \leq i \leq 2m-2$. Suppose $g_{i+1} > r_{2i+1}$ and, for contradiction, that $g_i \leq r_{2i-1}$.

The monochromatic edge $\{2i+1, 2i+3\}$ gives $g_{i+1} \leq \max\{r_{2i+1}, r_{2i+3}\}$, so, as $g_{i+1} > r_{2i+1}$, we have $g_{i+1} \leq r_{2i+3}$.
By hypothesis $c_{2i-1} < c_{2i}$, so the non-edge $\{2i-1, 2i\}$ gives $c_{2i} - c_{2i-1} > r_{2i-1} \geq g_i = c_{2i+1} - c_{2i-1}$. Hence, $c_{2i} > c_{2i+1}$.

If $c_{2i} > c_{2i+3}$, the non-edge $\{2i, 2i+3\}$ gives $c_{2i} - c_{2i+3} > r_{2i}$.
Combined with the assumption $g_{i+1} > r_{2i+1}$, this gives $c_{2i} - c_{2i+1} > r_{2i} + r_{2i+1}$, contradicting the edge $\{2i, 2i+1\}$.
Hence, as the non-edge $\{2i, 2i+3\}$ has distinct centres, $c_{2i+1} < c_{2i} < c_{2i+3}$.

Now, the non-edge $\{2i, 2i+3\}$ gives $c_{2i+3} - c_{2i} > r_{2i+3} \geq g_{i+1} = c_{2i+3} - c_{2i+1}$, i.e., $c_{2i} < c_{2i+1}$, a contradiction.
Hence, $g_i > r_{2i-1}$, which completes the induction and proves~\eqref{eq:cl8-1}.

Now let $2 \leq i \leq 2m-1$.
By~\eqref{eq:cl8-1}, $g_i > 0$, so $c_3 < c_5 < \cdots < c_{4m-1}$, and the monochromatic edge $\{2i-1,2i+1\}$ gives $g_i \leq \max\{r_{2i-1},r_{2i+1}\}$, so~\eqref{eq:cl8-1} gives
\begin{equation}
\label{eq:cl8-2}
  r_{2i+1} \geq g_i \quad \text{ for each } 2 \leq i \leq 2m-1.
\end{equation}
Thus, $g_{i+1} > r_{2i+1} \geq g_i$ for $2 \leq i \leq 2m-2$, so $g_2 < g_3 < \cdots < g_{2m-1}$.

Pair $g_{m+j}$ with $g_{j+1}$ for $1 \leq j \leq m-1$, so that the two indices lie between $2$ and $2m-1$. Summing over $j$ gives
\begin{equation}
\label{eq:cl8-3}
  c_{4m-1}-c_{2m+1} = \sum_{i=m+1}^{2m-1} g_i
  > \sum_{i=2}^{m} g_i = c_{2m+1}-c_3 .
\end{equation}

{\em The vertex $1$ lies to the left of every vertex of $B^m - z$ other than $3$.}
Let $S = V(B^m - z) \setminus \{1,3\}$.
Then $S$ induces a connected subgraph, since $x \in S$ is adjacent to every other vertex of $S$.
Moreover, $S \cap N[1] = \emptyset$, as $N_{B^m-z}(1) = \{3\}$ and $B^m - z$ is an induced subgraph.
Hence, by Corollary~\ref{cor:propagate}, the centres of $S$ all lie on one side of $c_1$.

If they lie to the left, then $c_5 < c_1$, so the non-edge $\{1,5\}$ gives $c_1 - c_5 > r_1$, while~\eqref{eq:cl8-1} gives $c_5 - c_3 = g_2 > r_3$.
Adding the two inequalities gives $c_1 - c_3 > r_1 + r_3$, contradicting the edge $\{1,3\}$.
Hence $c_1 < c_v$ for every $v \in S$.
In particular, $c_1 < c_2$, $c_1 < c_x$, and $c_1 < c_{4m}$.

If $\sigma(1) = +$, the monochromatic non-edge $\{1,x\}$ gives $c_x - c_1 > \max\{r_1, r_x\} \geq r_x$, and otherwise the bichromatic non-edge $\{1,x\}$ gives $c_x - c_1 > r_1 + r_x$.
In either case, by the definition of $d$, we have
\begin{equation}
\label{eq:cl8-4}
  c_x - r_x > c_1 + d .
\end{equation}
Moreover, the vertices $1, 3, \dots, 4m-1$ induce a path, and none of them is adjacent to $4m$.
So, by Corollary~\ref{cor:propagate}, their centres all lie on one side of $c_{4m}$. As $c_1 < c_{4m}$, they lie to the left, so $c_{4m} > c_{4m-1}$, and the non-edge $\{4m-1,4m\}$ gives
\begin{equation}
\label{eq:cl8-5}
  c_{4m}-c_{4m-1} > \max\{r_{4m-1},r_{4m}\} .
\end{equation}

{\em The vertex $3$ lies at most $d+r_3$ to the right of vertex $1$.}
If $\sigma(1) = -$, the assertion is implied by the edge $\{1,3\}$, which gives $c_3-c_1 \leq r_1+r_3 = d+r_3$.
Hence, we may assume that $\sigma(1) = +$, so $d = 0$. We show $c_3 - c_1 \leq r_3$.

We first show $c_2 < c_3$.
The non-edge $\{2,5\}$ gives $|c_2-c_5| > \max\{r_2,r_5\}$.
Here $c_2 > c_5$ is impossible, since $c_2-c_5 > r_2$ with $c_5-c_3 = g_2 > r_3$ gives $c_2-c_3 > r_2+r_3$, contradicting the edge $\{2,3\}$.
So $c_2 < c_5$, and $c_5-c_2 > r_5 \geq g_2 = c_5-c_3$ by~\eqref{eq:cl8-2}, so $c_2 < c_3$.

Since $c_1 < c_2$, the non-edge $\{1,2\}$ gives $c_2 - c_1 > r_1$, so $c_3-c_1 > c_2-c_1 > r_1$, and the monochromatic edge $\{1,3\}$ gives $c_3-c_1 \leq \max\{r_1,r_3\} = r_3$, as desired.

In each case, we have $c_3 - c_1 \leq d + r_3$.
Along with~\eqref{eq:cl8-1} and~\eqref{eq:cl8-2}, this gives
\begin{equation}
\label{eq:cl8-6}
  c_3-c_1-d  \leq  r_3  <  g_2  \leq  g_{2m-1} .
\end{equation}

{\em The interval $I_x$ reaches further right than $I_{4m-1}$.}
The hypothesis~\eqref{eq:cl8-ord} gives $c_x < c_{2m+1}$. Now $2m+1 < 4m-1$ since $m \geq 2$, so $c_x < c_{2m+1} < c_{4m-1}$.
Combined with~\eqref{eq:cl8-5}, this gives $c_{4m}-c_x > c_{4m}-c_{4m-1} > r_{4m}$.

If $\sigma(4m) = +$, the monochromatic edge $\{x,4m\}$ gives $c_{4m}-c_x \leq \max\{r_x,r_{4m}\}$.
Since $c_{4m}-c_x > r_{4m}$, that maximum is $r_x$, so $c_x+r_x \geq c_{4m} > c_{4m-1}+r_{4m-1}$, the second inequality by~\eqref{eq:cl8-5}.

If $\sigma(4m) = -$, the bichromatic non-edge $\{4m-1,4m\}$ gives $c_{4m}-c_{4m-1} > r_{4m}+r_{4m-1}$ and the bichromatic edge $\{x,4m\}$ gives $c_{4m}-c_x \leq r_x+r_{4m}$.
Together these give $c_x+r_x \geq c_{4m}-r_{4m} > c_{4m-1}+r_{4m-1}$.

Hence, for either colouring of the vertex $4m$, we obtain
\begin{equation}
\label{eq:cl8-7}
  c_x+r_x  >  c_{4m-1}+r_{4m-1} .
\end{equation}
Finally, putting the bounds above together,
\begin{align*}
  2c_x &  >  (c_1+d)+r_x+c_x  =  (c_1+d)+(c_x+r_x) && \text{by~\eqref{eq:cl8-4}}\\
       &  >  (c_1+d)+c_{4m-1}+r_{4m-1} && \text{by~\eqref{eq:cl8-7}}\\
       &  >  (c_1+d)+c_{4m-1}+(c_3-c_1-d)  =  c_3+c_{4m-1} && \text{by~\eqref{eq:cl8-2} and~\eqref{eq:cl8-6}}\\
       &  >  c_3+(2c_{2m+1}-c_3)  =  2c_{2m+1} && \text{by~\eqref{eq:cl8-3},}
\end{align*}
contradicting the hypothesis $c_x < c_{2m+1}$.
\end{proof}

The proof of Proposition~\ref{prop:impossible} also uses the following consequence of Lemma~\ref{lem:P4gen}: if the interior vertices of a longer path share a colour then their centres are monotone.

\begin{corollary}
\label{cor:noextremum}
Let $(\sigma, c, r)$ be a bicoloured representation of a graph, let $k \geq 3$, and let $v_0,\dots,v_k$ induce a path in this order.
If $\sigma(v_1) = \cdots = \sigma(v_{k-1})$ then $c_1 < c_2 < \cdots < c_{k-1}$ or $c_{k-1} < \cdots < c_2 < c_1$.
\end{corollary}

\begin{proof}
For $1 \leq j \leq k-2$ the vertices $v_{j-1}, v_j, v_{j+1}, v_{j+2}$ induce a path in this order, and its middle edge $v_jv_{j+1}$ is monochromatic, so Lemma~\ref{lem:P4gen} gives
$ \max\{c_{j-1},c_j\} < \min\{c_{j+1},c_{j+2}\} $ or $\max\{c_{j+1},c_{j+2}\} < \min\{c_{j-1},c_j\}$.
The first alternative at $j$ gives $c_j < c_{j+1}$ and $c_j < c_{j+2}$, and the second gives the reverse inequalities.
So the first at $j$ and the second at $j+1$ cannot both hold, since one gives $c_j < c_{j+2}$ and the other $c_{j+2} < c_j$, and likewise with the roles exchanged.
Hence, the same alternative holds at every $j$, and it gives $c_1 < c_2 < \cdots < c_{k-1}$ or $c_{k-1} < \cdots < c_2 < c_1$.
\end{proof}

In $B^m + q_1$ the vertex $q_1$ forces the first inequality of~\eqref{eq:cl8-ord} and $z$ forces the second.

\begin{proposition}
\label{prop:impossible}
No bicoloured representation of $B^m + q_1$ has $q_1$, $x$, $z$, and every joint in a single colour class.
\end{proposition}

\begin{proof}
Suppose, for contradiction, that $(\sigma, c, r)$ is such a representation.
Swapping the two colours if necessary, we may assume that $q_1$, $x$, $z$, and every joint are coloured $+$.
The vertices $1, 3, \dots, 4m-1, 4m-2$ induce a path in this order, and its interior vertices are exactly the joints, all coloured $+$, so Corollary~\ref{cor:noextremum} implies that the sequence $c_3, c_5, \dots, c_{4m-1}$ is strictly monotone. 
Negating the centres if necessary, we may assume $c_3 < c_5 < \cdots < c_{4m-1}$.

Now let $2 \leq i \leq 2m-1$.
The vertices $2i-2$, $2i-1$, $2i+1$, $2i$ induce a path in this order, whose middle edge joins the joints $2i-1$ and $2i+1$, so Lemma~\ref{lem:P4gen} gives $\max\{c_{2i-2},c_{2i-1}\} < \min\{c_{2i+1},c_{2i}\}$ or $\max\{c_{2i+1},c_{2i}\} < \min\{c_{2i-2},c_{2i-1}\}$.
The second implies $c_{2i+1} < c_{2i-1}$, contradicting the order just fixed, so the first holds and $c_{2i-1} < c_{2i}$.
Thus
\begin{equation}
\label{eq:imp-order}
  c_3 < c_5 < \cdots < c_{4m-1} \quad\text{and}\quad
  c_{2i-1} < c_{2i} \quad \text{ for } 2 \leq i \leq 2m-1 .
\end{equation}
Since~\eqref{eq:imp-order} holds, it suffices to show that both inequalities of~\eqref{eq:cl8-ord} hold.
Proposition~\ref{prop:cl8} then says that $x$ and the joints do not all have the same colour, a contradiction.

{\em The hypothesis $c_{4m-1} - c_{4m-3} > r_{4m-3}$.}
Suppose first that $c_{4m-2} < c_{4m-1}$.
Since $c_{4m-3} < c_{4m-2}$ by~\eqref{eq:imp-order}, the non-edge ${\{4m-3,4m-2\}}$ gives $c_{4m-2}-c_{4m-3} > r_{4m-3}$.
Along with $c_{4m-2} < c_{4m-1}$, this implies $c_{4m-1}-c_{4m-3} > r_{4m-3}$ as required. From here on, we may assume $c_{4m-1} \leq c_{4m-2}$.

Let $S = \{q_1, 1, 2, \dots, 4m-4\}$.
Then $S$ induces a connected subgraph, as $q_1$ is adjacent to every other vertex of $S$.
Moreover, $S \cap N[4m-1] = \emptyset$, as $N(4m-1) = \{4m-3, 4m-2, x, z\}$.
Hence, by Corollary~\ref{cor:propagate}, the centres of $S$ all lie on one side of $c_{4m-1}$, and since $3 \in S$ with $c_3 < c_{4m-1}$ by~\eqref{eq:imp-order}, they lie to the left.
In particular, $c_{q_1} < c_{4m-1}$.

Note that the vertices $1, q_1, 4m-3, 4m-1$ induce a path in this order, whose middle edge $\{q_1, 4m-3\}$ is monochromatic.
So Lemma~\ref{lem:P4gen} gives $\max\{c_1, c_{q_1}\} < \min\{c_{4m-3}, c_{4m-1}\}$ or $\max\{c_{4m-3}, c_{4m-1}\} < \min\{c_1, c_{q_1}\}$.
The latter would give $c_{4m-1} < c_{q_1}$, contrary to $c_{q_1} < c_{4m-1}$, so the former holds and, hence, $c_{q_1} < c_{4m-3}$.
Hence, we have $c_{q_1} < c_{4m-3} < c_{4m-1} \leq c_{4m-2}$ and the monochromatic non-edge $\{q_1, 4m-1\}$ gives $c_{4m-1} - c_{q_1} > \max\{r_{q_1}, r_{4m-1}\} \geq r_{q_1}$.

If we also have $\sigma(4m-2) = +$, the monochromatic edge $\{q_1, 4m-2\}$ gives $c_{4m-2} - c_{q_1} \leq \max\{r_{q_1}, r_{4m-2}\}$. 
Now $c_{4m-2} - c_{q_1} \geq c_{4m-1} - c_{q_1} > r_{q_1}$, so $\max\{r_{q_1}, r_{4m-2}\} = r_{4m-2}$.
Hence, $c_{q_1} \geq c_{4m-2} - r_{4m-2}$.
Along with the monochromatic non-edge  $\{4m-3, 4m-2\}$, this implies $c_{4m-3} < c_{4m-2} - r_{4m-2} \leq c_{q_1}$, contradicting $c_{q_1} < c_{4m-3}$.

We may now assume $\sigma(4m-2) = -$.
Suppose also, for contradiction, that $c_{4m-1} - c_{4m-3} \leq r_{4m-3}$.
The bichromatic non-edge $\{4m-3, 4m-2\}$ gives $c_{4m-2} - c_{4m-3} > r_{4m-3} + r_{4m-2} \geq (c_{4m-1} - c_{4m-3}) + r_{4m-2}$. 
Subtracting $c_{4m-1} - c_{4m-3}$ from both sides gives $c_{4m-2} - c_{4m-1} > r_{4m-2}$.

Finally, the bichromatic edge $\{q_1, 4m-2\}$ gives $c_{4m-2} - c_{q_1} \leq r_{q_1} + r_{4m-2}$. Along with $c_{4m-1} - c_{q_1} > r_{q_1}$ this yields $c_{4m-2} - c_{4m-1} < r_{4m-2}$, a contradiction.
So $c_{4m-1} - c_{4m-3} > r_{4m-3}$.

{\em The hypothesis $c_x < c_{2m+1}$.}
The vertices $z, 2m+1, 2m-1, 2m-2$ induce a path in this order, whose middle edge is monochromatic.
By Lemma~\ref{lem:P4gen},  $\max\{c_z,c_{2m+1}\} < \min\{c_{2m-1},c_{2m-2}\}$ or $\max\{c_{2m-1},c_{2m-2}\} < \min\{c_z,c_{2m+1}\}$.
The first implies $c_{2m+1} < c_{2m-1}$, contradicting~\eqref{eq:imp-order}, so the second holds and, hence, $c_{2m-2} < c_{2m+1}$.

Similarly, the vertices $z, 2m+1, x, 2m-2$ induce a path in this order, whose middle edge $\{2m+1, x\}$ is monochromatic. 
Now Lemma~\ref{lem:P4gen} gives $\max\{c_z,c_{2m+1}\} < \min\{c_x,c_{2m-2}\}$ or $\max\{c_x,c_{2m-2}\} < \min\{c_z,c_{2m+1}\}$.
The first implies $c_{2m+1} < c_{2m-2}$, contrary to $c_{2m-2} < c_{2m+1}$, so the second holds and $c_x < c_{2m+1}$, as desired.
\end{proof}

\subsection{\texorpdfstring{Proof of Proposition~\ref{prop:Qm-minimal}: $Q^m - u \in \bicoloured$}{Proof of Proposition \ref{prop:Qm-minimal}: Qm minus u is bicoloured-interval}}
\label{ap:Qm-minimal}

We use $50\%$-tolerance representations in centre and radius form. Each vertex $v$ is assigned a centre $c_v$ and a radius $r_v > 0$, with $I_v = [c_v-r_v,\ c_v+r_v]$ so that $uv \in E(G)$ if and only if $|c_u - c_v| \leq \max\{r_u, r_v\}$, or equivalently either $c_u \in I_v$ or $c_v \in I_u$.
Since this is the monochromatic case of Definition~\ref{def:bicoloured}, such a representation shows that $G \in \bicoloured$ (Proposition~\ref{prop:unit-in-B}).
We write $S_v = \{w : c_w \in I_v\}$ for the {\em star} at $v$, whose vertices are consecutive in the centre order, so $E(G)$ consists of the pairs $vw$ with $v \neq w \in S_v$ for each $v \in V(G)$.

The proof is a case analysis on the deleted vertex $u \in V(Q^m)$, with the following cases.
\[
\begin{array}{llrl}
\hline
\text{case} & \text{deleted vertex} & \text{number of vertices} & \text{representation}\\ \hline
\text{(C1)} & u = x \text{ or } u = z           & 2      & \text{$50\%$}\\
\text{(C2)} & u \in \{1,\dots,4m\}              & 4m     & \text{$50\%$}\\
\text{(C3)} & u = q_i,\ 1 \leq i \leq 2m-1      & 2m-1   & \text{bicoloured}\\
\hline
\end{array}
\]
We first set up three tools that all cases share.
Lemma~\ref{lem:omitxz} then builds the bases for Case~(C1) and Lemma~\ref{lem:baseC3} those for Case~(C2), Corollary~\ref{cor:C1C2} settles Cases~(C1) and~(C2), and Corollary~\ref{cor:C3} settles Case~(C3).

The first tool builds a representation of the chain from a sequence $g_1,\dots,g_{2m}$ of {\em gaps}.
Here $g_i$ is the distance between the odd centres $c_{2i-1}$ and $c_{2i+1}$ for $i \leq 2m-1$, while $g_{2m}$ serves only to place $c_{4m}$.
The parameter $\rho$ is the common radius of the even vertices, chosen small enough that their intervals contain no other centre.

\begin{definition}[The chain representation]
\label{def:chain}
Let $g_1,\dots,g_{2m}$ and $\rho$ be positive reals.
The {\em chain representation} on these gaps assigns centres and radii as
\[
\begin{array}{llll}
\hline
\text{vertex } v & \text{centre } c_v & \text{radius } r_v & \text{range}\\ \hline
1    & 0                              & g_1/4 & \\
2i-1 & c_{2i-3}+g_{i-1}               & g_{i-1}      & 2 \leq i \leq 2m\\
2    & g_1/2                   & \rho         & \\
2i   & c_{2i-1}+ (g_{i-1}+g_i)/2 & \rho         & 2 \leq i \leq 2m\\
\hline
\end{array}
\]
\end{definition}

Centres with consecutive labels in Definition~\ref{def:chain} differ by
\begin{equation}
\label{eq:diffs}
\begin{gathered}
  c_2-c_1 = c_3-c_2 = \mfrac12 g_1,\\
  c_{2i}-c_{2i-1} = \mfrac12(g_{i-1}+g_i) \quad \text{where } 2 \leq i \leq 2m,\\
  c_{2i+1}-c_{2i} = \mfrac12(g_i-g_{i-1}) \quad \text{where } 2 \leq i \leq 2m-1 .
\end{gathered}
\end{equation}

The next lemma describes the star at an odd vertex from local information alone, and assumes nothing about the gap sequence.

\begin{lemma}
\label{lem:oddstar}
Let $2 \leq i \leq 2m$, and let $(c, r)$ assign centres and radii to a set of vertices containing $2i-3$, $2i-2$, and $2i-1$, with the centres increasing in the labels.
If $r_{2i-1} = g$ and $c_{2i-1} = c_{2i-3}+g$ for some $g > 0$, then $I_{2i-1} = [\,c_{2i-3},\ c_{2i-1}+g\,]$. In particular, $S_{2i-1}$ consists of $2i-3$, $2i-2$, and $2i-1$, together with any vertex whose centre lies in $(\,c_{2i-1},\ c_{2i-1}+g\,]$.
\end{lemma}

\begin{proof}
The interval $I_{2i-1} = [\,c_{2i-1}-g,\ c_{2i-1}+g\,]$ has left endpoint $c_{2i-3}$ by hypothesis.
Since the centres increase with the labels, the centres in $[\,c_{2i-3},\ c_{2i-1}\,]$ are exactly those of $2i-3$, $2i-2$, and $2i-1$, while $I_{2i-1} \setminus [\,c_{2i-3},\ c_{2i-1}\,] = (\,c_{2i-1},\ c_{2i-1}+g\,]$, so every other vertex of $S_{2i-1}$ has its centre in that interval.
\end{proof}

\begin{lemma}
\label{lem:chain}
Suppose that $2\rho < g_1$ and $2\rho < g_i-g_{i-1}$ for $2 \leq i \leq 2m$, so the gaps increase.
Then the chain representation of Definition~\ref{def:chain} is a $50\%$-tolerance representation of the chain $B^m - \{x, z\}$ with $c_1 < c_2 < \cdots < c_{4m}$ and stars
\[
  S_1 = \{1\}, \quad
  S_{2i-1} = \{2i-3, 2i-2, 2i-1\} \quad \text{where } 2 \leq i \leq 2m, \quad
  S_{2i} = \{2i\} \quad \text{where } 1 \leq i \leq 2m .
\]
\end{lemma}

\begin{proof}
All the differences in~\eqref{eq:diffs} are positive, so $c_1 < c_2 < \cdots < c_{4m}$.
All of them are also larger than $\rho$, the first because $2\rho < g_1$, the second because $(g_{i-1}+g_i)/2 \geq g_{i-1} \geq g_1 > 2\rho$, and the third because $2\rho < g_i-g_{i-1}$.
Each even interval $I_{2i}$ therefore contains no centre other than $c_{2i}$, so $S_{2i} = \{2i\}$.

We have $I_1 = [\,-r_1,\ r_1\,]$ while $c_2 = 2r_1 > r_1$, so $S_1 = \{1\}$.
For $2 \leq i \leq 2m$ Definition~\ref{def:chain} gives $r_{2i-1} = g_{i-1}$ and $c_{2i-1} = c_{2i-3}+g_{i-1}$, while $c_{2i}-c_{2i-1} = (g_{i-1}+g_i)/2 > g_{i-1}$ by~\eqref{eq:diffs}.
So no centre lies in $(\,c_{2i-1},\ c_{2i-1}+g_{i-1}\,]$, and Lemma~\ref{lem:oddstar} gives $S_{2i-1} = \{2i-3,\, 2i-2,\, 2i-1\}$.

The stars therefore give the edges $\{2i-3,2i-1\}$ and $\{2i-2,2i-1\}$ for $2 \leq i \leq 2m$, which are exactly the edges of the chain. The vertex $4m$ lies in no star but its own.
\end{proof}

The second tool adds new vertices to the left of every centre of a given representation, each with an interval reaching exactly to the last centre of a prescribed prefix.
Each new vertex is then adjacent to that prefix, and the new vertices form a clique.
This is how the outer vertices $q_2,\dots,q_{2m-1}$ are attached in every case below.
The inner vertex $q_1$ needs to be handled separately, since $q_1 \not\sim q_{2m-1}$.

\begin{lemma}
\label{lem:prefixclique}
Suppose $G$ has a $50\%$-tolerance representation $(c,r)$ with pairwise distinct centres, and set $\ell = \min_v \{ c_v-r_v\}$.
Let $p_1,\dots,p_t$ be new vertices, and let $P_1,\dots,P_t$ be non-empty prefixes of the centre order of $G$.
Let $\pi_j = \max_{w \in P_j} \{ c_w \}$ be the largest centre in $P_j$.
Set
\[  c_{p_j} = \ell-1-\mfrac{j}{2t}, \quad r_{p_j} = \pi_j - c_{p_j}. \]
Then $(c,r)$ extended in this way is a $50\%$-tolerance representation of $G$ together with $p_1,\dots,p_t$, and its centres are pairwise distinct.
Each interval $I_{p_j}$ has left endpoint smaller than every centre and right endpoint $\pi_j$, and $N[p_j] = S_{p_j} = P_j \cup \{p_1,\dots,p_t\}$.
\end{lemma}

\begin{proof}
Since centres and radii of vertices in $G$ are preserved, it suffices to consider the vertices $p_j$.
The centres $c_{p_j}$ are distinct and lie to the left of $\ell$, hence to the left of every centre of $G$.
Since $P_j$ is non-empty, $\pi_j \geq \min_v \{c_v\} \geq \ell$, so $r_{p_j} \geq \ell - c_{p_j} = 1 + j/(2t) > 1$.
The right endpoint of $I_{p_j}$ is $c_{p_j}+r_{p_j} = \pi_j$, and its left endpoint lies below every centre, since $c_{p_j} < \ell$ and $r_{p_j} > 1 > |c_{p_j} - c_{p_{j'}}| = |j-j'|/(2t)$.
So $I_{p_j}$ contains every $c_{p_{j'}}$ and, as $P_j$ is a prefix of the centre order, exactly the centres of $P_j$ among those of $G$, giving $S_{p_j} = P_j \cup \{p_1,\dots,p_t\}$.

Finally, $c_w-r_w \geq \ell > c_{p_j}$ for every $w \in V(G)$, so $p_j$ does not lie in the star of any vertex of $G$.
Hence, $S_{p_j} = N[p_j]$.
\end{proof}

The third tool adjoins $x$ and $z$.
Definition~\ref{def:xz} places $c_x$ in a gap of the centre order with $I_x$ reaching every centre but $c_1$.
It starts $I_z$ in that same gap, to the right of $c_x$, so that $I_z$ contains exactly the surviving centres of $2m+1,\dots,4m$, and places $c_z$ to the right of every other interval.
In each of Cases~(C1)--(C3), whichever of $x$ and $z$ survives is placed this way.

\begin{definition}[Adjoining $x$ and $z$]
\label{def:xz}
Let $\Gamma = B^m - \{x, z\}$ or $\Gamma = B^m - \{x, z, u\}$ for a chain vertex $u$, and let $\nu \in \{4,\dots,4m\}$.
Let $(c,r)$ be a $50\%$-tolerance representation of $\Gamma$ with pairwise distinct centres, in which the vertices with label at most $\nu$ form a non-empty proper prefix of the centre order, the centres are non-negative, and $c_1 = 0$ when $1$ is present.
Set
\[  L = \max\{c_v : v \leq \nu\}, \quad H = \min\{c_v : v > \nu\}, \quad\text{ and }\quad \pi = \max\{c_v : v \in V(\Gamma)\}. \]
Let $L^+ = \max\{L,\ \mfrac12(c_1+\pi)\}$ if $1$ is present, and $L^+ = L$ if $1$ is deleted.
Adjoin $x$ by setting
\[
  c_x = \mfrac12(L^++H), \quad
  r_x = \max\{|c_x-c_v| : v \in V(\Gamma),\, v \neq 1\}.
\]
Let $H^- = \mfrac12(c_x+H)$ if $x$ is adjoined, and $H^- = H$ if $x$ is omitted.
Adjoin $z$ by setting
\[
  c_z = 1+\max\{c_v+r_v\}, \quad
  r_z = c_z-H^-,
\]
where the maximum is over every vertex present, including $x$ when it is adjoined.
\end{definition}

To place $x$ we need $L^+ < H$, so $c_x$ is strictly between $L$ and $H$.
When $1$ is deleted this is $L < H$, which holds since the vertices with label at most $\nu$ come first in the centre order.
When $1$ is present it reduces to $(c_1+\pi)/2 < H$, that is, to
\begin{equation}
\label{eq:budget}
  2H-c_1-\pi  >  0 ,
\end{equation}
whose left-hand side we call the {\em budget}.
The condition is natural, since we take $I_x$ to reach the largest centre $\pi$ while missing $c_1$, so $c_x$ is above the midpoint $(c_1+\pi)/2$, and the centre order forces $c_x$ below $H$.
The adjacencies of $x$ and $z$ are then fixed, as verified in the next lemma.

\begin{lemma}
\label{lem:xz}
Let $\Gamma$, $\nu$, and $(c,r)$ be as in Definition~\ref{def:xz}, and adjoin $x$, or $z$, or both.
\begin{enumerate}[label=\normalfont(\roman*)]
  \item If $x$ is adjoined and $L^+ < H$, then $L < c_x < H$, the centres of $\Gamma$ lying in $I_x$ are exactly those other than $c_1$, and $x$ is adjacent to exactly the surviving $v \neq 1$.
  \item If $z$ is adjoined, and $L^+ < H$ when $x$ is adjoined, then $c_z$ is larger than every other centre, and $z$ is adjacent to exactly the surviving $v$ with $c_v \geq H$.
  In particular, $x \not\sim z$.
\end{enumerate}
\end{lemma}

\begin{proof}
(i) Since $L \leq L^+ < H$, the centre $c_x = (L^++H)/2$ lies strictly between $L$ and $H$, and $H \leq \pi$.
By construction $I_x$ contains every centre other than $c_1$, so $x$ is adjacent to every surviving $v \neq 1$.
If $1$ is deleted there is nothing more to check, so suppose that $1$ is present, with $c_1 = 0$.

Since the centres are distinct and non-negative, every other centre satisfies $0 < c_v \leq \pi$. If $c_v \leq c_x$ then $c_x-c_v < c_x-c_1$, and if $c_v > c_x$ then $c_v-c_x \leq \pi-c_x < c_x-c_1$, because $2c_x = L^++H > 2L^+ \geq c_1+\pi$.
Hence $r_x < c_x-c_1$, so $c_1 \notin I_x$.

Finally, $\nu \geq 4$ and at most one chain vertex is deleted, so one of $2$ and $4$ survives with label at most $\nu$.
Denote this surviving vertex by $v$. 
Now $v$ is not adjacent to $1$ in $\Gamma$, so $c_v \notin I_1$, giving $r_1 < c_v \leq L < c_x$.
Hence $c_x \notin I_1$, and with $c_1 \notin I_x$ this gives $x \not\sim 1$.

(ii) Since $c_x < H \leq \pi$ when $x$ is adjoined, we have $H^- \leq \pi < c_z$, so $r_z > 0$.
By the choice of $c_z$, every other interval $I_v$ lies strictly to the left of $c_z$.
So $c_z$ is larger than every other centre and $c_z \notin I_v$.
Hence $z \sim v$ if and only if $c_v \in I_z = [ H^-,\ 2c_z-H^- ]$. 
Since every centre is smaller than $c_z < 2c_z-H^-$, this reduces to $c_v \geq H^-$.

When $x$ is omitted, $H^- = H$.
When $x$ is adjoined, part~(i) gives $L < c_x < H$, so $c_x < H^- < H$, while no centre of $\Gamma$ lies in $(L,H)$.
Either way $z$ is adjacent to exactly the surviving $v$ with $c_v \geq H$, and in particular $x \not\sim z$.
\end{proof}

For the rest of this subsection, set
\[
  \eta = \mfrac{1}{16m^2}, \quad \Lambda(i) = 1+i\eta, \quad
  \Sigma(i) = \mfrac{\Lambda(i)}{2m}, \quad \rho = \mfrac{\eta}{16m} .
\]
Note that $\Lambda$ and $\Sigma$ increase in $i$, and that, since $m \geq 2$,
\begin{equation}
\label{eq:LamSig}
\begin{gathered}
  \Sigma(i) < \mfrac{1}{m} \leq \mfrac12 < 1 \leq \Lambda(i) < 2 \quad\text{for } 1 \leq i \leq 2m,\\
  m\Sigma(2m) = \mfrac12+\mfrac{1}{16m} \leq \mfrac{17}{32}, \quad\text{ and }\quad
  m^2\eta = \mfrac{1}{16} .
\end{gathered}
\end{equation}

Every gap below is a {\em long} gap $\Lambda(i)$ or a {\em short} gap $\Sigma(i)$, apart from the single value $g^*$ that Definition~\ref{def:chain-u} uses for odd $u$.
We make the budget~\eqref{eq:budget} positive in three ways: by dropping the gaps from $\Lambda$ to $\Sigma$, by adjusting $\nu$, or by lowering the top of the chain.
Moreover, all gaps are at least $\Sigma(1)$ and any increase from one gap to the next is at least $\eta/(2m)$, both of which are larger than $2\rho = \eta/(8m)$.
So, at the index~$1$ and at every index where the gaps increase, each difference in~\eqref{eq:diffs} is positive and larger than $\rho$.
Consequently, Lemma~\ref{lem:chain} applies whenever the gaps increase at every index.

\subsubsection{\texorpdfstring{Deleting $x$ or $z$: the bases for (C1)}{Deleting x or z: the bases for (C1)}}
\label{sss:C1}

Here the chain is untouched.
Applying Lemma~\ref{lem:chain} with $g_i = \Lambda(i)$ for $1 \leq i \leq 2m$ gives a $50\%$-tolerance representation $(c, r)$ of the chain with centre order $1 < \cdots < 4m$.
The odd and even centres are, respectively,
\[
\begin{gathered}
  c_{2i-1} = (i-1)+\eta i(i-1)/2 \quad\text{for } 1 \leq i \leq 2m
  \quad\text{and}\\
  c_{2i} = i+\eta\bigl(i(i+1)-1\bigr)/2 \quad\text{for } 2 \leq i \leq 2m ,
\end{gathered}
\]
while $c_2 = (1+\eta)/2$.

The next lemma uses Definition~\ref{def:xz} to adjoin $z$ or $x$ to $(c,r)$.
Note that with $\nu = 2m$ we have $L = c_{2m}$ and $H = c_{2m+1}$, so the budget is $2c_{2m+1}-c_1-\pi = \eta(1/2-m^2) < 0$ by the formulas above.
Thus $L^+ \geq H$ and Lemma~\ref{lem:xz}(i) does not apply.
If $x$ is omitted, however, there is no budget constraint and only $z$ has to be placed, for which $\nu = 2m$ ensures $N(z) = \{2m+1,\dots,4m\}$.
If instead $z$ is omitted, we are free to choose a different value of $\nu$, and $\nu = 2m+1$ makes the budget positive.

\begin{lemma}
\label{lem:omitxz}
Let $(c,r)$ be the representation of the chain above.
\begin{enumerate}[label=\normalfont(\roman*)]
  \item Omitting $x$ and adjoining $z$ by Definition~\ref{def:xz} with $\nu = 2m$ gives a $50\%$-tolerance representation of $B^m-x$, in the centre order $1 < \cdots < 4m < z$.
  \item Omitting $z$ and adjoining $x$ by Definition~\ref{def:xz} with $\nu = 2m+1$ gives a $50\%$-tolerance representation of $B^m-z$, in the centre order $1 < \cdots < 2m+1 < x < 2m+2 < \cdots < 4m$.
\end{enumerate}
\end{lemma}

\begin{proof}
The chain representation satisfies the hypotheses of Definition~\ref{def:xz} for either value of $\nu$, since its centre order is $1 < \cdots < 4m$ and $c_1 = 0$.
Since $(c,r)$ represents the chain, only the adjacencies and the position of the adjoined vertex need checking.

(i) The hypothesis $L^+ < H$ of Lemma~\ref{lem:xz}(ii) is vacuous because $x$ is not adjoined, so $c_z$ is larger than every other centre and $z$ is adjacent to exactly the vertices with $c_v \geq H = c_{2m+1}$, namely $2m+1,\dots,4m$.

(ii) By Lemma~\ref{lem:xz}(i) it suffices to verify $L^+ < H$, which then gives that $x$ is adjacent to exactly $2,\dots,4m$ and that $c_{2m+1} = L < c_x < H = c_{2m+2}$.
Here $L = c_{2m+1} < c_{2m+2} = H$, and since $c_1 = 0$ and $\pi = c_{4m}$, we have
\[
\begin{aligned}
  2c_{2m+2}-c_1-\pi
  &= 2\bigl((m+1)-m\bigr) + \mfrac{\eta}{2}\Bigl( 2\bigl((m+1)(m+2)-1\bigr) - \bigl(2m(2m+1)-1\bigr) \Bigr)\\
  &= 2-\mfrac{\eta}{2}(2m^2-4m-3) > 0,
\end{aligned}
\]
where the last inequality holds since $m^2\eta = 1/16$.
\end{proof}

\subsubsection{Deleting a chain vertex: the bases for (C2)}
\label{sss:C2}

For each chain vertex $u$ we build the chain representation for $B^m - \{x, z, u\}$ that Case~(C2) rests on.
We then adjoin $x$ and $z$ by Definition~\ref{def:xz}, which requires the budget to be positive.
To achieve this, we drop the gaps from long to short at the deleted vertex.
The exceptions are $u = 1$, $u = 4m-1$, and $u = 4m$, where every gap stays long.
At $u = 1$ and $u = 4m$ the deletion itself suffices, while at $u = 4m-1$ the vertex $4m$ is placed just above $4m-2$, which lowers the top of the chain.

\begin{definition}[The chain representation for $B^m - \{x, z, u\}$]
\label{def:chain-u}
Let $u \in \{1,\dots,4m\}$.
We assign centres and radii to the vertices of $B^m - \{x, z, u\}$ as in Definition~\ref{def:chain}, with the gaps and the adjustments listed below.
Where the gaps drop from long to short, the sequence stops increasing, and further adjustments restore the centre order and the stars.
\begin{itemize}
\item If $u = 1$ or $u = 4m$, set $g_i = \Lambda(i)$ for $1 \leq i \leq 2m$.
\item If $u = 4m-1$, set $g_i = \Lambda(i)$ for $1 \leq i \leq 2m$, and set $c_{4m} = c_{4m-2}+\Sigma(2m)$.
\item If $u = 2$, set $g_1 = \Lambda(1)$ and $g_i = \Sigma(i)$ for $i \geq 2$, and set $r_1 = g_1$, $r_3 = g_2/4$, and $c_4 = c_3+g_2/2$.
\item If $u = 2k$ is even with $4 \leq u \leq 4m-2$, set
  \[
  \begin{gathered}
    g_i = \Lambda(i) \ \text{ for } i \leq k-1, \quad g_k = \Lambda(k-1), \quad g_i = \Sigma(i) \ \text{ for } i \geq k+1 ,\\
    r_{2k+1} = \mfrac14 g_{k+1}, \quad c_{2k+2} = c_{2k+1}+\mfrac12 g_{k+1} .
  \end{gathered}
  \]
\item If $u = 2k-1$ is odd with $3 \leq u \leq 4m-3$, set
  \[  g_i = \Lambda(i) \ \text{ for } 1 \leq i \leq k-2, \quad g_{k-1} = g^* = g_{k-2}+1, \quad g_i = \Sigma(i) \ \text{ for } i \geq k+1 , \]
  where $g_0 = 0$.
  Note that there is no gap $g_k$.
  The four formulas of Definition~\ref{def:chain} that call for $g_k$ therefore do not apply, and in their place we set
  \[
  \begin{gathered}
    c_{2k+1} = c_{2k-3}+g^*, \quad r_{2k+1} = \mfrac14 g_{k+1},\\
    c_{2k} = c_{2k+1}-\mfrac18 g_{k+1}, \quad c_{2k+2} = c_{2k+1}+\mfrac12 g_{k+1} .
  \end{gathered}
  \]
\end{itemize}
\end{definition}

The next lemma verifies that Definition~\ref{def:chain-u} gives a representation of $B^m - \{x, z, u\}$ and identifies the stars that differ from those of Lemma~\ref{lem:chain}.

\begin{lemma}
\label{lem:chain-u}
Definition~\ref{def:chain-u} gives a $50\%$-tolerance representation of $B^m - \{x, z, u\}$.
Its centres are pairwise distinct, in the order $1 < \cdots < 4m$ with $u$ removed.
If $u$ is odd or $u = 4m$, then the stars are exactly those displayed in Lemma~\ref{lem:chain} with $u$ removed.
For even $u$ with $u \neq 4m$, the same holds except that $S_{u+1} = \{u+1\}$, while $S_{u-1}$ is the star of Lemma~\ref{lem:chain} at $u-1$ together with $u+1$.
\end{lemma}

\begin{proof}
Write $u = 2k$ or $u = 2k-1$ with $1 \leq k \leq 2m$.
Since the edges of a representation are the pairs $vw$ with $v \neq w \in S_v$, the stars determine the edges.
Removing $u$ from every star of Lemma~\ref{lem:chain} removes exactly the edges at $u$.
For even $u \neq 4m$, the two further changes move the edge $\{u-1,u+1\}$ from $S_{u+1}$ to $S_{u-1}$, and so leave the edges unchanged.
So it is enough to establish the centre order and the stars claimed.

We first establish the centre order, which the star computations use.
By the remark following~\eqref{eq:LamSig}, the differences in~\eqref{eq:diffs} at the index $1$, and at every index where the gaps increase, are positive and are larger than $\rho$.
This covers every difference between consecutive surviving centres except the following, which come from Definition~\ref{def:chain-u}.
\begin{itemize}
\item For $u = 4m-1$: $c_{4m}-c_{4m-2} = \Sigma(2m)$.
\item For even $u \neq 4m$: $c_{2k+1}-c_{2k-1} = g_k$, and $c_{2k+2}-c_{2k+1} = g_{k+1}/2$. 
For $k < 2m-1$, we have $c_{2k+3}-c_{2k+2} = c_{2k+1} + g_{k+1} - c_{2k+2} = g_{k+1}/2$.
\item For odd $u$ with $3 \leq u \leq 4m-3$: $c_{2k+1}-c_{2k} = g_{k+1}/8$, $c_{2k+2}-c_{2k+1} = g_{k+1}/2$, $c_{2k+3}-c_{2k+2} = g_{k+1}/2$ when $k < 2m-1$, and
\[
\begin{aligned}
  c_{2k}-c_{2k-2} &= (c_{2k}-c_{2k+1}) + (c_{2k+1}-c_{2k-3}) + (c_{2k-3}-c_{2k-2})\\
  &= -\mfrac18 g_{k+1} + g^* - \mfrac12(g_{k-2}+g^*) = \mfrac12-\mfrac18 g_{k+1}.
\end{aligned}
\]
\end{itemize}
All of these differences are larger than $1/(16m)$ and, hence, larger than $\rho$.
The surviving centres therefore increase with their labels, consecutive ones differing by more than $\rho$.
Since every even radius is $\rho$, no interval $I_{2i}$ contains a centre other than $c_{2i}$.
So $S_{2i} = \{2i\}$ at each surviving even vertex.

For the odd stars, it suffices to consider the indices $k$ and $k+1$.
At every other index $i \geq 2$, the vertices $2i-3$, $2i-2$, $2i-1$ survive and Definition~\ref{def:chain-u} keeps the formulas of Definition~\ref{def:chain} at $i$.
The gaps also increase at $i$, so Lemma~\ref{lem:oddstar} gives $S_{2i-1} = \{2i-3,2i-2,2i-1\}$, as in the proof of Lemma~\ref{lem:chain}.
For $u \notin \{1,2\}$ we have $r_1 = g_1/4$ and $c_2 = g_1/2 = 2r_1$, so $S_1 = \{1\}$ as in Lemma~\ref{lem:chain}.
We consider the cases of Definition~\ref{def:chain-u} in order, with the even deletions split by index.

{\em Case 1: $u \in \{1,\, 4m-1,\, 4m\}$.}
Every gap is $\Lambda(i)$, so the hypotheses of Lemma~\ref{lem:chain} hold.
For $u = 1$ and $u = 4m$ no value is changed, so every surviving star is the one that lemma gives.
For $u = 4m-1$ the odd vertex at the index $k = 2m$ is the deleted $4m-1$ itself, and there is no index $k+1$.
So there is nothing to check.

{\em Case 2: $u$ even with $u \neq 4m$, $i = k$.}
When $u \geq 4$ we have $g_k = \Lambda(k-1) = g_{k-1}$, so the gaps do not increase at $k$.
We apply Lemma~\ref{lem:oddstar} with $g = g_{k-1}$.
The right endpoint of $I_{2k-1}$ is $c_{2k-1}+g_{k-1} = c_{2k-1}+g_k = c_{2k+1}$.
Since $2k$ is deleted, the only centre in $(\,c_{2k-1},\ c_{2k+1}\,]$ is $c_{2k+1}$, so $S_{2k-1} = \{2k-3,2k-2,2k-1,2k+1\}$.

At $u = 2$ we have $r_1 = g_1$, so $I_1 = [\,-g_1,\ g_1\,]$.
Its right endpoint is $c_3 = g_1$, and $2$ is deleted, so $S_1 = \{1,3\}$.

{\em Case 3: $u$ even with $u \neq 4m$, $i = k+1$.}
Here $g_{k+1} = \Sigma(k+1) < 1 < g_k$.
We have $r_{2k+1} = g_{k+1}/4$ and $c_{2k+2} = c_{2k+1}+g_{k+1}/2$, while $c_{2k+1} = c_{2k-1}+g_k$.
As $2k$ is deleted, the centres nearest $c_{2k+1}$ are $c_{2k+2}$, at distance $g_{k+1}/2$, and $c_{2k-1}$, at distance $g_k > 1$.
Both are larger than $r_{2k+1} = g_{k+1}/4$, so $S_{2k+1} = \{2k+1\}$.

{\em Case 4: $u$ odd with $3 \leq u \leq 4m-3$.}
Here $g_{k-1} = g^* = g_{k-2}+1$ and $g_{k+1} = \Sigma(k+1) < 1$, and there is no gap $g_k$.
The deleted vertex $2k-1$ is the odd one at the index $k$, so only $S_{2k+1}$ needs to be considered.
Definition~\ref{def:chain-u} sets $r_{2k+1} = g_{k+1}/4$ and $c_{2k+1} = c_{2k-3}+g^*$, the position Definition~\ref{def:chain} would have given $2k-1$.

The centres nearest $c_{2k+1}$ on either side are $c_{2k}$, at distance $g_{k+1}/8$, and $c_{2k+2}$, at distance $g_{k+1}/2$. The next centre to the left of $c_{2k}$ is $c_{2k-2}$, at distance $1/2$ from $c_{2k+1}$.
Of these only $g_{k+1}/8$ is at most $r_{2k+1} = g_{k+1}/4$, and every other centre lies to the right of $c_{2k+2}$ or left of $c_{2k-2}$.
So $S_{2k+1} = \{2k,2k+1\}$, the star of Lemma~\ref{lem:chain} at $2k+1$ with the deleted $2k-1$ removed.
\end{proof}

With the representation of Definition~\ref{def:chain-u} in hand, we adjoin $x$ and $z$ as in Definition~\ref{def:xz}.
Lemma~\ref{lem:xz} requires the budget to be positive when $1$ is present.
The next lemma establishes this with a uniform bound, which is also utilised in \S\ref{sss:C3}.

\begin{lemma}
\label{lem:budget}
For every $u \in \{2,\dots,4m\}$ the representation of Definition~\ref{def:chain-u}, taken with $\nu = 2m$, has budget $2H-c_1-\pi \geq 11/32$.
\end{lemma}

\begin{proof}
Since $u \geq 2$, the vertex $1$ always survives and $c_1 = 0$ by Definition~\ref{def:chain}.
By Lemma~\ref{lem:chain-u} the centre order is $1 < \cdots < 4m$ with $u$ removed.

Write $u = 2k$ or $u = 2k-1$ with $1 \leq k \leq 2m$.
For odd $u$ with $3 \leq u \leq 4m-3$, we take $g_k$ to be $0$ wherever a sum ranges over it, since it is otherwise undefined.
Then $c_{2i+1} = \sum_{j \leq i} g_j$ for every surviving odd vertex $2i+1$, by telescoping from $c_1 = 0$.
For odd $u$ the step across the deleted vertex is $c_{2k+1}-c_{2k-3} = g^* = g_{k-1}+g_k$.
We first show that
\begin{equation}
\label{eq:pi}
  \pi  \leq  \sum_{i \leq 2m-1} g_i+\Sigma(2m) .
\end{equation}
When $k \leq 2m-1$, neither $4m-1$ nor $4m$ is deleted, so $\pi = c_{4m}$.
Since $c_{4m-1} = \sum_{i \leq 2m-1} g_i$, \eqref{eq:pi} is equivalent to $c_{4m}-c_{4m-1} \leq \Sigma(2m)$.
\begin{itemize}
\item If $k \leq 2m-2$, the index $2m$ is unchanged and the gaps $g_{2m-1} = \Sigma(2m-1)$ and $g_{2m} = \Sigma(2m)$ are both short.
So Definition~\ref{def:chain} gives $c_{4m}-c_{4m-1} = (\Sigma(2m-1)+\Sigma(2m))/2 \leq \Sigma(2m)$, where the inequality holds since $\Sigma$ is increasing.
\item If $k = 2m-1$, then $c_{4m} = c_{4m-1}+g_{2m}/2$ by Definition~\ref{def:chain-u}, so $c_{4m}-c_{4m-1} \leq \Sigma(2m)/2$.
\end{itemize}
If $k = 2m$, then every gap is long.
For $u = 4m$ the vertex $4m$ is deleted, so $\pi = c_{4m-1} = \sum_{i \leq 2m-1} g_i$, and~\eqref{eq:pi} holds.
For $u = 4m-1$, Definition~\ref{def:chain-u} gives $\pi = c_{4m} = c_{4m-2}+\Sigma(2m)$.
Since $g_{2m-2} < g_{2m-1}$, we have $c_{4m-2} = c_{4m-3}+(g_{2m-2}+g_{2m-1})/2 < c_{4m-3}+g_{2m-1} = \sum_{i \leq 2m-1} g_i$, and~\eqref{eq:pi} follows.

With~\eqref{eq:pi} in hand, we turn to the budget itself.
Let $H' = \sum_{i \leq m} g_i$ and $\tau = \sum_{i=m+1}^{2m-1} g_i$, so that $\sum_{i \leq 2m-1} g_i = H'+\tau$.
When $u \neq 2m+1$ we have $H = c_{2m+1} = H'$, and we set $\kappa = 0$.
When $u = 2m+1$ we have $k = m+1$, so $H = c_{2m+2} = c_{2m+3}-g_{m+2}/8$ by Definition~\ref{def:chain-u}.
Also $c_{2m+3} = H'$, since $g_{m+1} = 0$.
Here we set $\kappa = g_{m+2}/8 = \Sigma(m+2)/8$, so that $H = H'-\kappa$ in both cases.
Combined with~\eqref{eq:pi}, this gives
\[
  2H-c_1-\pi  \geq  2(H'-\kappa)-(H'+\tau+\Sigma(2m))  =  H'-(\tau+2\kappa)-\Sigma(2m) .
\]

Suppose first that $k \leq m$.
Then $g_i = \Sigma(i)$ for every $m+1 \leq i \leq 2m-1$, so $\tau$ is a sum of $m-1$ terms, each at most $\Sigma(2m)$.
As $\kappa = 0$, this gives $\tau+2\kappa \leq (m-1)\Sigma(2m)$.
Also, $H' \geq g_1 \geq 1$, since $g_1 = \Lambda(1)$ unless $u = 3$, where $g_1 = g^* = 1$.
Hence, $H'-(\tau+2\kappa) \geq 1-(m-1)\Sigma(2m)$.

Suppose instead that $k \geq m+1$.
Then $g_i = \Lambda(i) \geq 1$ for every $i \leq m$, except that $g_m = g^* \geq 2$ when $u = 2m+1$.
So $H' \geq m$.
We claim that $\tau+2\kappa \leq (m-1)\Lambda(2m)$, where $\tau$ sums the $m-1$ indices $m+1,\dots,2m-1$.
\begin{itemize}
\item If $u$ is even or $u = 4m-1$, then $\kappa = 0$ and every gap summed in $\tau$ is $\Lambda(i)$, $\Lambda(k-1)$, or $\Sigma(i)$, hence at most $\Lambda(2m)$.
\item If $u = 2m+1$, then $k = m+1$ contributes nothing, each of the other $m-2$ indices contributes $\Sigma(i) \leq \Lambda(2m)$, and $2\kappa = \Sigma(m+2)/4 \leq \Lambda(2m)$.
\item If $u$ is odd with $2m+3 \leq u \leq 4m-3$, then $\kappa = 0$, the indices $k-1$ and $k$ together contribute $g^* = g_{k-2}+1 \leq 2\Lambda(2m)$, and each of the other $m-3$ indices contributes at most $\Lambda(2m)$.
\end{itemize}
The claim follows.
Since $\Lambda(2m) = 1+2m\eta$, we have $(m-1)\Lambda(2m) \leq (m-1)+2m^2\eta$, and hence $H'-(\tau+2\kappa) \geq 1-2m^2\eta$.

In both cases
\[
  2H-c_1-\pi  \geq  H'-(\tau+2\kappa)-\Sigma(2m)  \geq  1-2m^2\eta-m\Sigma(2m)
   \geq  1-\mfrac18-\mfrac{17}{32}  =  \mfrac{11}{32} ,
\]
the last inequality by~\eqref{eq:LamSig}.
\end{proof}

The following lemma is the main result of this subsubsection.
It combines Definitions~\ref{def:chain-u} and~\ref{def:xz} into a representation of $B^m-u$ for every chain vertex $u$. \S\ref{sss:C1C2} uses these representations to settle Case~(C2).

\begin{lemma}
\label{lem:baseC3}
Let $u \in \{1,\dots,4m\}$.
Definition~\ref{def:chain-u}, together with $x$ and $z$ adjoined by Definition~\ref{def:xz} with $\nu = 2m$, gives a $50\%$-tolerance representation of $B^m-u$, in the centre order $\ord(m)$ with $u$ deleted.
\end{lemma}

\begin{proof}
By Lemma~\ref{lem:chain-u}, Definition~\ref{def:chain-u} gives a $50\%$-tolerance representation of $B^m - \{x, z, u\}$ whose centre order is $1 < \cdots < 4m$ with $u$ removed.
In particular, the centres are distinct and non-negative, the vertices of label at most $2m$ form a non-empty proper prefix, and $c_1 = 0$ when $1$ survives, as Definition~\ref{def:xz} requires.
The hypothesis $L^+ < H$ of Lemma~\ref{lem:xz} holds at $u = 1$ because $L^+ = L < H$, and at every other $u$ by Lemma~\ref{lem:budget}.
So Lemma~\ref{lem:xz} applies: $x$ is adjacent to exactly the surviving chain vertices other than $1$, and $z$ to exactly the surviving $v \geq 2m+1$, with $x \not\sim z$.
These are exactly the edges of $B^m-u$ at $x$ and $z$.
Moreover, $L < c_x < H$ and $c_z$ is larger than every other centre, so the centre order is $\ord(m)$ with $u$ deleted.
\end{proof}

\subsubsection{\texorpdfstring{Adjoining $q_1$: (C1) and (C2)}{Adjoining q1: (C1) and (C2)}}
\label{sss:C1C2}

Case~(C1) deletes $x$ or $z$ and Case~(C2) deletes a chain vertex, so in each of them $q_1$ survives and has to be adjoined to a representation of $B^m-u$.
Lemma~\ref{lem:prefixclique} cannot attach it, since $q_1 \not\sim q_{2m-1}$, so $c_{q_1}$ has to lie among the centres of $B^m-u$ rather than to their left.
We first name the centres that decide where $q_1$ can go.

Fix $u \in V(B^m)$, and let $C(m) = \{2, 4, \dots, 2m\} \cup \{2m+3, 2m+5, \dots, 4m-1\}$ be the set of cuts of the vertices $q_1, \dots, q_{2m-1}$.
For a cut $s \in C(m)$ write $\pre(s) = \{\omega_1,\dots,\omega_s\} \setminus \{u\}$ for the surviving vertices among the first $s$ of $\ord(m)$.
Each representation of $B^m-u$ built above has centre order $\ord(m)$ with $u$ deleted, except the one with $u = z$.
For $u = z$, Lemma~\ref{lem:omitxz}(ii) exchanges $x$ and $2m+1$, so the first $s$ vertices of its centre order form $\pre(s)$ except when $s = 2m+1$.
As $2m+1 \notin C(m)$, in both orders $\pre(s)$ is a prefix of the centre order for every $s \in C(m)$.

For a representation $(c,r)$ of $B^m-u$ with one of the centre orders noted above, we write $\hi(s) = \max\{c_v : v \in \pre(s)\}$ for the largest centre of a vertex of $\pre(s)$, $\lo = \min\{c_v : v \in V(B^m-u)\}$ for the smallest centre, and  set  $M = \hi(4m-1)$.
Since $\pre(s) \subseteq \pre(s')$ for $s \leq s'$, the function $\hi$ is non-decreasing on $C(m)$.

The next lemma adjoins $q_1$. The hypothesis $2\hi(2)-\lo < M$ is the analogue of the budget for $x$, and is what allows us to adjoin $q_1$ in the required order and with the appropriate adjacencies.

\begin{lemma}
\label{lem:lift}
Let $u \in V(B^m)$, and let $(c,r)$ be a $50\%$-tolerance representation of $B^m-u$ with pairwise distinct centres. Suppose the centre order induced by $(c, r)$ is $\ord(m)$ with $u$ deleted, or, when $u = z$, that order with $x$ and $2m+1$ exchanged.
Let~$b$ be the vertex following $\pre(2)$ in the centre order, so that $b = 3$, or $b = 4$ if $u = 3$.
Suppose that $2 \hi(2)-\lo  <  M$ and set
\[
  c_{q_1} = \mfrac12 \bigl( \hi(2)+\min\{ c_b,\ \mfrac12(\lo+M) \} \bigr) ,
  \quad r_{q_1} = M-c_{q_1} .
\]
Then adjoining $q_1$ gives a $50\%$-tolerance representation of $B^m - u + q_1$ with pairwise distinct centres.
In this representation, $\hi(2) < c_{q_1} < c_b \leq \hi(4)$.
Moreover, $I_{q_1}$ has left endpoint smaller than every centre of $(c,r)$ and right endpoint $M$.
\end{lemma}

\begin{proof}
We start by proving the assertions about $I_{q_1}$.
Note that $\hi(2) < c_b$ by the choice of $b$, and $\hi(2) < (\lo+M)/2$ by hypothesis.
Since $c_{q_1}$ is the midpoint of $\hi(2)$ and $\min\{c_b,\ (\lo+M)/2\}$, we obtain $\hi(2) < c_{q_1} < \min\{c_b,\ (\lo+M)/2\}$.
In particular, $\hi(2) < c_{q_1} < c_b \leq \hi(4)$, where the last inequality holds since $b \in \pre(4)$.

Since $c_{q_1} < (\lo+M)/2$, we have $r_{q_1} = M-c_{q_1} > (M-\lo)/2$, which is positive since $\lo \leq \hi(2) < (\lo+M)/2$.
For the same reason the left endpoint $2c_{q_1}-M$ of $I_{q_1}$ is smaller than $\lo$.
Hence, the left endpoint is smaller than every centre of $(c,r)$, while the right endpoint is $M$ by definition.

We now consider adjacencies.
Since $q_1$ has cut $4m-1$, its neighbourhood in $B^m - u + q_1$ is $\pre(4m-1)$, and the vertices of $B^m-u$ outside $\pre(4m-1)$ are the surviving members of $\{4m-1,\, 4m,\, z\}$.
By the preceding paragraph, the centres of $(c,r)$ lying in $I_{q_1}$ are exactly those at most $M = \hi(4m-1)$, which are the centres of $\pre(4m-1)$.
So $S_{q_1} \cap V(B^m-u) = \pre(4m-1)$. 

Hence, to show that $q_1$ is adjacent to exactly $\pre(4m-1)$, it suffices to show that $c_{q_1} \notin I_w$ for every surviving $w \in \{4m-1,\, 4m,\, z\}$.
Such a $w$ follows $\pre(4m-1)$ in the centre order, so $c_w > M \geq c_b$, as $b \in \pre(4) \subseteq \pre(4m-1)$.
Since $b \in \{3,4\}$, it is not adjacent to $4m-1$, $4m$, or $z$.
As $c_b < c_w$, the centre $c_b$ lies to the left of $I_w$, that is, $c_b < c_w - r_w$.
Hence $c_{q_1} < c_b < c_w - r_w$, so $c_{q_1} \notin I_w$.

It follows that $q_1$ is adjacent to exactly the vertices of $\pre(4m-1)$.
The edges among the vertices of $B^m-u$ are unchanged, so adjoining $q_1$ gives a $50\%$-tolerance representation of $B^m - u + q_1$.
Its centres are pairwise distinct, since $c_{q_1}$ lies in the open interval $(\hi(2),\, c_b)$, which contains no centre of $(c,r)$ by the choice of $b$.
\end{proof}

We now adjoin $q_1$ and the outer vertices to the three representations of $B^m-u$ built above, which settles Cases~(C1) and~(C2) at once.

\begin{corollary}
\label{cor:C1C2}
For every $u \in \{1,\dots,4m\} \cup \{x,z\}$ the graph $Q^m-u$ lies in $\halftolerance$.
\end{corollary}

\begin{proof}
Let $(c,r)$ be the representation of $B^m-u$ given by Lemma~\ref{lem:baseC3} when $u$ is a chain vertex, by Lemma~\ref{lem:omitxz}(i) when $u = x$, and by Lemma~\ref{lem:omitxz}(ii) when $u = z$.
We apply Lemma~\ref{lem:lift} to adjoin $q_1$, giving a $50\%$-tolerance representation of $B^m - u + q_1$ with $\hi(2) < c_{q_1} < \hi(4)$.
To see that the lemma may be applied, note that the centre order of $(c,r)$ is $\ord(m)$ with $u$ deleted, or, when $u = z$, that order with $x$ and $2m+1$ exchanged, so it suffices to verify $2 \hi(2)-\lo < M$.

Recall, by Definitions~\ref{def:chain} and~\ref{def:chain-u}, that $c_1 = 0$ when $1$ survives and $c_2 = g_1/2$ when $2$ survives.
Since $\pre(2) \subseteq \{1, 2\}$ is a prefix of the centre order, $\lo$ is the centre of its first vertex and $\hi(2)$ the centre of its last.
It follows that $2 \hi(2)-\lo$ equals $g_1$, $g_1/2$, or $0$, if both $1$ and $2$ survive, $u = 1$, or $u = 2$, respectively.
In each case $2 \hi(2)-\lo \leq g_1$, so it suffices to show $g_1 < \hi(4m-1) = M$.

By Definitions~\ref{def:chain} and~\ref{def:chain-u}, $g_1 = c_3$ when $u \neq 3$, and $g_1 = c_1 + g^* = c_5$ otherwise.
Since $\{4m-3,4m-2\} \subseteq \{\omega_1,\dots,\omega_{4m-1}\}$ and only $u$ is deleted, at least one of them lies in $\pre(4m-1)$. Hence, $M \geq c_w$, where $w = 4m-2$ if it survives and $w = 4m-3$ otherwise.
If $u \neq 3$, then $w \geq 4m-3 \geq 5 > 3$, and since the chain centres increase with the label, $g_1 = c_3 < c_w \leq M$.
Similarly, if $u = 3$, then $w = 4m-2 \geq 6 > 5$, so $g_1 = c_5 < c_w \leq M$.

Finally, we attach the outer vertices by applying Lemma~\ref{lem:prefixclique} with prefixes as follows.
For $2 \leq i \leq 2m-1$, let $P_i$ be the set of vertices of $B^m - u + q_1$ with centre at most $\hi(s_i)$.
If $s_i = 2$, then $P_i = \pre(2)$, since $\hi(2) < c_{q_1}$.
If $s_i \geq 4$, then $P_i = \pre(s_i) \cup \{q_1\}$, since $c_{q_1} < \hi(4) \leq \hi(s_i)$.
Either way $P_i$ is a non-empty prefix of the centre order with largest centre $\hi(s_i)$.
So $q_2, \dots, q_{2m-1}$ form a clique, and among the vertices of $B^m - u + q_1$ each $q_i$ is adjacent to exactly $P_i$.
Since $s_i \geq 4$ exactly when $i \leq 2m-2$, these are the edges of $Q^m-u$, and we have a $50\%$-tolerance representation of $Q^m-u$.
\end{proof}

\subsubsection{\texorpdfstring{Deleting $q_i$: (C3)}{Deleting qi: (C3)}}
\label{sss:C3}

For $i \geq 2$ the graph $Q^m-q_i$ contains $B^m + q_1$, so by Proposition~\ref{prop:impossible} it has no $50\%$-tolerance representation, and none of the constructions above applies.
Instead, we use a bicoloured representation in which a single vertex $y$ is coloured differently.
Such a representation of a graph $G$ amounts to a $50\%$-tolerance representation $(c,r)$ of $G-y$ together with an interval $I_y$ of positive length which meets $I_v$ precisely when $yv \in E(G)$.
For $i = 1$ a second colour is not necessary, and a $50\%$-tolerance representation of $Q^m - q_1$ can be given.
However, for brevity, we treat the $q_i$ uniformly, which suffices for our purposes.

\begin{corollary}
\label{cor:C3}
For every $i$ with $1 \leq i \leq 2m-1$, the graph $Q^m-q_i$ lies in $\bicoloured$.
\end{corollary}

\begin{proof}
Let $y = \omega_{s_i} + 1$ be the joint just above the cut of $q_i$, so that $q_j \sim y$ if and only if $s_j > s_i$.
The proof of Corollary~\ref{cor:C1C2} gives a $50\%$-tolerance representation $(c, r)$ of $Q^m-y$, built from Definition~\ref{def:chain-u} with $u = y$ by adjoining $x$ and $z$ via Lemma~\ref{lem:baseC3}, $q_1$ via Lemma~\ref{lem:lift}, and the outer vertices via Lemma~\ref{lem:prefixclique}.
Deleting $q_i$ from $(c, r)$ leaves a $50\%$-tolerance representation of $Q^m-q_i-y$, which we still denote by $(c, r)$.
To prove the assertion, it suffices to exhibit an interval $I_y$ of positive length which meets $I_w$ precisely when $w \in N_{Q^m-q_i}(y)$.

Suppose $y = 2k-1$ with $2 \leq k \leq 2m$, and set $I_y = [A,B]$, where
\[
  A = c_{y-2}+r_{y-2}
  \quad\text{and}\quad
  B = \begin{cases}
        c_{y+2}-r_{y+2} & \text{if } y \leq 4m-3,\\
        c_{y-1}              & \text{if } y = 4m-1 .
      \end{cases}
\]
It remains to verify the required properties of $I_y$. We refer to Definition~\ref{def:chain-u} throughout. 

{\em Positive length.}
Suppose first that $y \leq 4m-3$. 
By the choice of $A$, we have $A = c_{2k-3} + r_{2k-3}$.
Here, if $k = 2$, then $r_{2k-3} = r_1 = g_1/4 = g^* - 3/4$, as $g_1 = g^* = 1$, and, otherwise, $r_{2k-3} = g_{k-2} = g^* - 1$.
In either case, $A \leq c_{2k-3} + g^* - 3/4 = c_{2k+1} - 3/4$.
On the other hand, $B = c_{2k+1} - r_{2k+1} = c_{2k+1} - g_{k+1}/4 > c_{2k+1} - 1/8$, where the last inequality holds since $g_{k+1} = \Sigma(k+1) < 1/2$. 
As $c_{y-1} = c_{2k-3}+(g_{k-2}+g^*)/2 = c_{2k+1}-1/2$, we obtain $A < c_{y-1} < B$.

Suppose instead that $y = 4m-1$, so that $i = 1$. Here, we have $A = c_{4m-3} + r_{4m-3} = c_{4m-3} + g_{2m-2}$ and $B = c_{4m-2} = c_{4m-3} + (g_{2m-2} + g_{2m-1})/2$.
Hence, $B - A = (g_{2m-1} - g_{2m-2})/2 = \eta/2 > 0$, where the second equality holds since every gap is long.

In both cases, we have $c_{y-2} < A < c_{y-1} \leq B$, so $I_y$ has positive length and contains $c_{y-1}$.

{\em Adjacencies with the chain and $x$.}
Recall that the chain neighbours of $y$ are $y-2$, $y-1$ and additionally $y+2$ if $y \leq 4m-3$.
Moreover, $y$ is adjacent to $x$.

By the choice of $A$ and $B$, the interval $I_y$ meets $I_{y-2}$, which has right endpoint $A$ by construction, and, for $y \leq 4m-3$, also $I_{y+2}$, which has left endpoint $B$.
$I_y$ also meets $I_{y-1}$ since $c_{y-1} \in I_y$.
Moreover, since $y-1 \neq 1$ and $I_x$ contains every centre other than $c_1$ by Lemma~\ref{lem:xz}(i), $I_y$ also meets $I_x$.
It remains to show that $I_y$ misses $I_w$ for every other chain vertex $w$.

Suppose first that $w$ is even, so that $S_w = \{w\}$ by Lemma~\ref{lem:chain-u}.
If $w \leq y-3$, then $c_w+\rho < c_{w+1} \leq c_{y-2} < A$, so $I_w$ lies to the left of $I_y$.
Similarly, if $w \geq y+3$ (so $y \leq 4m-3$), then $c_w-\rho > c_{w-1} \geq c_{2k+1} > B$, so $I_w$ lies to the right of $I_y$.
If $w = y+1$ and $y \leq 4m-3$, then $c_w = c_{2k+1}-g_{k+1}/8 = B+g_{k+1}/8$.
Since $g_{k+1} = \Sigma(k+1) > 1/(2m) > \eta/(2m) = 8\rho$, we obtain $c_w-\rho > B$, so $I_w$ lies to the right of $I_y$.
If $w = y+1$ and $y = 4m-1$, then $I_w$ begins at $c_{4m}-\rho = B+\Sigma(2m)-\rho > B$, as $B = c_{4m-2}$ and $c_{4m} = c_{4m-2}+\Sigma(2m)$, so again $I_w$ lies to the right of $I_y$.

Suppose next that $w$ is odd, so that $w+1 \notin S_w$ by Lemma~\ref{lem:chain-u}.
If $w \leq y-4$ (so $y \geq 5$), then $I_w$ ends below $c_{w+1} \leq c_{y-3} < A$, so $I_w$ lies to the left of $I_y$.
If $w \geq y+4$ (so $y \leq 4m-3$), then $I_w$ begins at $c_{w-2} \geq c_{2k+1} > B$ by Definition~\ref{def:chain}, so $I_w$ lies to the right of $I_y$.

{\em Adjacency with $z$.}
Recall that $y \sim z$ if and only if $y \geq 2m+1$.
By Definition~\ref{def:xz} and Lemma~\ref{lem:xz}, we have that $I_z = [\,H^-,\ 2c_z-H^-\,]$, that $L < c_x < H^- < H$, and that $c_z$ is larger than every other centre.

Since $B < c_{4m} < c_z$, the interval $I_y$ meets $I_z$ if and only if $B \geq H^-$.
If $y \geq 2m+3$, then $B > c_{y-2} \geq c_{2m+1} = H > H^-$, so $I_y$ meets $I_z$.
If $y \leq 2m-1$, then $B = c_{2k+1} - g_{k+1}/4 < c_{2k+1} - g_{k+1}/8 = c_{2k} \leq c_{2m} = L < H^-$, so $I_y$ misses $I_z$. 

Suppose now that $y = 2m+1$, so that $k = m+1$.
Here, $L = c_{2m} = c_{2k-3}+(g_{k-2}+g^*)/2$ and $H = c_{2k} = c_{2k+1}-g_{k+1}/8 = c_{2k-3}+g^* - g_{k+1}/8$, so
\begin{equation}
\label{eq:corC3-1}
H-L  =  \mfrac12(g^*-g_{k-2})-\mfrac18 g_{k+1}  =  \mfrac12-\mfrac18\Sigma(m+2)
  \quad\text{and}\quad
  B  =  c_{2k+1}-\mfrac14 g_{k+1}  =  H-\mfrac18\Sigma(m+2).
\end{equation}
Moreover, $H^- = (L^+ + 3H)/4$, since $c_x = (L^+ + H)/2$ and $H^- = (c_x+H)/2$.
Hence, $H^- \leq B$ if and only if $L^+ \leq H - \Sigma(m+2)/2$. 
Since $L^+ = \max\{L,\ (c_1+\pi)/2\}$, it suffices to bound both terms in the maximum.

By~\eqref{eq:corC3-1}, $L \leq H - \Sigma(m+2)/2$ is equivalent to $\Sigma(m+2) \leq 4/5$. 
By Lemma~\ref{lem:budget}, we have $2H - c_1 - \pi \geq 11/32$, so $c_1 + \pi \leq 2H - \Sigma(m+2)$ is implied by $ \Sigma(m+2) \leq 11/32$. 
Both hold since $\Sigma(m+2) \leq \Sigma(2m) \leq 17/(32m) \leq 17/64$ by~\eqref{eq:LamSig}.
Hence, $I_y$ meets $I_z$, as required.

{\em Adjacencies with the vertices $q_j$.}
Recall that $y \sim q_j$ if and only if $s_j > s_i$.
By Lemmas~\ref{lem:prefixclique} and~\ref{lem:lift}, every surviving $I_{q_j}$ has left endpoint smaller than every centre of $B^m-y$ and right endpoint $\hi(s_j)$, where $\hi(s_1) = M$.
Since the left endpoint of $I_{q_j}$ lies below $c_{y-2} < A$, the interval $I_{q_j}$ meets $I_y$ if and only if $\hi(s_j) \geq A$.

By Lemma~\ref{lem:baseC3}, the centre order of $B^m-y$ is $\ord(m)$ with $y$ deleted.
The only vertex of $\ord(m)$ separating two consecutive chain vertices is $x$, which lies between $2m$ and $2m+1$.
But $y-1$ is even, so the vertex preceding $y-1$ is $y-2$.
Since $c_{y-2} < A < c_{y-1}$ and $\pre(s_j)$ is a prefix of this order, $\hi(s_j) \geq A$ if and only if $y-1 \in \pre(s_j)$.
As $y-1 = \omega_{s_i}$ and $s_i$ is the cut of the deleted $q_i$, this holds if and only if $s_j > s_i$, as required.
\end{proof}

\end{document}